\documentclass[11pt, reqno]{amsart}
\usepackage{txfonts}
\usepackage{amsmath,amssymb,amsthm,mathrsfs,enumerate,bm,xcolor,multirow,pbox}
\usepackage{graphicx,color,framed,tikz,caption,subcaption}
\usepackage{enumitem}
\setlist{leftmargin=9mm}
\usepackage[colorlinks,linkcolor=black,citecolor=black,urlcolor=black]{hyperref}
\allowdisplaybreaks[4]
\numberwithin{equation}{section}
\newcommand{\N}{\mathbb{N}}
\newcommand{\R}{\mathbb{R}}

\newcommand{\pnorm}[2]{\lVert #1\rVert_{#2}}
\newcommand{\bigpnorm}[2]{\big\lVert#1\big\rVert_{#2}}

\newcommand{\abs}[1]{\lvert#1\rvert}
\newcommand{\bigabs}[1]{\big\lvert#1\big\rvert}
\newcommand{\biggabs}[1]{\bigg\lvert#1\bigg\rvert}
\newcommand{\iprod}[2]{\langle#1,#2\rangle}
\newcommand{\bigiprod}[2]{\big\langle#1,#2\big\rangle}
\newcommand{\biggiprod}[2]{\bigg\langle#1,#2\bigg\rangle}

\renewcommand{\epsilon}{\varepsilon}

\renewcommand{\d}[1]{\mathrm{d}#1}

\newcommand{\bigop}{\mathcal{O}_{\mathbf{P}}}

\newcommand{\smallo}{\mathfrak{o}}
\newcommand{\bigo}{\mathcal{O}}

\renewcommand{\tilde}{\widetilde}

\DeclareMathOperator{\E}{\mathbb{E}}
\DeclareMathOperator{\Prob}{\mathbb{P}}

\DeclareMathOperator{\dv}{div}

\DeclareMathOperator{\tr}{tr}
\DeclareMathOperator{\var}{Var}
\DeclareMathOperator{\cov}{Cov}

\DeclareMathOperator{\op}{op}

\DeclareMathOperator{\err}{\texttt{err}}
\DeclareMathOperator{\seq}{\mathsf{seq}}

\DeclareMathOperator{\prox}{\mathsf{prox}}

\DeclareMathOperator{\rem}{Rem}

\DeclareMathOperator*{\argmin}{arg\,min\,}
\DeclareMathOperator{\arcosh}{arcosh}

\theoremstyle{definition}\newtheorem{problem}{Problem}[section]
\theoremstyle{definition}\newtheorem{definition}[problem]{Definition}
\theoremstyle{remark}

\theoremstyle{remark}\newtheorem{remark}{Remark}
\theoremstyle{definition}
\theoremstyle{plain}\newtheorem{theorem}[problem]{Theorem}
\theoremstyle{plain}
\theoremstyle{plain}\newtheorem{lemma}[problem]{Lemma}
\theoremstyle{plain}\newtheorem{proposition}[problem]{Proposition}
\theoremstyle{plain}\newtheorem{corollary}[problem]{Corollary}
\theoremstyle{plain}

\AtBeginDocument{%
	\def\MR#1{}
}

\begin{document}

\title[A leave-one-out approach to AMP]{A leave-one-out approach to approximate message passing}

\author[Z. Bao]{Zhigang Bao}

\address[Z. Bao]{
	Department of Mathematics, the University of Hong Kong, Hong Kong.
}
\email{zgbao@hku.hk}

\author[Q. Han]{Qiyang Han}

\address[Q. Han]{
	Department of Statistics, Rutgers University, Piscataway, NJ 08854, USA.
}
\email{qh85@stat.rutgers.edu}

\author[X. Xu]{Xiaocong Xu}

\address[X. Xu]{
	Data Sciences and Operations Department, University of Southern California, Los Angeles, CA 90007, USA.
}
\email{xuxiaoco@marshall.usc.edu}

\date{\today}


\begin{abstract}
Approximate message passing (AMP) has emerged both as a popular class of iterative algorithms and as a powerful analytic tool in a wide range of statistical estimation problems and statistical physics models.  A well established line of AMP theory proves Gaussian approximations for the empirical distributions of the AMP iterate in the high dimensional limit, under the GOE random matrix model and other rotational invariant ensembles.

This paper provides a non-asymptotic, leave-one-out representation for the AMP iterate that holds under a broad class of Gaussian random matrix models with general variance profiles. In contrast to the typical AMP theory that describes the first-order behavior for the empirical distributions of the AMP iterate via a low dimensional state evolution, our leave-one-out representation yields an intrinsically high dimensional state evolution formula, which provides a second-order, non-asymptotic characterization for the possibly heterogeneous, entrywise behavior of the AMP iterate under the prescribed random matrix models. 

To exemplify some distinct features of our AMP theory in applications, we analyze, in the context of regularized linear estimation, the precise stochastic behavior of the Ridge estimator for independent and non-identically distributed observations whose covariates exhibit general variance profiles. We find that its finite-sample distribution is characterized via a weighted Ridge estimator in a heterogeneous Gaussian sequence model. Notably, in contrast to the i.i.d. sampling scenario, the effective noise and regularization are now full dimensional vectors determined via a high dimensional system of equations.

Our leave-one-out method of proof differs significantly from the widely adopted conditioning approach for rotational invariant ensembles, and relies instead on an inductive method that utilizes almost solely integration-by-parts and concentration techniques.
\end{abstract}

\maketitle

\setcounter{tocdepth}{1}
\tableofcontents

\sloppy

\section{Introduction}

\subsection{Overview}

Approximate message passing (AMP), originated from statistical physics and graphical ideas \cite{thouless1977solution,koller2009probabilistic,montanari2012graphical}, has emerged as a popular class of first-order iterative algorithms that find diverse applications in both statistical estimation problems and probabilistic analyses of statistical physics models. Some non-exhaustive examples include vector/matrix estimation problems \cite{donoho2009message,rangan2011generalized,bayati2012lasso,krzakala2012probabilistic,deshpande2014information,schniter2015compressive,donoho2016high,kabashima2016phase,montanari2016nonnegative,deshpande2017asymptotic,lesieur2017constrained,fletcher2018iterative,barbier2019optimal,sur2019modern,montanari2021estimation,celentano2022fundamental,mondelli2022approximate}, and the analyses of spin glass and perceptron models \cite{bolthausen2014iterative,bolthausen2019morita,ding2019capacity,fan2021replica,bolthausen2022gardner,fan2022tap}. We refer the readers to \cite{feng2022unifying} for a recent survey on the AMP and more references can be found therein.

In its simplest form, with a symmetric  input matrix $A\in \R^{n\times n}$ distributed according to $\mathrm{GOE}(n)$, an initialization $z^{(0)} \in \R^n$ independent of $A$, and a sequence of possibly non-linear functions $\{\mathsf{F}_t: \R \to \R\}$,  the AMP algorithm iteratively generates $z^{(1)},z^{(2)},\ldots \in \R^n$, according to
\begin{align}\label{eqn:intro_AMP}
z^{(t+1)}\equiv A \mathsf{F}_t(z^{(t)})-b_t \mathsf{F}_{t-1}(z^{(t-1)}),\quad t=0,1,2,\ldots.
\end{align}
Here for a vector $z \in \R^n$, $\mathsf{F}_t(z)\in \R^n$ is understood as applied component-wise. The most crucial part in the AMP iterate (\ref{eqn:intro_AMP}) rests in the choice of the so-called Onsager correction scalar factor $b_t \approx n^{-1}\sum_{\ell \in [n]}\E \mathsf{F}_t'(z^{(t)}_\ell) \in \R$. For this choice of $b_t$, a common heuristic in the literature (cf. \cite{bayati2011dynamics,bayati2015universality}) postulates a similar distributional behavior of the AMP iterate $\{z^{(t)}\}$ to the iterate $\{\tilde{z}^{(t)}\}$ defined by
\begin{align}\label{eqn:intro_AMP_ind_A}
\tilde{z}^{(t+1)}= A^{(t)} \mathsf{F}_t(\tilde{z}^{(t)}),\quad t=0,1,2,\ldots,
\end{align}
where $\{A^{(t)}\}$ are i.i.d. copies of $A$. A well established line of AMP theory \cite{bayati2011dynamics,javanmard2013state,bayati2015universality,berthier2020state} makes this heuristic precise, by proving that in the high dimensional limit $n \to \infty$, the empirical distribution of the AMP iterate $\{z^{(t)}\}$ matches that of $\{\tilde{z}^{(t)}\}$, which is easily seen to be Gaussian with variance recursively characterized by the so-called \emph{state evolution}. More precisely, \cite{bayati2011dynamics} proves that for any fixed $t=0,1,2,\ldots$, and any good enough test function $\psi: \R \to \R$, it holds almost surely that
\begin{align}\label{eqn:intro_AMP_limit_dist}
\lim_{n\to \infty}\frac{1}{n}\sum_{k \in [n]} \psi \big(z_k^{(t+1)}\big)= \E \psi\big(\mathcal{N}(0,\sigma_{t+1}^2)\big),
\end{align}
where for $t\geq 1$, $\{\sigma_t\}\subset \R_{\geq 0}$ is defined via the one-dimensional state evolution
\begin{align}\label{eqn:intro_se}
\sigma_{t+1}^2 \equiv \E \mathsf{F}_t^2\big(\mathcal{N}(0,\sigma_t^2)\big),\quad \sigma_1^2\equiv \lim_{n\to \infty} \frac{\pnorm{\mathsf{F}_0(z^{(0)})}{}^2}{n}.
\end{align}

The method of proof for (\ref{eqn:intro_AMP_limit_dist})-(\ref{eqn:intro_se}) in \cite{bayati2011dynamics} is based on a conditioning technique initiated in \cite{bolthausen2014iterative}, which, in an extreme synthesis, iteratively computes the conditional distribution of $A=\mathrm{GOE}(n)$ conditioned on random linear constraints. This approach allows for generalizations to more complicated AMP iterates under rotationally invariant Gaussian matrices/tensors, with initialization $z^{(0)}$ that can be possibly correlated with $A$, cf. \cite{rangan2011generalized,javanmard2013state,berthier2020state,montanari2021estimation,elalaoui2021optimization,gerbelot2023graph}. This intimate connection to the rotational invariance property of $\mathrm{GOE}(n)$ facilitates adaptation of this technique to other random matrix ensembles with such invariance properties, cf. \cite{ma2017orthogonal,rangan2019vector,takeuchi2021bayes,fan2022approximate,zhong2024approximate}. Moreover, the conditioning technique is flexible enough to provide non-asymptotic versions of (\ref{eqn:intro_AMP_limit_dist}) and (\ref{eqn:intro_se}), cf.  \cite{rush2018finite,li2022non,cademartori2023non}.

On the other hand, the requisite rotational invariance property poses non-trivial technical challenges in generalizing (\ref{eqn:intro_AMP_limit_dist})-(\ref{eqn:intro_se}) to other random matrix ensembles without such properties. A simple test example in this regard is to provide a direct proof to (\ref{eqn:intro_AMP_limit_dist})-(\ref{eqn:intro_se}) for the classical mean 0 variance $1/n$ Wigner matrices. Even for this basic universality problem, existing direct proofs resort to the completely different method of moments \cite{bayati2015universality}, which inevitably involves combinatorial calculations for the highly complicated matrix polynomials appearing in the unfolded AMP iterates. See also \cite{chen2021universality} for another Gaussian interpolation approach with exact combinatorial calculations. Although conceptually simple and amenable to extensions to other random matrix models (cf. \cite{dudeja2023universality,wang2024universality}), the moment method is confined to provide weak convergence in distribution in a suitable limiting sense, and therefore falls short of capturing the non-asymptotic, algorithmic behavior of the AMP that are particularly important in its applications to statistical and signal processing problems.

The goal of this paper is to introduce a new, leave-one-out method to understand the finer-scale, non-asymptotic behavior of the AMP iterate. In particular, our method provides a direct representation of the AMP iterate $\{z^{(t)}\}$ in (\ref{eqn:intro_AMP}) in terms of a suitable version of (\ref{eqn:intro_AMP_ind_A}) via an associated leave-one-out AMP. As will be clear below, our representation holds for a large class of Gaussian random matrix models with  independent entries and general variance profiles, down to the level of each individual coordinates. Crucially, our leave-one-out method requires neither the rotational invariance property as in the conditioning technique, nor combinatorially complex calculations as in the method of moments. In view of the ubiquity of the leave-one-out idea in the broader context of random matrix theory, and its well-known technical robustness and flexibility compared to, e.g., the moment method, in various random matrix models \cite{bai2010spectral,benaych2016lectures,bao2017local}, we anticipate further developments of our method in analyzing the behavior of more complicated AMP algorithms beyond the random matrix model worked out in this paper.

\subsection{A leave-one-out representation of the AMP iterate}
For a Gaussian symmetric matrix $A$ with independent upper triangular entries, we show that each coordinate of the AMP iterate (\ref{eqn:intro_AMP}), where the Onsager correction term $b_t \mathsf{F}_{t-1}(z^{(t-1)})$ is now replaced by $b_t \circ \mathsf{F}_{t-1}(z^{(t-1)})$ with a vector $\big(b_{t,k}=\sum_{\ell \in [n]} \var(A_{k\ell})\E \big[\mathsf{F}_{t,\ell}'(z^{(t)}_\ell)|z^{(0)}\big]\big)_{k \in [n]}$, delivers an efficient representation via an associated leave-one-out AMP iterate. Specifically, for any $k \in [n]$, let $\{z^{(t)}_{[-k]}\}$ be the leave-one-out AMP iterate, starting from the same initialization and then subsequently updated by replacing $A$ in (\ref{eqn:intro_AMP}) via its leave-one-out version $A_{[-k]}$ that sets all elements in the $k$-th row and column to be $0$. We prove in Theorem \ref{thm:AMP_sym_loo} that for $t=0,1,2,\ldots$,
\begin{align}\label{eqn:intro_AMP_loo_representation}
z^{(t+1)}_k=\bigiprod{A_k}{\mathsf{F}_t(z^{(t)}_{[-k]})}+\bigop\Big({(\log n)^{c(1\vee t)^3}}\cdot n^{-1/2}\Big). 
\end{align}
While we have used $\bigop$ above for simplicity of presentation, the leave-one-out representation in (\ref{eqn:intro_AMP_loo_representation}) is intrinsically non-asymptotic, and holds for the prescribed random, symmetric Gaussian matrices $A$ with very general variance profiles, along with general separable nonlinear functions $\mathsf{F}_t: \R^n\to \R^n$, $\mathsf{F}_t(z)=\big(\mathsf{F}_{t,\ell}(z_\ell)\big)_{\ell \in [n]}$, whose components $\mathsf{F}_{t,\ell}:\R\to \R$ may differ significantly from each other. 

An immediate consequence of (\ref{eqn:intro_AMP_loo_representation}) is an efficient method to compute the non-asymptotic distribution of an individual coordinate $z_k^{(t+1)}\in \R$ at iterate $t+1$ via $z^{(t)}\in \R^n$ at iterate $t$. To proceed, as $A_k$ is independent of $\mathsf{F}_t (z_{[-k]}^{(t)})$, by first conditioning on $A_{[-k]}$ and then using concentration of the average along with the approximation $z^{(t)}\approx z^{(t)}_{[-k]}$ (in an $\ell_2$ sense), we are then naturally led to  $z_k^{(t+1)}\stackrel{d}{\approx} \mathcal{N}\big(0, \sum_{\ell \in [n]}\var(A_{k\ell}) \E\mathsf{F}_{t,\ell}^2(z_{\ell}^{(t)})\big)$. Formalizing these heuristics, we prove in Theorem \ref{thm:AMP_sym} that for $t=0,1,2,\ldots$,
\begin{align}\label{eqn:intro_AMP_entrywise_dist}
z_k^{(t+1)}\stackrel{d}{\approx}\mathcal{N}(0,\sigma_{t+1,k}^2), 
\end{align}
where for $t\geq 1$, $\{\sigma_{t}\}\subset \R_{\geq 0}^n$ are now vectors determined via the following, possibly high dimensional state evolution
\begin{align}\label{eqn:intro_highd_se}
\sigma_{t+1,k}^2\equiv \sum_{\ell \in [n]}\var(A_{k\ell})\cdot  \E\big[\mathsf{F}_{t,\ell}^2\big(\mathcal{N}(0,\sigma_{t,\ell}^2)\big)|z^{(0)}\big],\quad k \in [n],
\end{align}
with the initial condition $\sigma_{1,k}^2\equiv \sum_{\ell \in [n]}\var(A_{k\ell})\cdot  \mathsf{F}_{0,\ell}^2(z^{(0)}_\ell)$ for $k \in [n]$. 

The behavior of the AMP iterate---primarily in the formalism (\ref{eqn:intro_AMP_limit_dist})---under our random matrix models with general variance profiles, has been previously examined in \cite{bayati2015universality} via the method of moments, and a further adaptation of 
this moment method allowing for sparse matrices was performed in \cite{hachem2023approximate}. As mentioned above, the method of moments is designed for proving asymptotic weak convergence, and therefore does not capture the actual, non-asymptotic behavior of the AMP iterate, which is expected to exhibit different behavior across its coordinates as well as with respect to the problem dimension in our setting. As our theory (\ref{eqn:intro_AMP_entrywise_dist}) provides entrywise controls of $z^{(t+1)}$ on the order $1$ as opposed to a global $\ell_2$ norm control on the order $\bigo(\sqrt{n})$ via the standard formalism (\ref{eqn:intro_AMP_limit_dist}),  (\ref{eqn:intro_AMP_entrywise_dist})-(\ref{eqn:intro_highd_se}) can be naturally viewed as a non-asymptotic, second-order formula to (\ref{eqn:intro_AMP_limit_dist})-(\ref{eqn:intro_se}) that provides entrywise distributional characterizations for the AMP iterates in these heterogeneous situations.

\subsection{Proof techniques}

We now provide a technical glimpse into the core of our new leave-one-out method. Upon using the leave-one-out AMP, we show in Proposition \ref{prop:AMP_loo_decom} that $z^{(t+1)}_k-\bigiprod{A_k}{\mathsf{F}_t(z^{(t)}_{[-k]})}$ admits an exact error decomposition. This decomposition then essentially boils the problem down to controlling the normalized trace of a sequence of random matrices $M_{1}^{(t)},\ldots, M_{t}^{(t)} \in \R^{n\times n}$, defined via the following matrix recursion: initialized with $ M_{-1}^{(t)}=0_{n\times n}, M_0^{(t)}=\mathsf{D}^{(t)}$, for $s \in [t]$,
\begin{align}\label{eqn:intro_M_recursion}
M_{s}^{(t)}\equiv \big(M_{s-1}^{(t)} A - M_{s-2}^{(t)} \mathsf{B}_{t-s+1}\big) \mathsf{D}^{(t-s)},
\end{align}
where $\mathsf{D}^{(s)}$ is a random diagonal matrix that collects all elements $\{\mathsf{F}_{s,\ell}'(z^{(s)}_\ell )\}_{\ell \in [n]}$ (or some of its variants), and $\mathsf{B}_s$ is a diagonal matrix that collects those of $\{b_{s,\ell}\}_{\ell \in [n]}$. 

The underlying mechanism for the normalized trace $n^{-1}\tr(M_s^{(t)})$ to be small for $1\leq s\leq t$ is non-trivial. In the simplest case where $A=\mathrm{GOE}(n)$, and $\mathsf{F}_t$'s are identity maps, the matrix couple $(\mathsf{D}^{(s)},\mathsf{B}_s)$ becomes $\mathsf{D}^{(\cdot)}=\mathsf{B}_\cdot=I_n$, and therefore the cancellations in the trace calculations for $M_s^{(t)}$ can be directly attributed to the orthogonality of Chebyshev polynomials of the second kind with respect to the (scaled) semicircle density on $[-1,1]$ (see the discussion after the statement of Proposition \ref{prop:trace_M} for details). 

In the general case where the variance profile of $A$ is arbitrary and $\mathsf{F}_t$'s are general non-linear functions, there appears no obvious analogue to the above global reduction to special orthogonal polynomials. We instead prove the trace control $n^{-1}\E\tr(M_s^{(t)})=\bigo\big((\log n)^{ct^3}\cdot n^{-1/2}\big)$ via an inductive argument (cf. Proposition \ref{prop:trace_M}). At a high level, our inductive argument recursively reduces the degree of the complicated matrix polynomial in $n^{-1}\E\tr(M_s^{(t)})$ via repeated applications of Gaussian integration by parts to the recursion (\ref{eqn:intro_M_recursion}). The feasibility of this strategy relies on the fact that applying Gaussian integration by parts once to (\ref{eqn:intro_M_recursion}) results in a reduced matrix polynomial that shares a similar form with $\{M_\cdot^{(\cdot)}\}$, due to the special coupling of $(\mathsf{D}^{(s)},\mathsf{B}_s)$. Therefore, by repeating the procedure, the large complexity of the matrix polynomial in $n^{-1}\E\tr(M_s^{(t)})$ for a general pair $1\leq s\leq t$ can be recursively reduced all the way down to the base cases $1\leq s\leq t\leq 2$. The latter problem can then be handled via straightforward Gaussian concentration of linear and quadratic forms.

The key technical advantage of our leave-one-out method, as briefly described above, lies in its soft analytic nature that explores the Gaussianity of $A$ only via Gaussian integration by parts, along with some Gaussian concentration. It is now well understood that a large array of Gaussian random matrix results proved via such integration-by-parts techniques can be generalized to non-Gaussian settings including the sparse matrices and even the non-Gaussian matrices with correlated entries, using the method of cumulant expansion, cf. \cite{khorunzhy1996asymptotic,lytova2009central,he2017mesoscopic,lee2018local,erdos2019random}. We expect similar analytic techniques to be applicable in our setting as well. Notably, non-asymptotic universality results can also be established through an alternative comparison method; see, e.g., \cite{han2024entrywise} for such results in the context of a broader class of first-order algorithms. To maintain focus on the key proof techniques underlying our new leave-one-out method, we limit our discussion to the Gaussian setting, with extensions to other matrix ensembles to be explored in future work.

Finally, we note that while our leave-one-out method shares certain conceptual similarities with the `cavity method' in \cite{chen2021convergence} for AMP, the two approaches are technically different. The cavity method constructs a cavity iteration using self-avoiding walks, leaving out one new row (or column) of $A$ at each iteration. This construction facilitates an easy derivation of the cavity iteration's asymptotic Gaussian law. The (averaged) law of the AMP iterates are subsequently obtained by proving their proximity to the cavity iterations in the $\ell_2$ norm, via a second-moment calculation with exact combinatorial analysis. In contrast, our leave-one-out AMP iteration excludes only a single row (or column) across all iterations. While this construction does not provide an immediate distribution theory for the leave-one-out AMP, the approximation of our iteration in (\ref{eqn:intro_AMP_loo_representation}) is valid in the much stronger $\ell_\infty$ norm, leveraging a soft analytic method as outlined above.

\subsection{An application to regularized estimation}

 Apart from the aforementioned technical advantages, our new leave-one-out method also holds promise for a diverse range of direct applications through the leave-one-out AMP representation (\ref{eqn:intro_AMP_loo_representation}) and its consequential, non-asymptotic second-order theory (\ref{eqn:intro_AMP_entrywise_dist})-(\ref{eqn:intro_highd_se}).

Here, as an illustration of some new features resulting from the high dimensionality of the state evolution (\ref{eqn:intro_highd_se}), we focus on one specific example of Ridge regression in the context of regularized estimation. Suppose we observe the pairs of data $\{(A_i,Y_i)\}_{i \in [m]}$ from the linear model
\begin{align}\label{def:linear_model}
Y_i = A_i^\top \mu_0 +\xi_i,\quad i \in [m].
\end{align}
Here $A_i \in \R^n$, $\xi_i \in \R$ and $Y_i \in \R$ are the covariate vector, the measurement error, and the response associated with the sample $i \in [m]$, respectively. We are interested in estimation/recovery of the unknown signal $\mu_0\in \R^n$ via the following Ridge estimator \cite{hoerl1970ridge}: for a given tuning level $\lambda>0$,
\begin{align}\label{def:ls_estimator}
\hat{\mu}&\equiv \hat{\mu}(\lambda)\equiv \argmin_{\mu \in \R^n} \bigg\{\frac{1}{2}\sum_{i \in [m]} \big(Y_i-A_i^\top \mu\big)^2+ \frac{\lambda}{2}\pnorm{\mu}{}^2\bigg\}.
\end{align}
The Ridge estimator (\ref{def:ls_estimator}) has a long history in statistical applications, and its precise asymptotic properties have been investigated under the i.i.d. sampling setting from various aspects, cf. \cite{dicker2016ridge,dobriban2018high,wu2020optimal,hastie2022surprises,cheng2022dimension,han2023distribution}.

Using our AMP theory (\ref{eqn:intro_AMP_entrywise_dist})-(\ref{eqn:intro_highd_se}), we may describe the finite-sample distribution of $\hat{\mu}$ for independent, non-identically distributed (i.n.i.d.) observations with general variance profiles on the covariate $A_i$'s. Let $
y^{\seq}(\gamma^\ast)\equiv \mu_0+ \gamma^\ast \circ Z_n$, $ Z_n\sim \mathcal{N}(0,I_n)$,
be the (heteroscedastic) Gaussian sequence model with  `effective noise' $\gamma^\ast \in \R_{\geq 0}^n$. With the observation $y^{\seq}(\gamma^\ast)$, let the sequence Ridge estimator with `effective regularization' $\lambda \tau_{b^\ast} \in \R_{>0}^n$ be defined via
\begin{align}\label{eqn:ridge_seq_est}
\hat{\mu}^{\seq}(\gamma^\ast;\tau_{b^\ast})\equiv \argmin_{\mu \in \R^n}\bigg\{\frac{1}{2}\pnorm{y^{\seq}(\gamma^\ast)-\mu}{}^2+\frac{\lambda}{2}\sum_{j \in [n]} \tau_{b^\ast,j}\mu_j^2\bigg\}.
\end{align}
Here the crucial, vector-valued parameters $\gamma^\ast \in \R^n_{\geq 0}$ and $\tau_{b^\ast}\in \R^n_{>0}$ are determined implicitly via a high dimensional system of $m+n$ equations in Eqn.  (\ref{eqn:fpe_ridge}) ahead. Interestingly, under a certain transformation, we may recast the system of equations for $\tau_{b^\ast}$ into a specific quadratic vector equation that arises from the random matrix theory literature \cite{alt2017local,ajanki2019quadratic}. With the above notation, we prove in Theorem \ref{thm:ridge_dist} the distributional approximation $\hat{\mu}\stackrel{d}{\approx}\hat{\mu}^{\seq}(\gamma^\ast;\tau_{b^\ast})$, in the sense that for any good enough test function $\psi: \R^2 \to \R$, 
\begin{align}\label{eqn:ridge_seq_dist}
\frac{1}{n}\sum_{j \in [n]}\psi\big(\hat{\mu}_j,\mu_{0,j}\big)\approx \frac{1}{n}\sum_{j \in [n]} \E \psi\big(\hat{\mu}_j^{\seq}(\gamma^\ast;\tau_{b^\ast}),\mu_{0,j}\big)
\end{align}
holds with high probability.

The essential new feature of (\ref{eqn:ridge_seq_dist}) lies in the inherent high dimensionality of the effective noise $\gamma^\ast$ and regularization $\lambda \tau_{b^\ast}$ for the equivalent, sequence Ridge estimator $\hat{\mu}^{\seq}(\gamma^\ast;\tau_{b^\ast})$, when the sampling scheme of (\ref{def:linear_model}) is i.n.i.d. with a general variance profile on the covariate $A_i$'s.  In contrast, in the i.i.d. sampling scheme with either isotropic or anisotropic Gaussian designs, it is now well understood that variants of the formula (\ref{eqn:ridge_seq_dist}) can also be proved for a large class of regularized least squares estimators via Gaussian comparison methods, with a pair of \emph{scalar} effective noise and regularization determined via a system of two equations in two unknowns; see \cite{thrampoulidis2018precise,miolane2021distribution, han2023universality} for isotropic designs and \cite{celentano2023lasso,han2023distribution} for anisotropic designs. Unfortunately, these random process methods require i.i.d. sampling schemes in an essential way, and therefore do not generalize, at least in its current form, to general i.n.i.d sampling schemes. As such, while we have only verified the formula (\ref{eqn:ridge_seq_dist}) for the Ridge estimator (\ref{def:ls_estimator}) to avoid excessive technicalities unrelated to our main AMP theory, we anticipate our theory to be useful in proving similar formulae as in (\ref{eqn:ridge_seq_dist}) for a broad class of regularized least squares estimators under general i.n.i.d. sampling schemes that are otherwise beyond the reach of existing methods.

\subsection{Organization}
The rest of the paper is organized as follows. We provide formal statements in Section \ref{section:main_results} for the leave-one-out representation (\ref{eqn:intro_AMP_loo_representation}) and its consequences (\ref{eqn:intro_AMP_entrywise_dist})-(\ref{eqn:intro_highd_se}) (in a more general form) that hold for both the symmetric and asymmetric AMP algorithms. The application into regularized estimation for the Ridge estimator is detailed in Section \ref{section:regularized_ls}. We provide the main steps for the proof of the leave-one-out representation (\ref{eqn:intro_AMP_loo_representation}) in Section \ref{section:proof_AMP_loo}, with the proofs of many technical details deferred to Section \ref{section:proof_technical_results_proof_main_thm}. The proofs of all other results are contained in Sections \ref{section:proof_remaining_main}-\ref{section:proof_ridge} and the Appendix.

\subsection{Notation}
For any two integers $m,n \in \mathbb{Z}$, let $[m:n]\equiv \{m,m+1,\ldots,n\}$ when $m\leq n$ and $\emptyset$ otherwise. Let $[m:n)\equiv [m:n-1]$, $(m:n]\equiv [m+1:n]$, and we write $[n]\equiv [1:n]$. For $a,b \in \R$, $a\vee b\equiv \max\{a,b\}$ and $a\wedge b\equiv\min\{a,b\}$. For $a \in \R$, let $a_\pm \equiv (\pm a)\vee 0$.

For a vector $x \in \R^n$, let $\pnorm{x}{p}$ denote its $p$-norm $(0\leq p\leq \infty)$, and we simply write $\pnorm{x}{}\equiv\pnorm{x}{2}$. For vectors $x,y\in \R^n$, let $\iprod{x}{y}\equiv x^\top y =\sum_{i \in [n]} x_i y_i$. We often write $x^{-1} \equiv (x_i^{-1})_{i \in [n]} \in \R^n$, $x^2\equiv x\circ x = (x_i^2)_{i \in [n]} \in \R^n$ and $\mathfrak{D}_x\equiv \mathrm{diag}(x)\equiv (x_i\bm{1}_{i=j})_{i, j \in[n]} \in \R^{n\times n}$. We use $\{e_k\}$ to denote the canonical basis of the Euclidean space, whose dimension should be self-clear from the context. 

For a matrix $M \in \R^{m\times n}$, let $\pnorm{M}{\op},\pnorm{M}{F}$ denote the spectral and Frobenius norm of $M$, respectively. $I_n$ is reserved for an $n\times n$ identity matrix, written simply as $I$ (in the proofs) if no confusion arises. For two matrices $M,N \in \R^{m\times n}$ of the same size, let $M\circ N \equiv (M_{ij}N_{ij})\in \R^{m\times n}$ be their Hadamard product. For square matrices $M,N \in \R^{n\times n}$, let $\iprod{M}{N}\equiv \tr(M^\top N) =\sum_{i,j \in [n] } M_{ij}N_{ij}$.

We use $C_{x}$ to denote a generic constant that depends only on $x$, whose numeric value may change from line to line unless otherwise specified. $a\lesssim_{x} b$ and $a\gtrsim_x b$ mean $a\leq C_x b$ and $a\geq C_x b$, abbreviated as $a=\bigo_x(b), a=\Omega_x(b)$ respectively;  $a\asymp_x b$ means $a\lesssim_{x} b$ and $a\gtrsim_x b$, abbreviated as $a=\Theta_x(b)$. $\bigo$ and $\smallo$ (resp. $\mathcal{O}_{\mathbf{P}}$ and $\mathfrak{o}_{\mathbf{P}}$) denote the usual big and small O notation (resp. in probability). For a random variable $X$, we use $\Prob_X,\E_X$ (resp. $\Prob^X,\E^X$) to indicate that the probability and expectation are taken with respect to $X$ (resp. conditional on $X$).

For $\Lambda>0$ and $\mathfrak{p}\in \N$, a measurable map $f:\R^n \to \R$ is called \emph{$\Lambda$-pseudo-Lipschitz of order $\mathfrak{p}$} iff 
\begin{align}\label{cond:pseudo_lip}
\abs{f(x)-f(y)}\leq \Lambda\cdot  (1+\pnorm{x}{}+\pnorm{y}{})^{\mathfrak{p}-1}\cdot\pnorm{x-y}{},\quad \forall x,y \in \R^{n}.
\end{align}
Moreover, $f$ is called \emph{$\Lambda$-Lipschitz} iff $f$ is $\Lambda$-pseudo-Lipschitz of order $1$, and in this case we often write $\pnorm{f}{\mathrm{Lip}}\leq L$, where $\pnorm{f}{\mathrm{Lip}}\equiv \sup_{x\neq y} \abs{f(x)-f(y)}/\pnorm{x-y}{}$. For a univariate map $f:\R\to \R$, we write $f^{(q)}$ as the $q$-th derivative of $f$.

For a proper, closed convex function $f$ defined on $\R^n$, its \emph{proximal operator} $\prox_f(\cdot;\tau)$ for any $\tau>0$ is defined by $\prox_f(x;\tau)\equiv \argmin_{z \in \R^n} \big\{\frac{1}{2\tau}\pnorm{x-z}{}^2+f(z)\big\}$. 

\section{Main results}\label{section:main_results}

\subsection{Symmetric AMP}

Consider the general non-linear AMP
\begin{align}\label{def:AMP}
z^{(t+1)}&= A \mathsf{F}_t (z^{(t)})- b_{t}\circ \mathsf{F}_{t-1}(z^{(t-1)}),\quad t=0,1,2,\ldots,
\end{align}
where $\mathsf{F}_t: \R^n\to \R^n$ is `separable' in the sense that $\mathsf{F}_{t}(z)=(\mathsf{F}_{t,\ell}(z_\ell))_{\ell \in [n]}$. Here $z^{(-1)}=0_n, \mathsf{F}_{-1}\equiv 0$ and the initialization $z^{(0)}\in \R^n$ is either fixed or independent of $A$. We consider the following random matrix model for $A$: for some deterministic, symmetric matrix $V \in \R^{n\times n}_{\geq 0}$ with non-negative entries,
\begin{align*}
A\equiv V\circ G_n,\quad \{G_{n,ij}\}_{i\leq j}\stackrel{\textrm{i.i.d.}}{\sim} \mathcal{N}(0,1/n).
\end{align*}
For $t=1,2,\ldots$, the Onsager correction vector $b_t=(b_{t,k})_{k \in [n]}$ is defined via
\begin{align}
b_{t,k}&\equiv \frac{1}{n}\sum_{\ell \in [n]} V_{k\ell}^2 \E \big[\mathsf{F}_{t,\ell}' (z^{(t)}_\ell)|z^{(0)}\big],\quad k \in [n].
\end{align}
The Onsager correction vector $b_t$ here is chosen as deterministic to streamline the proof. The usual data-driven version of the AMP, with Onsager correction vectors computed as the empirical averages, can be handled with additional technicalities; see Theorem \ref{thm:AMP_oracle_data_diff} for a precise statement.

Consider the following leave-one-out AMP: for $k \in [n]$,
\begin{align}\label{def:AMP_loo}
z_{[-k]}^{(t+1)}&= A_{[-k]} \mathsf{F}_t (z_{[-k]}^{(t)})-b_{t}\circ\mathsf{F}_{t-1} (z_{[-k]}^{(t-1)}),\quad t=0,1,2,\ldots,
\end{align}
with the same $z^{(-1)}_{[-k]}=0_n$ and initialization $z^{(0)}_{[-k]}=z^{(0)}$. Here $A_{[-k]}\equiv (A_{ij}\bm{1}_{i,j\neq k})_{i,j \in [n]}$ is obtained by setting all elements in the $k$-th row and column of $A$ to be $0$, and keeping all other elements of $A$ unchanged. 

As the leave-one-out AMP iterate $\{z_{[-k]}^{(t)}\}$ still makes use of the randomness of almost the whole matrix $A$ except for its $k$-th row and column, we expect that $\{z_{[-k]}^{(t)}\}$ is close to the original AMP iterate $\{z^{(t)}\}$ except for their $k$-th coordinates. Actually, the following theorem shows that after one more step of iteration, the information on the $z^{(t+1)}_k$ can be extracted from $z_{[-k]}^{(t)}$ instead of using the original $z^{(t)}$.

\begin{theorem}\label{thm:AMP_sym_loo}
	Suppose the following hold.
	\begin{enumerate}
		\item[(S1)] $\max_{i,j \in [n]}\abs{V_{ij}}\leq K$ holds for some $K\geq 2$.
		\item[(S2)] $\mathsf{F}_{s,\ell} \in C^2(\R)$ for all $s \in [0:t], \ell \in [n]$, and there exists some $\Lambda\geq 2$ such that
		\begin{align*}
		\max_{s \in [0:t]}\max_{\ell \in [n]}\Big\{\abs{\mathsf{F}_{s,\ell}(0)}+\max_{q=1,2}\pnorm{\mathsf{F}_{s,\ell}^{(q)} }{\infty}\Big\}\leq \Lambda.
		\end{align*}
	\end{enumerate} 
   Then for any $D>0$, there exists some universal constant $c_0>0$ and another constant $c_1=c_1(D)>0$, such that with $\Prob(\cdot |z^{(0)})$-probability at least $1-c_1 n^{-D}$, uniformly in $k \in [n]$,
\begin{align*}
\bigabs{z^{(t+1)}_k-\bigiprod{A_k}{\mathsf{F}_t (z_{[-k]}^{(t)})} }&\leq \big(c_1 K\Lambda\log n\cdot (1+\pnorm{z^{(0)}}{\infty})\big)^{c_0 (1\vee t)^3}\cdot n^{-1/2}.
\end{align*}
\end{theorem}

The above theorem is proved in Section \ref{section:proof_AMP_loo}.

\begin{remark}
Theorem \ref{thm:AMP_sym_loo} can be formulated in a slightly more general leave-$k$-out form as follows. Let $\emptyset\neq \mathcal{P}\subset [n]$, and consider the following leave-$\mathcal{P}$-out AMP: starting from $z_{[-\mathcal{P}]}^{(-1)}=0_n$ and the same initialization $z^{(0)}_{[-\mathcal{P}]}=z^{(0)}$, let
		\begin{align*}
		z_{[-\mathcal{P}]}^{(t+1)}&= A_{[-\mathcal{P}]} \mathsf{F}_t (z_{[-\mathcal{P}]}^{(t)})-b_{t}\circ\mathsf{F}_{t-1} (z_{[-\mathcal{P}]}^{(t-1)}),\quad t=0,1,2,\ldots,
		\end{align*}
		where $A_{[-\mathcal{P}]}\equiv (A_{ij}\bm{1}_{i,j\notin \mathcal{P}})_{i,j \in [n]}$. Then under the same assumptions as in Theorem \ref{thm:AMP_sym_loo}, with $\Prob(\cdot|z^{(0)})$-probability at least $1-c_1 n^{-D}$, uniformly in $k \in \mathcal{P}$,
		\begin{align*}
		\bigabs{z^{(t+1)}_k-\bigiprod{A_k}{\mathsf{F}_t (z_{[-\mathcal{P}]}^{(t)})} }&\leq \big(c_1 K\Lambda\log n\cdot (1+\pnorm{z^{(0)}}{\infty})\big)^{c_0 (1\vee t)^3}\cdot n^{-1/2}.
		\end{align*}
		The constants $c_0,c_1$ may depend further on $\abs{\mathcal{P}}$. The above leave-$\mathcal{P}$-out representation may be used to obtain joint normal approximations for $(z_{k}^{(t+1)})_{k \in \mathcal{P}}\in \R^{\abs{\mathcal{P}} }$. To keep notation simple, below we only work with the case $\abs{\mathcal{P}}=1$. 
\end{remark}

As mentioned in the introduction, the leave-one-out representation in Theorem \ref{thm:AMP_sym_loo} provides a principled method to approximate the distribution for each coordinate of the AMP iterate $\{z^{(t)}\}$, in that $z_k^{(t+1)}$ is approximately a centered Gaussian random variable whose variance is completely determined by that of $\{z^{(t)}_{\ell}\}_{\ell \in [n]}$. The following high dimensional state evolution describes the entire covariance structure of $(z_k^{(t+1)},z_k^{(t)},\ldots,z_k^{(1)})$ for any $k \in [n]$.

\begin{definition}[\emph{State evolution for symmetric AMP}]\label{def:AMP_se}
Let $Z^{(1)},Z^{(2)},\ldots \in \R^n$ be a sequence of correlated $n$-dimensional centered Gaussian vectors with independent entries. The correlation structure along the iteration path is determined via the recursive updates: with $Z^{(0)}\equiv z^{(0)}$, for $s_1,s_2 \in \mathbb{Z}_{\geq 0}$,
\begin{align*}
\cov\big(Z^{(s_1+1)}_k,Z^{(s_2+1)}_k\big)\equiv \frac{1}{n}\sum_{\ell \in [n]} V_{k\ell}^2 \E \big[\mathsf{F}_{s_1,\ell}\big(Z^{(s_1)}_\ell\big) \mathsf{F}_{s_2,\ell}\big(Z^{(s_2)}_\ell\big)|z^{(0)}\big],\quad k \in [n]. 
\end{align*}
\end{definition}

The following theorem provides a formal distributional characterization for $(z^{(t+1)}_k,z^{(t)}_k,\ldots,z^{(1)}_k)\stackrel{d}{\approx} (Z^{(t+1)}_k,Z^{(t)}_k,\ldots,Z^{(1)}_k)$ for each $k \in [n]$; its proof can be found in Section \ref{subsection:proof_AMP_sym}.

\begin{theorem}\label{thm:AMP_sym}
	Suppose (S1)-(S2) hold for some $K,\Lambda\geq 2$. Fix any $\Lambda$-pseudo-Lipschitz function $\psi:\R^{t+1} \to \R$ of order $\mathfrak{p}\in \N$. Then there exist some universal constant $c_0>1$ and another constant $c_{\mathfrak{p}}>1$ depending on $\mathfrak{p}$ only, such that 
	\begin{align}\label{ineq:AMP_sym_1}
	&\max_{k \in [n]}\bigabs{\E \big[\psi\big( z^{(t+1)}_k,z^{(t)}_k,\ldots,z^{(1)}_k\big)|z^{(0)}\big]-\E \big[\psi \big(Z^{(t+1)}_k,Z^{(t)}_k,\ldots,Z^{(1)}_k\big)|z^{(0)}\big] }\nonumber\\
	&\qquad  \leq \big(K\Lambda\log n\cdot (1+\pnorm{z^{(0)}}{\infty})\big)^{c_{\mathfrak{p}} (1\vee t)^3}\cdot n^{-1/c_0^{1\vee t}}.
	\end{align}
	Moreover, for any $\Lambda$-pseudo-Lipschitz function $\psi_1:\R \to \R$ of order $\mathfrak{p}\in \N$,
	\begin{align}\label{ineq:AMP_sym_2}
	&\max_{k \in [n]}\bigabs{\E \big[\psi_1\big( z^{(t+1)}_k\big)|z^{(0)}\big]-\E \big[\psi_1 \big(Z^{(t+1)}_k\big)|z^{(0)}\big] }\nonumber\\
	&\leq  \big(K\Lambda\sigma_\ast^{[t+1],-1}\log n\cdot (1+\pnorm{z^{(0)}}{\infty})\big)^{c_{\mathfrak{p}} (1\vee t)^3} \cdot n^{-1/2}.
	\end{align}
	Here $\sigma_\ast^{[t+1]}\equiv  1\wedge \min\limits_{s \in [t+1]}\min\limits_{k \in [n]} \var^{1/2}(Z_{k}^{(s)}|z^{(0)})$.
\end{theorem}

An important feature of the above theorem is that the distributional characterization holds for \emph{individual} coordinates, as opposed to the typical AMP theory that holds for the average of all coordinates, cf. \cite{bayati2011dynamics,javanmard2013state,berthier2020state}. This non-asymptotic, entrywise distributional characterization is particularly relevant with the presence of a general variance profile matrix $V$, as the behavior of the AMP iterate $\{z^{(t)}\}$ may differ significantly across different coordinates and with respect to the problem dimension $n$.

\begin{remark}
The worsened bound in the first inequality (\ref{ineq:AMP_sym_1}) in Theorem \ref{thm:AMP_sym} arises from the inductive argument used to control the Frobenius norm error of the covariance matrices between the $n$-dimensional Gaussian vector and its leave-one-out counterpart in the AMP state evolution, without any apriori lower bounds on their smallest eigenvalues. Indeed, in typical applications where the AMP iterate $\{z^{(t)}\}$ is expected to converge as $t \to \infty$, as $z^{(t+1)}$ becomes increasingly indistinguishable from $z^{(t)}$, the smallest eigenvalue of $\cov\big((Z_k^{(s+1)})_{s\in [0:t]}\big)$ is expected to decay rapidly as $t$ grows. 
\end{remark}

The $C^2$ smoothness condition on $\{\mathsf{F}_t\}$ in (S2) can be substantially relaxed to a Lipschitz condition, when the main interest lies in the empirical distribution of the AMP iterate $\{z^{(t)}\}$. For technical reasons, in the following theorem (proved in Section \ref{subsection:proof_AMP_sym_avg}), we work with a modified AMP iterate $\{\overline{z}^{(t)}\}$ whose initialization remains the same $\overline{z}^{(0)}=z^{(0)}$, and the Onsager correction vectors $\{\overline{b}_{t,k}\equiv n^{-1}\sum_{\ell \in [n]} V_{k\ell}^2 \E [\mathsf{F}_{t,\ell}'(Z_\ell^{(t)})|z^{(0)}]\}_{k \in [n]}$ are now defined via the Gaussian random vectors $\{Z^{(\cdot)}\}$ in the state evolution in Definition \ref{def:AMP_se}, with $\{\mathsf{F}_{t,\cdot}'\}$'s taken as the weak derivatives of $\{\mathsf{F}_{t,\cdot}\}$.

\begin{theorem}\label{thm:AMP_sym_avg}
	Suppose (S1) holds and (S2) is replaced by the following:
	\begin{enumerate}
		\item[(SA2)] There exists some $\Lambda\geq 2$ such that
		\begin{align*}
		\max_{s \in [0:t]}\max_{\ell \in [n]}\Big\{\abs{\mathsf{F}_{s,\ell}(0)}+\pnorm{\mathsf{F}_{s,\ell} }{\mathrm{Lip}}\Big\}\leq \Lambda.
		\end{align*}
	\end{enumerate} 
	Fix a sequence of $\Lambda$-pseudo-Lipschitz functions $\{\psi_k: \R^{t+1}\to \R\}_{k \in [n]}$ of order $2$. Then for any $D>0$, there exist a universal constant $c_0>1$ and another constant $c_1=c_1(D)>0$, such that with $\Prob(\cdot|z^{(0)})$-probability at least $1- c_1 n^{-D}$,
	\begin{align*}
	&\biggabs{ \frac{1}{n}\sum_{k \in [n]} \Big(\psi_k\big( \overline{z}^{(t+1)}_{k},\overline{z}^{(t)}_{k},\ldots,\overline{z}^{(1)}_{k}\big)-  \E \big[\psi_k\big( Z^{(t+1)}_{k},Z^{(t)}_{k},\ldots,Z^{(1)}_{k}\big)|z^{(0)}\big]\Big) }\\
	&\leq \big(c_1K\Lambda\sigma_{*,\psi}^{-1} \log n\cdot (1+\pnorm{z^{(0)}}{\infty})\big)^{c_0 (1\vee t)^3}\cdot n^{-1/c_0^{1\vee t}}.
	\end{align*}
	Here $
	\sigma_{\ast,\psi}^2\equiv 1\wedge \min\limits_{ k \in [n]}\min\limits_{s \in S_{\psi_k} \cap \mathbb{Z}_{\geq 2} }  \var(Z^{(s-1)}_k|z^{(0)})$, where $S_{\psi_k}$ is the minimum-sized subset of $[t+1]$ so that $\psi_k(z)$ depends on $z$ only via $z_{S_{\psi_k}}$. 
\end{theorem}

\begin{remark}
A few remarks are in order:
\begin{enumerate}
	\item The range of $t$ for both Theorems \ref{thm:AMP_sym} and \ref{thm:AMP_sym_avg} to be effective is probably sub-optimal. In the GOE case, the averaged distributional characterization for the AMP iterate $\{z^{(t)}\}$ (as in Theorem \ref{thm:AMP_sym_avg} above) is known to hold up to $t\ll \log n/\log\log n$, cf. \cite{rush2018finite}, and up to $t=\bigo\big(n/\mathrm{polylog }(n)\big)$ in the rank one spiked Wigner model, cf. \cite{li2022non,li2023approximate}. A different method of analysis in \cite{celentano2023local} proves the convergence of an AMP algorithm with spectral initialization as $t \to \infty$ in the $\mathbb{Z}_2$ synchronization model. It remains open to examine the validity of our theory, in particular in the regime $t\gg \log n$, under a general variance profile matrix $V$ and nonlinear $\{\mathsf{F}_t\}$'s.
	\item Since the results in \cite{li2022non} remain valid for much larger values of $t$, it is instructive to briefly compare their approach with ours. Their analysis can in fact  be viewed as a nontrivial generalization of the conditioning technique in \cite{bayati2011dynamics}, but it continues to make essential use of the rotational invariance property. A key technical contribution of \cite{li2022non} is the identification of precise conditions (e.g., \cite[Assumption 2]{li2022non}) under which the error terms remain small as $t$ increases. However, these conditions must be verified on a case-by-case basis---any general AMP theory cannot accommodate rapid growth in $t$, as even linear AMP may diverge quickly, for instance when $\mathsf{F} = 10\,\mathrm{id}$. In contrast, our leave-one-out approach avoids any reliance on the rotational invariance property of $A$ and applies more broadly to AMP with general variance profiles. In our more general setting, it remains unclear whether the iteration $t$ can grow rapidly. An interesting direction for future research is to identify specific conditions on the nonlinearities $\mathsf{F}$ that ensure stability even as $t$ becomes large.
    \item We compare Theorem \ref{thm:AMP_sym_avg} above to related results in the literature.
    \begin{enumerate}
        \item For $V$ satisfying $\max_{i \in [n]} \abs{n^{-1}\sum_{j \in [n]} V_{ij}^2- 1}\approx 0$, \cite[Theorem 2.4]{wang2024universality} obtained an asymptotic characterization for the AMP iterate that matches the GOE case with the scalar state evolution (\ref{eqn:intro_se}). For such $V$'s, the limiting spectrum of $A=V\circ G_n$ is the same semicircle law as in the GOE case, cf. \cite[Theorem 2.1]{erdos2012bulk}.
        \item When both $\{\mathsf{F}_t\}$'s and $\psi$ are polynomials, by applying the moment method to the matrix polynomials appearing in the unfolded AMP iterate, \cite[Theorem 3]{bayati2015universality} proved that the limiting behavior of $n^{-1}\sum_{k \in [n]} \psi\big( (\overline{z}^{(s)}_{k})_{s\in[t+1]}\big)$ is universal across a class of generalized Wigner ensembles; see \cite{hachem2023approximate} for further refinements to accommodate general $\{\mathsf{F}_t\}$'s and $\psi$, as well as possibly sparse $V$'s.
    \end{enumerate}
    Our Theorem \ref{thm:AMP_sym_avg} allows for complicated patterns of $V$ whose corresponding spectrum of $A$ can be significantly different from the semicircle law as in (a) (cf. \cite{ajanki2019quadratic}). Moreover, as mentioned in the introduction, our Theorem \ref{thm:AMP_sym_avg} is non-asymptotic in nature, whereas the moment method used in \cite{bayati2015universality,hachem2023approximate} as in (b) is intrinsically confined to obtain asymptotic results. 
    \item By consider the special univariate choice $\{\psi_k\equiv \psi_{1;k}:\R\to \R\}$ in Theorem \ref{thm:AMP_sym_avg}, we can also derive---similar to (\ref{ineq:AMP_sym_2})---an improved error bound of $n^{-1/c_0}$ for the average $n^{-1}\sum_{k \in [n]} \big(\psi_{1;k}( \overline{z}^{(t+1)}_{k})-  \E [\psi_{1;k}( Z^{(t+1)}_{k})|z^{(0)}]\big) $. However, since we are unaware of any significant applications for this improved bound, we omit these details for the sake of a cleaner presentation.
\end{enumerate}
\end{remark}

\subsection{Asymmetric AMP}
We consider an asymmetric version of the AMP, which is initialized with $(u^{(0)},v^{(0)})\in \R^m\times \R^n$ (here $u^{(0)} \in \R^m$ is a dummy variable with $\mathsf{G}_0\equiv 0$ introduced merely for notational consistency; the algorithm is actually initialized at $v^{(0)}$), and subsequently updated for $t=0,1,2,\ldots$, according to 
\begin{align}\label{def:AMP_asym}
\begin{cases}
u^{(t+1)} = A \mathsf{F}_{t}(v^{(t)})-b_t^{\mathsf{F}}\circ \mathsf{G}_{t}(u^{(t)}) \in \R^m,\\
v^{(t+1)}= A^\top \mathsf{G}_{t+1}(u^{(t+1)})-b_{t+1}^{\mathsf{G}}\circ\mathsf{F}_{t}(v^{(t)})\in \R^n.
\end{cases}
\end{align}
Here we consider the following random matrix model for $A\in \R^{m\times n}$: for some deterministic matrix $V \in \R^{m\times n}_{\geq 0}$ with non-negative entries,
\begin{align*}
A\equiv V\circ G_{m\times n},\quad \{(G_{m\times n})_{ij}\}_{i \in [m], j \in [n]}\stackrel{\textrm{i.i.d.}}{\sim} \mathcal{N}(0,1/m).
\end{align*}
The variance scaling $1/m$ is chosen in accordance to the AMP literature \cite{bayati2011dynamics,berthier2020state}. We also follow the same convention that the functions $\mathsf{F}_t= (\mathsf{F}_{t,\ell})_{\ell \in [n]}:\R^n \to \R^n$, $\mathsf{G}_t= (\mathsf{G}_{t,k})_{k \in [m]}:\R^m \to \R^m$ are understood as applied component-wise. For $t=1,2,\ldots$, the Onsager correction vectors $b_t^{\mathsf{F}} \in \R^m, b_t^{\mathsf{G}} \in \R^n$ are defined via
\begin{align}
\begin{cases}
b_{t,k}^{\mathsf{F}}\equiv \frac{1}{m}\sum_{\ell \in [n]} V_{k \ell}^2 \E \big[\mathsf{F}_{t,\ell}'(v^{(t)}_\ell)|v^{(0)}\big], &\quad k \in [m],\\
b_{t,\ell}^{\mathsf{G}}\equiv \frac{1}{m}\sum_{k \in [m]} V_{k\ell}^2 \E \big[\mathsf{G}_{t,k}'(u^{(t)}_k)|v^{(0)}\big], &\quad \ell \in [n].
\end{cases}
\end{align}
In the canonical applications of the asymmetric AMP (\ref{def:AMP_asym}) into high dimensional linear estimation problems, $m,n$ are understood as the sample size and the signal dimension, respectively. Moreover, $u^{(t+1)}\in \R^m$ is the residual vector, and $v^{(t+1)}\in \R^n$ is the AMP iterate for the statistical estimator of interest.

We shall first consider leave-one-out representation of $(u^{(t+1)},v^{(t+1)})\in \R^m\times \R^n$ in a similar vein to Theorem \ref{thm:AMP_sym_loo}. There are two leave-one-out versions associated with the asymmetric AMP (\ref{def:AMP_asym}):
\begin{itemize}
	\item (\emph{Leave-one-row-out version}): For $k \in [m]$, with the same initialization $(u^{(0)}_{[-k]},v^{(0)}_{[-k]})=(u^{(0)},v^{(0)})\in \R^m\times \R^n$, let for $t=0,1,2,\ldots$
	\begin{align}\label{def:AMP_asym_loo_row}
	\begin{cases}
	u^{(t+1)}_{[-k]} = A_{[-k],\cdot} \mathsf{F}_{t}(v^{(t)}_{[-k]})-b_t^{\mathsf{F}}\circ \mathsf{G}_{t}(u^{(t)}_{[-k]}) \in \R^m,\\
	v^{(t+1)}_{[-k]}= A_{[-k],\cdot}^\top \mathsf{G}_{t+1}(u^{(t+1)}_{[-k]})-b_{t+1}^{\mathsf{G}}\circ\mathsf{F}_{t}(v^{(t)}_{[-k]})\in \R^n.
	\end{cases}
	\end{align}
	\item (\emph{Leave-one-column-out version}): For all $\ell \in [n]$, with the same initialization $(u^{(0)}_{(\ell)},v^{(0)}_{(\ell)})=(u^{(0)},v^{(0)})\in \R^m\times \R^n$, let for $t=0,1,2,\ldots$
	\begin{align}\label{def:AMP_asym_loo_column}
	\begin{cases}
	u^{(t+1)}_{(-\ell)} = A_{\cdot,(-\ell)} \mathsf{F}_{t}(v^{(t)}_{(-\ell)})-b_t^{\mathsf{F}}\circ \mathsf{G}_{t}(u^{(t)}_{(-\ell)}) \in \R^m,\\
	v^{(t+1)}_{(-\ell)}= A_{\cdot,(-\ell)}^\top \mathsf{G}_{t+1}(u^{(t+1)}_{(-\ell)})-b_{t+1}^{\mathsf{G}}\circ\mathsf{F}_{t}(v^{(t)}_{(-\ell)})\in \R^n.
	\end{cases}
	\end{align}
\end{itemize}
Here for any $k \in [m]$ (resp. $\ell \in [n]$), $A_{[-k],\cdot}\equiv (A_{ij}\bm{1}_{i\neq k})_{i \in [m],j \in [n]}$ (resp. $A_{\cdot,(-\ell)}\equiv (A_{ij}\bm{1}_{j\neq \ell})_{i \in [m],j \in [n]}$) is obtained by setting all elements in the $k$-th row (resp. $\ell$-th column) of $A$ to be $0$ and keeping all other elements of $A$ unchanged.

In statistical language, (\ref{def:AMP_asym_loo_row}) and (\ref{def:AMP_asym_loo_column}) are obtained via the original AMP iterate (\ref{def:AMP_asym}) by discarding the $k$-th sample and the $\ell$-th predictor, respectively. 

\begin{theorem}\label{thm:AMP_asym_loo}
	Suppose the following hold for $\phi\equiv m/n$.
\begin{enumerate}
	\item[(S*1)] $(\phi^{-1}+1)^{1/2}\max_{k \in [m], \ell \in [n]}\abs{V_{k\ell}}\leq K$ holds for some $K\geq 2$.
	\item[(S*2)] $\mathsf{F}_{s,\ell}, \mathsf{G}_{s,k} \in C^2(\R)$ for all $s \in [0:t], k \in [m], \ell \in [n]$, and there exists some $\Lambda\geq 2$ such that
	\begin{align*}
	\max_{s \in [0:t]}\max_{k\in [m],\ell \in [n]}\Big\{\abs{\mathsf{F}_{s,\ell}(0)}+\abs{\mathsf{G}_{s,k}(0)}+\max_{q=1,2}\Big(\pnorm{\mathsf{F}_{s,\ell}^{(q)} }{\infty}+\pnorm{\mathsf{G}_{s,k}^{(q)} }{\infty}\Big)\Big\}\leq \Lambda.
	\end{align*}
\end{enumerate} 
  Then for any $D>0$, there exist some universal constant $c_0>0$ and another constant $c_1=c_1(D)>0$, such that with $\Prob(\cdot |v^{(0)})$-probability at least $1-c_1 n_\phi^{-D}$ where $n_\phi\equiv n(1+\phi)=m+n$, uniformly in $k \in [m], \ell \in [n]$,
\begin{align*}
&\bigabs{u^{(t+1)}_k-\bigiprod{A_{k\cdot}}{\mathsf{F}_t (v_{[-k]}^{(t)})} }\vee \bigabs{v^{(t+1)}_\ell-\bigiprod{A_{\cdot\ell}}{\mathsf{G}_{t+1} (u_{(-\ell)}^{(t+1)})} }\\
&\leq \big(c_1 K\Lambda\log n_\phi\cdot (1+\pnorm{v^{(0)}}{\infty})\big)^{c_0 (1\vee t)^3}\cdot n_\phi^{-1/2}.
\end{align*}
\end{theorem}
Note that in the above theorem, we have not assumed apriori $m\asymp n$, although this will be usually required in concrete applications.

Similar to Theorems \ref{thm:AMP_sym_loo} and \ref{thm:AMP_sym}, the leave-one-out representation above also leads to distributional characterizations for $\{u_k^{(t+1)}\},\{v_\ell^{(t+1)}\}$ via a high dimensional state evolution, defined as follows.
\begin{definition}[\emph{State evolution for asymmetric AMP}]\label{def:AMP_asym_se}
Let $U^{(1)}, U^{(2)}, \ldots\in \R^m$ and $V^{(1)}, V^{(2)},\ldots\in \R^n$ be two sequences of correlated, centered Gaussian vectors with independent entries. The correlation structure for each sequence along the iteration is determined via the recursive updates: with $V^{(0)}\equiv v^{(0)}$, for $s_1,s_2\in \mathbb{Z}_{\geq 0}$ and $k \in [m]$, $\ell \in [n]$,
	\begin{align*}
	\begin{cases}
	\cov(U^{(s_1+1)}_k, U^{(s_2+1)}_k)\equiv \frac{1}{m}\sum_{\ell \in [n]} V_{k\ell}^2 \E \big[\mathsf{F}_{s_1,\ell}\big(V^{(s_1)}_\ell\big)\mathsf{F}_{s_2,\ell}\big(V^{(s_2)}_\ell\big)|v^{(0)}\big],\\
	\cov(V^{(s_1+1)}_\ell, V^{(s_2+1)}_\ell)\equiv \frac{1}{m}\sum_{k \in [m]} V_{k\ell}^2 \E \big[\mathsf{G}_{s_1+1,k}\big(U_k^{(s_1+1)}\big) \mathsf{G}_{s_2+1,k}\big(U_k^{(s_2+1)}\big)|v^{(0)}\big].
	\end{cases}
	\end{align*}
\end{definition}
We note that the above recursion does not  specify the correlation between the two sequences $(U^{(1)},U^{(2)},\ldots)$ and $(V^{(1)},V^{(2)},\ldots)$. 

The following theorem provides an analogue of Theorem \ref{thm:AMP_sym} for the asymmetric AMP iterate $\{u_k^{(t+1)}\},\{v_\ell^{(t+1)}\}$.

\begin{theorem}\label{thm:AMP_asym}
	Suppose (S*1)-(S*2) hold for some $K,\Lambda\geq 2$. Fix any $\Lambda$-pseudo-Lipschitz function $\psi:\R^{t+1} \to \R$ of order $\mathfrak{p}\in \N$. Then there exist some universal constant $c_0>1$ and another constant $c_{\mathfrak{p}}>1$ depending on $\mathfrak{p}$ only, such that
	\begin{align}
	&\max_{k \in [m]}\bigabs{\E \big[\psi\big( u^{(t+1)}_k,u^{(t)}_k,\ldots,u^{(1)}_k\big)|v^{(0)}\big]-\E \big[\psi \big(U^{(t+1)}_k,U^{(t)}_k,\ldots,U^{(1)}_k\big)|v^{(0)}\big] }\nonumber\\
	&\qquad + \max_{\ell \in [n]}\bigabs{\E \big[\psi\big( v^{(t+1)}_\ell,v^{(t)}_\ell,\ldots,v^{(1)}_\ell\big)|v^{(0)}\big]-\E \big[\psi \big(V^{(t+1)}_\ell,V^{(t)}_\ell,\ldots,V^{(1)}_\ell\big)|v^{(0)}\big] }\nonumber\\
	& \leq \big(K\Lambda\log n_\phi\cdot (1+\pnorm{v^{(0)}}{\infty})\big)^{c_{\mathfrak{p}} (1\vee t)^3}\cdot n_\phi^{-1/c_0^{1\vee t}}.
	\end{align}
	Moreover, for any $\Lambda$-pseudo-Lipschitz function $\psi_1:\R \to \R$ of order $\mathfrak{p}\in \N$,
	\begin{align}
	&\max_{k \in [m]}\bigabs{\E \big[\psi_1\big( u^{(t+1)}_k\big)|v^{(0)}\big]-\E \big[\psi_1 \big(U^{(t+1)}_k\big)|v^{(0)}\big] }\nonumber\\
	&\qquad + \max_{\ell \in [n]}\bigabs{\E \big[\psi_1\big( v^{(t+1)}_\ell\big)|v^{(0)}\big]-\E \big[\psi_1 \big(V^{(t+1)}_\ell\big)|v^{(0)}\big] }\nonumber\\
	& \leq \big(K\Lambda \sigma_\ast^{[t+1],-1}\log n_\phi\cdot (1+\pnorm{v^{(0)}}{\infty})\big)^{c_{\mathfrak{p}} (1\vee t)^3}\cdot n_\phi^{-1/2}.
	\end{align}
	Here $\sigma_\ast^{[t+1]}\equiv  1\wedge \min\limits_{s \in [t+1]}\min\limits_{k \in [m],\ell \in [n]} \var^{1/2}(U_{k}^{(s)}|z^{(0)})\wedge \var^{1/2}(V_{\ell}^{(s)}|z^{(0)})$.
\end{theorem}

Similar to Theorem \ref{thm:AMP_sym_avg}, we may relax the $C^2$ smoothness condition in (S*2) when looking at the behavior of the empirical distributions of the modified AMP iterate $\{\overline{u}^{(t)}\},\{\overline{v}^{(t)}\}$ (whose Onsager correction vectors $\{\overline{b}_t^{\mathsf{F}}\}, \{\overline{b}_t^{\mathsf{G}}\}$ are defined by the Gaussian random vectors $\{U^{(\cdot)}\},\{V^{(\cdot)}\}$ in the state evolution in Definition \ref{def:AMP_asym_se}).

\begin{theorem}\label{thm:AMP_asym_avg}
	Suppose (S*1) holds and (S*2) is replaced by the following:
\begin{enumerate}
	\item[(SA*2)] There exists some $\Lambda\geq 2$ such that
	\begin{align*}
	\max_{s \in [0:t]}\max_{k\in [m],\ell \in [n]}\Big\{\abs{\mathsf{F}_{s,\ell}(0)}+\abs{\mathsf{G}_{s,k}(0)}+\pnorm{\mathsf{F}_{s,\ell} }{\mathrm{Lip}}+\pnorm{\mathsf{G}_{s,k} }{\mathrm{Lip}}\Big\}\leq \Lambda.
	\end{align*}
\end{enumerate} 
Fix a sequence of $\Lambda$-pseudo-Lipschitz function $\{\psi_k: \R^{t+1}\to \R\}_{k \in [m\vee n]}$ of order $2$. Then for any $D>0$, there exists some universal constant $c_0>0$ and another constant $c_1=c_1(D)$ such that with $\Prob(\cdot |v^{(0)})$-probability at least $1-c_1 n_\phi^{-D}$,
\begin{align*}
&\biggabs{\frac{1}{n_\phi}\sum_{k \in [m]}\Big(\psi_k\big( \overline{u}^{(t+1)}_k,\overline{u}^{(t)}_k,\ldots,\overline{u}^{(1)}_k\big)-\E \big[\psi_k \big(U^{(t+1)}_k,U^{(t)}_k,\ldots,U^{(1)}_k\big)|v^{(0)}\big]\Big)}\\
&\qquad \vee \biggabs{ \frac{1}{n_\phi}\sum_{\ell \in [n]}  \Big(\psi_\ell\big( \overline{v}^{(t+1)}_\ell,\overline{v}^{(t)}_\ell,\ldots,\overline{v}^{(1)}_\ell\big)-\E \big[\psi_\ell \big(V^{(t+1)}_\ell,V^{(t)}_\ell,\ldots,V^{(1)}_\ell\big)|v^{(0)}\big]\Big)  }\\
&\leq \big(c_1K\Lambda\sigma_{*,\psi}^{-1} \log n_\phi\cdot (1+\pnorm{v^{(0)}}{\infty})\big)^{c_0 (1\vee t)^3}\cdot n_\phi^{-1/c_0^{1\vee t}}
\end{align*}
Here $
\sigma_{\ast,\psi}^2\equiv 1\wedge \min\limits_{ k \in [m],\ell \in [n]}\big\{\min\limits_{s \in S_{\psi_k}\cap \mathbb{Z}_{\geq 2}}  \var(U^{(s-1)}_k|v^{(0)})\wedge \min\limits_{s \in S_{\psi_\ell}\cap \mathbb{Z}_{\geq 2}}  \var(V^{(s-1)}_\ell|v^{(0)})\big\}$, with $S_{\psi_\cdot}$ defined in Theorem \ref{thm:AMP_sym_avg}.
\end{theorem}

The proofs of Theorems \ref{thm:AMP_asym_loo}, \ref{thm:AMP_asym} and \ref{thm:AMP_asym_avg} for the asymmetric AMP (\ref{def:AMP_asym}) can be reduced to the symmetric case via a reduction scheme (cf. \cite{javanmard2013state,berthier2020state}). For sake of completeness, we provide some details for this reduction in Section \ref{subsection:proof_AMP_asym}.

\subsection{AMP with data-driven Onsager correction vectors}

\subsubsection{Symmetric case}

Consider the following data-driven AMP iterate:
\begin{align}\label{def:AMP_data}
\hat{z}^{(t+1)}&= A \mathsf{F}_t (\hat{z}^{(t)})- \hat{b}_{t}\circ \mathsf{F}_{t-1}(\hat{z}^{(t-1)}),\quad t=0,1,2,\ldots,
\end{align}
where $\mathsf{F}_t$ is defined the same as that in (\ref{def:AMP}). Here $\hat{z}^{(-1)} = 0_n$, $\mathsf{F}_{-1} \equiv 0$ and the initialization $\hat{z}^{(0)} = z^{(0)}$. For $t=1,2,\ldots$, the data-driven Onsager correction vector $\hat{b}_t=(\hat{b}_{t,k})_{k \in [n]}$ is defined via
\begin{align}\label{def:hat_b}
\hat{b}_{t,k}&\equiv \frac{1}{n}\sum_{\ell \in [n]} V_{k\ell}^2 \mathsf{F}_{t,\ell}' (\hat{z}^{(t)}_\ell),\quad k \in [n].
\end{align}
Compared to the oracle AMP iterate in (\ref{def:AMP}), the major difference here lies in the data-driven Onsager correction vectors $\{\hat{b}_t\}_{t\geq 0}$ in (\ref{def:hat_b}). 

The following theorem provides an $\ell_\infty$ control between the oracle AMP iterate (\ref{def:AMP}) and the data-driven AMP iterate (\ref{def:AMP_data}).

\begin{theorem}\label{thm:AMP_oracle_data_diff}
	Suppose the conditions in Theorem \ref{thm:AMP_sym_loo}	hold. Then for any $\vartheta \in (0,1/8)$ and  $D >0$, there exists some universal constant $c_0 >0$ and another constant $c_1 = c_1(\vartheta, D) >0$ such that  with $\Prob(\cdot |z^{(0)})$-probability at least $1-c_1 n^{-D}$,
	\begin{align*}
	\pnorm{\hat{z}^{(t+1)} - z^{(t+1)}}{\infty} \leq \big(c_1 K\Lambda \log n \cdot ( 1+ \pnorm{z^{(0)}}{\infty} )\big)^{c_0t^3} \cdot  n^{-1/4+\vartheta}.
	\end{align*} 
\end{theorem}

Combining the above theorem with Theorem \ref{thm:AMP_sym_loo}, we conclude that the data-driven AMP iterate $\hat{z}^{(t+1)}_k$ shares the same leave-one-out representation as $z^{(t+1)}_k$. Consequently, applying an identical state evolution analysis shows that Theorem \ref{thm:AMP_sym} remains valid when ${z}^{(\cdot)}_k$ is replaced by $\hat{z}^{(\cdot)}_k$. 

The proof of Theorem \ref{thm:AMP_oracle_data_diff} utilizes another non-trivial leave-one-out analysis, as it requires the strongest $\ell_\infty$ control rather than the relatively straightforward $\ell_2$ control; see Section \ref{subsection:proof_AMP_oracle_data_diff} for details. We believe this proof method can be adapted to obtain $\ell_\infty$ controls between different versions of AMPs. For instance, recall the AMP iterate $\{\bar{z}^{(t)}\}$ defined before Theorem \ref{thm:AMP_sym_avg} with the Onsager correction vectors $\{\overline{b}_{t,k}\equiv n^{-1}\sum_{\ell \in [n]} V_{k\ell}^2 \E [\mathsf{F}_{t,\ell}'(Z_\ell^{(t)})|z^{(0)}]\}_{k \in [n]}$. Then under the same assumptions as in Theorem \ref{thm:AMP_oracle_data_diff} and using the same notation, it is easy to prove that with $\Prob(\cdot |z^{(0)})$-probability at least $1-c_1 n^{-D}$,
\begin{align}\label{ineq:AMP_oracles_diff}
\pnorm{\bar{z}^{(t+1)} - z^{(t+1)}}{\infty} \leq \big(c_1 K\Lambda \sigma_\ast^{[t+1],-1} \log n \cdot ( 1+ \pnorm{z^{(0)}}{\infty} )\big)^{c_0t^3} \cdot n^{-1/4+\vartheta},
\end{align} 
where $\sigma_\ast^{[t+1]}$ is defined in Theorem \ref{thm:AMP_sym}. A sketch of proof for (\ref{ineq:AMP_oracles_diff}) is provided at the end of Section \ref{subsection:proof_AMP_oracle_data_diff}.

\subsubsection{Asymmetric case}
Consider the following data-driven AMP iterate: 
\begin{align}\label{def:AMP_asym_data}
\begin{cases}
\hat{u}^{(t+1)} = A \mathsf{F}_{t}(\hat{v}^{(t)})-\hat{b}_t^{\mathsf{F}}\circ \mathsf{G}_{t}(\hat{u}^{(t)}) \in \R^m,\\
\hat{v}^{(t+1)}= A^\top \mathsf{G}_{t+1}(\hat{u}^{(t+1)})-\hat{b}_{t+1}^{\mathsf{G}}\circ\mathsf{F}_{t}(\hat{v}^{(t)})\in \R^n,
\end{cases}
\end{align}
where $\mathsf{F}_t$ and $\mathsf{G}_t$ are defined the same as those in (\ref{def:AMP_asym}) and the initialization is given by $(\hat{u}^{(0)},\hat{v}^{(0)})= (u^{(0)}, v^{(0)})$. Recall that $\hat{u}^{(0)} \in \R^m$ is a dummy variable with $\mathsf{G}_0\equiv 0$, introduced merely for notational consistency. The algorithm is actually initialized at $\hat{v}^{(0)}$. For $t=1,2,\ldots$, the correction vectors $\hat{b}_t^{\mathsf{F}} \in \R^m, \hat{b}_t^{\mathsf{G}} \in \R^n$ are defined via
\begin{align}
\begin{cases}
\hat{b}_{t,k}^{\mathsf{F}}\equiv \frac{1}{m}\sum_{\ell \in [n]} V_{k \ell}^2 \mathsf{F}_{t,\ell}'(\hat{v}^{(t)}_\ell), &\quad k \in [m],\\
\hat{b}_{t,\ell}^{\mathsf{G}}\equiv \frac{1}{m}\sum_{k \in [m]} V_{k\ell}^2 \mathsf{G}_{t,k}'(\hat{u}^{(t)}_k), &\quad \ell \in [n].
\end{cases}
\end{align}
We formally state below an $\ell_\infty$-control on the difference between the oracle AMP iterate (\ref{def:AMP_asym}) and the data-driven AMP iterate (\ref{def:AMP_asym_data}). Recall that $n_\phi = m + n$.
\begin{theorem}
	Suppose the conditions in Theorem \ref{thm:AMP_asym_loo}	 hold. Then for any $\vartheta \in (0,1/8)$ and  $D >0$, there exists some universal constant $c_0 >0$ and another constant $c_1 = c_1(\vartheta, D) >0$ such that  with $\Prob(\cdot |v^{(0)})$-probability at least $1-c_1 n_\phi^{-D}$,
	\begin{align*}
		\pnorm{\hat{u}^{(t+1)} - u^{(t+1)}}{\infty} \vee \pnorm{\hat{v}^{(t+1)} - v^{(t+1)}}{\infty} \leq \big(c_1 K\Lambda \log n_\phi \cdot ( 1+ \pnorm{v^{(0)}}{\infty} )\big)^{c_0t^3} \cdot  n_\phi^{-1/4+\vartheta}.
	\end{align*} 
\end{theorem} 
The above theorem can be proven via Theorem \ref{thm:AMP_oracle_data_diff} using the same reduction scheme in Section \ref{subsection:proof_AMP_asym}. We omit these repetitive details.

\section{Applications to regularized estimation}\label{section:regularized_ls}

Recall the linear model (\ref{def:linear_model}) and the Ridge estimator $\hat{\mu}$ in (\ref{def:ls_estimator}). We write $A \in \R^{m\times n}$ as the design matrix that collects all $A_i$'s as its rows, and $\xi,Y \in \R^m$ as the error and response vector that collects all elements $\{\xi_i\}$ and $\{Y_i\}$. We will usually omit notational dependence on $\lambda>0$ for simplicity.

\subsection{The fixed point equation}
To describe our fixed point equation, let
\begin{align*}
\mathscr{V}_{m}\equiv m^{-1}V\circ V = \E A\circ A \in \R^{m\times n}
\end{align*}
be the variance profile matrix of $A$.  Consider the following system of (non-linear) equations in $(\gamma,b) \in \R^n_{\geq 0} \times [0,1)^m$: 
\begin{align}\label{eqn:fpe_ridge}
\begin{cases}
&\gamma^2 = \mathfrak{D}_{\tau_b}^2\mathscr{V}_m^\top \mathfrak{D}_{1_m-b}^2\big(\xi^2+ \mathscr{V}_m \mathfrak{D}_{\lambda \tau_b\circ (1_n+\lambda \tau_b)^{-1}}^2 \mu_0^2\big)\\
&\qquad\qquad +  \mathfrak{D}_{\tau_b}^2\mathscr{V}_m^\top \mathfrak{D}_{1_m-b}^2 \mathscr{V}_m \mathfrak{D}_{1_n+\lambda \tau_b}^{-2} \gamma^2 \in \R^n,\\
&b\circ (1_m-b)^{-1}= \mathscr{V}_m \big(\tau_b\circ (1_n+\lambda \tau_b)^{-1}\big) \in \R^m.
\end{cases}
\end{align}
Here $\tau_b\equiv (\mathscr{V}_m^\top (1_m-b))^{-1} \in \R^n$, and recall for a generic vector $x \in \R^n$, we write $x^{-1}=(x_i^{-1})_{i \in [n]}\in \R^n$, $x^2=(x_i^2)_{i \in [n]} \in \R^n$ and $\mathfrak{D}_x=(x_i\bm{1}_{i=j})_{i,j\in [n]}\in \R^{n\times n}$.

The following proposition proves the existence and uniqueness of the solution pair $(\gamma^\ast,b^\ast)$ to the high dimensional system of equations (\ref{eqn:fpe_ridge}).

\begin{proposition}\label{prop:ridge_fpe_exist_uniq}
	Suppose that $\pnorm{V_{k\cdot}}{},\pnorm{V_{\cdot\ell}}{}>0$ for all $k \in [m],\ell \in [n]$. Then there exists a unique pair $(\gamma^\ast,b^\ast) \in C(\R_{>0}\to \R^n_{\geq 0} \times [0,1)^m)$ that solves the system of equations (\ref{eqn:fpe_ridge}). 	
\end{proposition}

In fact, the above proposition proves a stronger statement in that the solution $\lambda \mapsto (\gamma^\ast(\lambda),b^\ast(\lambda))$ exists uniquely as a continuous function of $\lambda$. As a consequence, the `effective regularization' $\tau_{b^\ast(\cdot)}\equiv \big(\mathscr{V}_m^\top (1_m-b^\ast(\cdot))\big)^{-1}:\R_{>0}\to \R_{>0}^n$ (the precise meaning of this terminology will be clear in the next subsection) is locally bounded away from both $0$ and $\infty$.

It is worth noting that for the Gaussian i.i.d. design, i.e., $V=1_{m\times n}$ (or equivalently, $\mathscr{V}_m=1_{m\times n}/m$), any solution $(\gamma^\ast,b^\ast)$ to (\ref{eqn:fpe_ridge}) must take the form $(\gamma^\ast,b^\ast)=(\gamma^\ast_0 1_n,b^\ast_0 1_n)$ for some scalars $\gamma^\ast_0\geq 0, b^\ast_0\in [0,1)$. In this case, (\ref{eqn:fpe_ridge}) reduces to a system of two equations in two unknowns, and can be solved with a closed form by direct calculations (see, e.g., Eqn. (\ref{eqn:rs_iid_fpe}) below).

The proof of Proposition \ref{prop:ridge_fpe_exist_uniq} is fairly non-trivial due to the high dimensionality of the system of equations (\ref{eqn:fpe_ridge}). The key to the proof is to identify appropriate metric-like structures, under which certain transformations of the solutions $(\gamma^\ast,b^\ast)$ for (\ref{eqn:fpe_ridge}) are fixed points of a strict contraction mapping. More concretely:
\begin{itemize}
	\item For the first equation of (\ref{eqn:fpe_ridge}), with a suitable linear transformation of $\gamma$, we will prove strict contraction with respect to the $\ell_\infty$ metric;
	\item For the second equation of (\ref{eqn:fpe_ridge}), inspired by random matrix techniques that deal with vector-valued Stieljes transforms (see, e.g., \cite{alt2017local,ajanki2019quadratic} and references therein), under a suitable non-linear transformation of $b$, we will first prove strict contraction with respect to a weighted-$\ell_\infty$, symmetric function that does not strictly define a metric. We then use a hyperbolic transformation of this symmetric function that leads to a well-defined metric to induce the existence of a fixed point.	
\end{itemize}
The details of the arguments can be found in Section \ref{subsection:proof_ridge_fpe_exist_uniq}.

\subsection{Distribution of the Ridge estimator}

Recall $\hat{\mu}^{\seq}(\gamma^\ast;\tau_{b^\ast})$ defined in (\ref{eqn:ridge_seq_est}).
Let $\hat{R} \equiv Y - A\hat{\mu}$ be the residual and 
\begin{align*}
R^\ast \equiv\Big( (\mathfrak{D}_{1_m-b^\ast}^2\mathscr{V}_m\mathfrak{D}_{\tau_{b^\ast}})\mathfrak{D}_{1+\lambda\tau_{b^\ast}}^{-2}(\mathfrak{D}_{\lambda^2\tau_{b^\ast}}\mu_0^2  + \mathfrak{D}_{\tau_{b^\ast}}^{-1}(\gamma^\ast)^2 )\Big)^{1/2} \circ \mathsf{Z}_m + \mathfrak{D}_{1_m-b^\ast}\xi.
\end{align*}
be the `population' version of $\hat{R}$. Here $\mathsf{Z}_m \sim \mathcal{N}(0,I_m)$. The following theorem provides a formal description of (\ref{eqn:ridge_seq_dist}) and connects $\hat{R}$ and $R^\ast$. Recall that $\Prob^{\xi}$ and $\E^{\xi}$ represent the probability and expectation conditional on $\xi$, respectively.

\begin{theorem}\label{thm:ridge_dist}
	Suppose $1/K\leq m/n, \lambda \leq K$ for some $K\geq 2$. Further suppose $m^{-1/2}\big(\pnorm{V_{k\cdot}}{}\wedge  \pnorm{V_{\cdot\ell}}{}\big)\geq 1/K$, and  $V_{k\ell}\leq K$ holds for all $k \in [m],\ell \in [n]$. Fix any $\Lambda$-pseudo-Lipschitz function $\psi: \R^2\to \R$ and $\psi_0: \R\to \R$ of order $2$ with $\Lambda \geq 2$. Then for any $D>1$, there exists some constant $c_1=c_1(D,K,\Lambda)>0$, with $\Prob^\xi$-probability at least $1-c_1 n^{-D}$,
	\begin{align}\label{ineq:thm:ridge_dist}
	&\biggabs{\frac{1}{n}\sum_{j \in [n]} \psi\big(\hat{\mu}_j,\mu_{0,j}\big)-\frac{1}{n}\sum_{j \in [n]} \E^\xi \psi\big(\hat{\mu}^{\seq}_j(\gamma^\ast;\tau_{b^\ast}),\mu_{0,j}\big)}\nonumber\\
	&\qquad \vee \biggabs{\frac{1}{m}\sum_{i \in [m]} \psi_0\big(\hat{R}_i\big)-\frac{1}{m}\sum_{i \in [m]} \E^\xi \psi_0\big(R^\ast_i\big)}\leq c_1 L_{\mu_0,\xi} ^2\cdot e^{-(\log \log n)^{0.99}/c_1}.
	\end{align}
	Here $L_{\mu_0,\xi}\equiv  1 + n^{-1/2}(\pnorm{\mu_0}{}\vee \pnorm{\xi}{})$.
\end{theorem}

We note that here the error $\xi$ is treated as fixed in the above theorem. It is easy to generalize to the `usual' statistical setting with random $\xi$'s; we omit these details. 

For the Gaussian i.i.d. design case $V=1_{m\times n}$, distributional results as in Theorem \ref{thm:ridge_dist} can be proven via Gaussian comparison methods \cite{thrampoulidis2018precise,han2023universality,han2023distribution}. In this case, for more specific choices of $\psi$, for instance $\psi(\mu,\mu_0)\equiv \pnorm{\mu-\mu_0}{}^2$ (which corresponds to the $\ell_2$ error of $\hat{\mu}$), analysis of $n^{-1}\sum_{j \in [n]} \psi\big(\hat{\mu}_j,\mu_{0,j}\big)$ can also be appealed directly to existing random matrix methods (and therefore Gaussianity is typically not required); see e.g., \cite{dicker2016ridge,dobriban2018high,hastie2022surprises}.

Below we present some simulation for Theorem \ref{thm:ridge_dist}. In the simulation, we use $m=100,n=200$, $\mu_0=1_{n}$, and fix a variance profile matrix $V \in \R^{m\times n}$ and an error vector $\xi \in \R^m$, whose entries form a random draw from i.i.d. $\abs{\mathcal{N}(1,1)}$ and $\mathcal{N}(0,1)$ respectively. The simulation results are summarized as follows:
\begin{itemize}
	\item In the left panel of Figure \ref{fig:1}, we verify that the $\ell_2$ estimation error for the Ridge estimator $\hat{\mu}$ matches the prediction in Theorem \ref{thm:ridge_dist}, both under Gaussian and non-Gaussian designs with a general variance profile. Specifically, for non-Gaussian designs, we use (centered) Bernoulli distributions and Student's $t$-distribution with 10 degrees of freedom (denoted as $t(10)$).
	\item In the middle and right panels of Figure \ref{fig:1}, we plot the bias and variance for the single coordinate $\hat{\mu}_1$ versus their theoretical counterparts for $\hat{\mu}^{\seq}_1$. These quantities are seen to match well with each other in that $\E \hat{\mu}_1\approx \E \hat{\mu}_1^{\seq}$ and $\var(\hat{\mu}_1)\approx \var(\hat{\mu}_1^{\seq})$.
\end{itemize}
 We note here $\E \hat{\mu}_1$ and $\var(\hat{\mu}_1)$ are estimated by $\hat{\E} \hat{\mu}_1\equiv B^{-1}\sum_{b \in [B]} \hat{\mu}_1(b)$ and $\widehat{\var}(\hat{\mu}_1)\equiv (B-1)^{-1}\sum_{b \in [B]}(\hat{\mu}_1(b)-\hat{\E} \hat{\mu}_1)^2$ via Monte Carlo replicates $\{\hat{\mu}_1(b)\}_{b \in [B]}$. Moreover, the quantities $\E \hat{\mu}_1^{\seq} = {\mu_{0,j}}/(1+\lambda \tau_{b^\ast,1})$ and $\var(\hat{\mu}_1^{\seq}) = \big({\gamma_1^\ast}/(1+\lambda \tau_{b^\ast,1})\big)^2$ involve the the first coordinates of $\tau_{b^\ast}$ and $\gamma^\ast$, which can be efficiently computed using an iterative algorithm constructed in the proof of Proposition \ref{prop:ridge_fpe_exist_uniq}.

\begin{figure}[t]
	\begin{minipage}[t]{0.3\textwidth}
		\includegraphics[width=\textwidth]{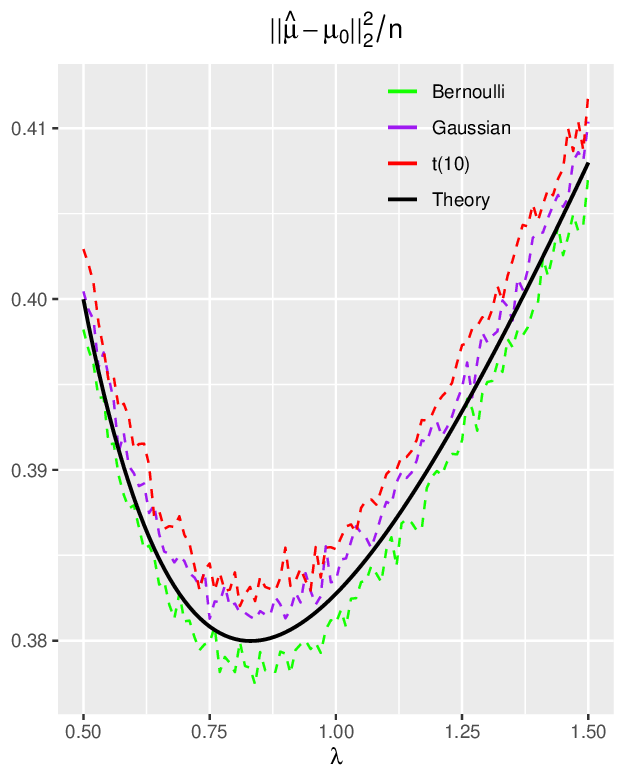}
	\end{minipage}
	\begin{minipage}[t]{0.3\textwidth}
		\includegraphics[width=\textwidth]{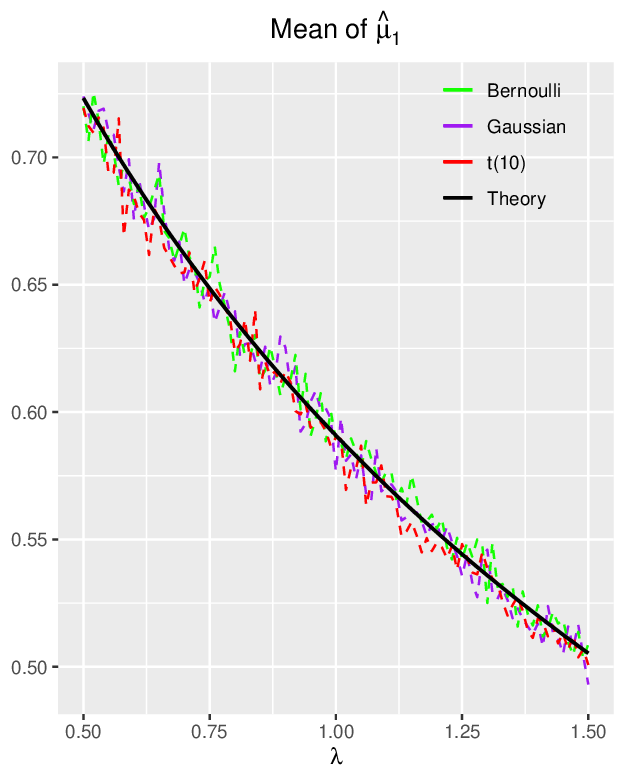}
	\end{minipage}
	\begin{minipage}[t]{0.3\textwidth}
	\includegraphics[width=\textwidth]{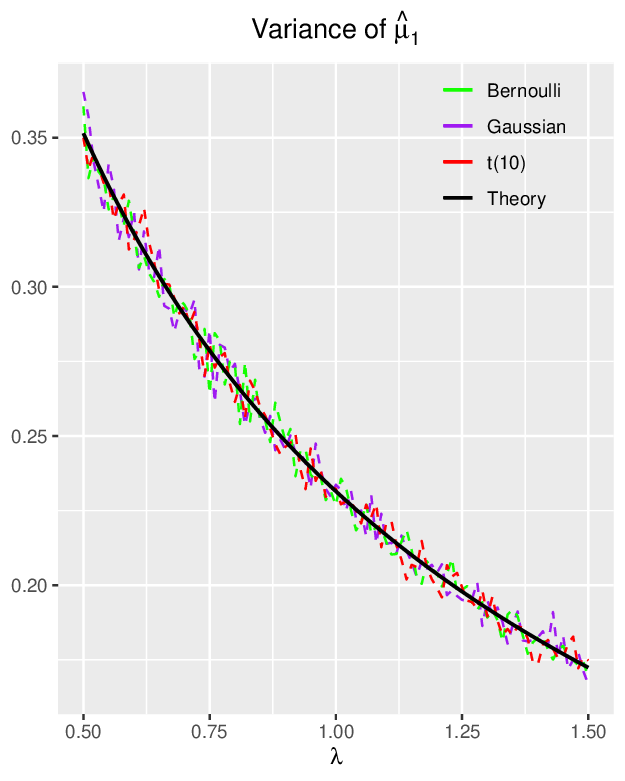}
\end{minipage}
	\caption{\emph{Left panel}: Empirical vs. theoretical $\ell_2$ estimation error of the Ridge estimator $\hat{\mu}$. Different colors represent different design matrix distributions. \emph{Middle and right panels}: Empirical vs. theoretical values of $\E \hat{\mu}_1$ and $\var(\hat{\mu}_1)$. \emph{Simulation parameters}: $m=100,n=200$ with $B=5000$ Monte Carlo repetitions. }
	\label{fig:1}
\end{figure}

\subsection{Main steps for the proof of Theorem \ref{thm:ridge_dist} via AMP}

We will prove Theorem \ref{thm:ridge_dist} by approximating the Ridge estimator $\hat{\mu}$ via an AMP iterate, whose distributional properties can be analyzed via the AMP theory developed in Section \ref{section:main_results}. This proof strategy is developed in \cite{bayati2012lasso} for the Lasso estimator, and in \cite{donoho2016high} for robust estimators, both under Gaussian i.i.d. designs (i.e. $V=1_{m\times n}$). The key technical difference in our setting is that we need to work with a suitable transformation of the random matrix $A$ in the constructed AMP iterate, due to the general variance profile matrix $V$. To this end, let
\begin{align*}
&A_{b^\ast}\equiv \mathfrak{D}^{1/2}_{1_m-b^\ast}A \mathfrak{D}_{\tau_{b^\ast}}^{1/2},\quad\xi_{b^\ast}\equiv \mathfrak{D}^{1/2}_{1_m-b^\ast}\xi,\quad \theta_{0;b^\ast}\equiv \mathfrak{D}_{\tau_{b^\ast}}^{-1/2} \mu_0,
\end{align*}
and let $\mathsf{f}_\eta:\R^n\to \R$ be defined by $\mathsf{f}_{\eta}(\mu)\equiv 2^{-1}\lambda\sum_{\ell \in [n]} \eta_\ell \mu_\ell^2$.
Consider the following iterative algorithm, initialized with $r^{(0)}=0_m, \theta^{(0)}=\theta_{0;b^\ast}$, and updated according to
\begin{align}\label{def:AMP_ls}
\begin{cases}
r^{(t+1)}\equiv A_{b^\ast}(\theta_{0;b^\ast}- \theta^{(t)})+\xi_{b^\ast}+ b^\ast\circ r^{(t)} \in \R^m,\\
\theta^{(t+1)}\equiv \prox_{\mathsf{f}_{\tau_{b^\ast} }} \big(\theta^{(t)}+A_{b^\ast}^\top r^{(t+1)}\big) \in \R^n. 
\end{cases}
\end{align}
Let $(\hat{r},\hat{\theta})$ be a stationary point of (\ref{def:AMP_ls}), $\mu^{(t+1)}\equiv \mathfrak{D}_{\tau_{b^\ast}}^{1/2} \theta^{(t+1)}$, and $R^{(t+1)}\equiv \mathfrak{D}_{1_m-b^\ast}^{1/2}r^{(t+1)}$.

The following proposition provides the precise connection between both (i) the stationary point of (\ref{def:AMP_ls}) to the pair of the residual $\hat{R}$ and the Ridge estimator $\hat{\mu}$, and (ii) (\ref{def:AMP_ls}) and the standard form of an asymmetric AMP iterate.

\begin{proposition}\label{prop:ridge_AMP}
	The following hold.
	\begin{enumerate}
		\item The stationary point $(\hat{r},\hat{\theta})$ of (\ref{def:AMP_ls}) satisfies $(\hat{R}, \hat{\mu})= (\mathfrak{D}_{1_m-b^\ast}^{1/2}\hat{r}, \mathfrak{D}_{\tau_{b^\ast}}^{1/2} \hat{\theta})$. 
		\item (\ref{def:AMP_ls}) can be recast into the standard asymmetric AMP iterate (\ref{def:AMP_asym}) via the following identification: with initialization $v^{(0)}=0$, for $t=0,1,2,\ldots$, the sequence $(u^{(t+1)},v^{(t+1)})\in \R^m\times \R^n$ is updated according to
		\begin{align}\label{def:AMP_standard_form_ls}
		\begin{cases}
		u^{(t+1)}=\xi_{b^\ast}-r^{(t+1)}, & v^{(t+1)}=\theta_{0;b^\ast}-\theta^{(t)}-A_{b^\ast}^\top r^{(t+1)},\\
		\mathsf{G}_t(u)=(u-\xi_{b^\ast})\bm{1}_{t\geq 1}, & \mathsf{F}_t(v)= \big(\prox_{\mathsf{f}_{ \tau_{b^\ast} }}(\theta_{0;b^\ast}-v)-\theta_{0;b^\ast}\big)\bm{1}_{t\geq 1}.
		\end{cases}
		\end{align}
		Moreover, we have 
		\begin{align*}
		\begin{cases}
		\mu^{(t+1)}=\prox_{\lambda\pnorm{\cdot}{}^2/2}(\mu_{0}-\mathfrak{D}_{\tau_{b^\ast}}^{1/2} v^{(t+1)};\tau_{b^\ast}),\\
		R^{(t+1)}=-\mathfrak{D}_{1_m-b^\ast}^{1/2}u^{(t+1)}+\mathfrak{D}_{1_m-b^\ast}\xi.
		\end{cases}	
		\end{align*}
	\end{enumerate}
	
\end{proposition}

\begin{proof}
	\noindent (1). By the first-order optimality condition for $\hat{\mu}$, for any $\eta \in \R^n_{>0}$,
	\begin{align*}
	\hat{\mu} = \prox_{\mathsf{f}_{\eta} }\Big(\hat{\mu}+\mathfrak{D}_\eta A^\top \big(A(\mu_0-\hat{\mu})+\xi\big)\Big).
	\end{align*}
	On the other hand, a stationary point $\hat{\theta}$ of (\ref{def:AMP_ls}) satisfies
	\begin{align*}
	& \hat{\theta}= \prox_{\mathsf{f}_{\tau_{b^\ast}}} \Big[ \hat{\theta}+A_{b^\ast}^\top \mathfrak{D}_{1_m-b^\ast}^{-1} \Big(A_{b^\ast}\big( \theta_{0;b^\ast}- \hat{\theta}\big)+\xi_{b^\ast}\Big)\Big],\\
	\Leftrightarrow \quad & \mathfrak{D}_{\tau_{b^\ast}}^{1/2}\hat{\theta}= \prox_{\mathsf{f}_{\tau_{b^\ast} } }  \Big[ \mathfrak{D}_{\tau_{b^\ast}}^{1/2} \hat{\theta}+\mathfrak{D}_{\tau_{b^\ast}}^{1/2}A_{b^\ast}^\top \mathfrak{D}_{1_m-b^\ast}^{-1} \Big(A_{b^\ast}\big( \theta_{0;b^\ast}- \hat{\theta}\big)+\xi_{b^\ast}\Big)\Big],\\
	\Leftrightarrow \quad & \mathfrak{D}_{\tau_{b^\ast}}^{1/2} \hat{\theta} = \prox_{\mathsf{f}_{\tau_{b^\ast} } }  \Big[ \mathfrak{D}_{\tau_{b^\ast}}^{1/2} \hat{\theta}+\mathfrak{D}_{\tau_{b^\ast}} A^\top \Big(A\big( \mu_{0}-\mathfrak{D}_{\tau_{b^\ast}}^{1/2} \hat{\theta}\big)+\xi\Big)\Big].
	\end{align*}
	Comparing the above two equations leads to $\hat{\mu} = \mathfrak{D}_{\tau_{b^\ast}}^{1/2} \hat{\theta}$. A stationary point $\hat{r}$ of (\ref{def:AMP_ls}) satisfies
	\begin{align*}
	&\hat{r} = A_{b^\ast}(\theta_{0;b^\ast}- \hat{\theta})+\xi_{b^\ast}+ b^\ast\circ \hat{r}, \\
	\Leftrightarrow \quad & \mathfrak{D}_{1_m-b^\ast} \hat{r} = \mathfrak{D}_{1_m-b^\ast}^{1/2} A\mathfrak{D}_{\tau_{b^\ast}}^{1/2} (\mathfrak{D}_{\tau_{b^\ast}}^{-1/2}\mu_0- \hat{\theta})+\mathfrak{D}_{1_m-b^\ast}^{1/2}\xi, \\
	\Leftrightarrow \quad & \mathfrak{D}_{1_m-b^\ast}^{1/2} \hat{r} = A\mu_0 + \xi - A\hat{\mu} = Y-A\hat{\mu}  = \hat{R}.
	\end{align*}
	This completes the proof of (1).
	
	\noindent (2). From (\ref{def:AMP_standard_form_ls}), it is easily verified that for $t\geq 1$, $\mathsf{F}_t(v^{(t)})=\theta^{(t)}-\theta_{0;b^\ast}$ and $\mathsf{G}_t(u^{(t)})=-r^{(t)}$. Furthermore, the Onsager correction vector $b_{t}^{\mathsf{G}}= 1_n$, as for all $\ell \in [n]$, $b_{t,\ell}^{\mathsf{G}}=\sum_{k \in [m]} \E (A_{b^\ast})_{k\ell}^2=1$ due to the normalization $\tau_{b^\ast}$. Meanwhile, the other Onsager correction vector $b_{t}^{\mathsf{F}}=b^\ast$, as for $k \in [m]$,
	\begin{align*}
	b_{t,k}^{\mathsf{F}}= \sum_{\ell \in [n]} \E (A_{b^\ast})_{k\ell}^2 \big(1+\lambda \tau_{b^\ast,\ell}\big)^{-1} = \frac{1-b^\ast_k}{m}\sum_{\ell \in [n]} V_{k\ell}^2 \frac{ \tau_{b^\ast,\ell} }{1+\lambda \tau_{b^\ast,\ell}}=b^\ast_k,
	\end{align*}
	where the last identity uses the second equation of (\ref{eqn:fpe_ridge}). Due to the linearity of $\mathsf{F}_t,\mathsf{G}_t$, it makes no difference to use either $\{b_\cdot^{\mathsf{F}}, b_\cdot^{\mathsf{G}}\}$ or $\{\overline{b}_\cdot^{\mathsf{F}}, \overline{b}_\cdot^{\mathsf{G}}\}$.
\end{proof}

Next we shall quantify the $\ell_2$ error between the iterate $\theta^{(t)}$ and the stationary point $\hat{\theta}=\mathfrak{D}_{\tau_{b^\ast}}^{-1/2}\hat{\mu}$ in (\ref{def:AMP_ls}). 

\begin{proposition}\label{prop:ridge_theta_error}
	Suppose $1/K\leq m/n, \lambda\leq K$ for some $K\geq 2$. Further suppose $m^{-1/2}\big(\pnorm{V_{k\cdot}}{}\wedge  \pnorm{V_{\cdot\ell}}{}\big)\geq 1/K$, and  $V_{k\ell}\leq K$ holds for all $k \in [m],\ell \in [n]$. Then for any $D>1$, there exist some universal constant $c_0>0$ and another constant $c_1=c_1(D)>0$, such that for all $1\leq t\leq \log n/c_0$, with $\Prob^\xi(\cdot|\theta^{(0)})$-probability at least $1-c_1 n^{-D}$,
	\begin{align*}
	&n^{-1}\max\big\{\pnorm{\theta^{(t)}-\hat{\theta}}{}^2, \pnorm{r^{(t)}-\hat{r}}{}^2\big\}\\
	&\leq K^{c_0} \cdot n^{-1}\big(\pnorm{\mu_0}{}^2\vee \pnorm{\xi}{}^2\big)\cdot (1-K^{-c_0})^t+\big(K \log n\cdot (1+\pnorm{z^{(0)}}{\infty})\big)^{c_0 t^3}\cdot n^{-1/c_0^t}.
	\end{align*}
\end{proposition}

An important step in the proof of Proposition \ref{prop:ridge_theta_error} is to relate the first equation of (\ref{eqn:fpe_ridge}) to the high dimensional state evolution in Definition \ref{def:AMP_asym_se}, which then allows applications of the joint distributional characterizations in Theorem \ref{thm:AMP_asym_avg}.

Under Gaussian i.i.d. designs, similar $\ell_2$ error controls between a suitably designed AMP iterate and a given regularized regression estimator of interest are obtained in several cases, typically in the high dimensional limit $m/n\to \phi \in (0,\infty)$ followed by $t\to \infty$. Examples along this line include the Lasso estimator \cite{bayati2012lasso}, robust estimators \cite{donoho2016high}, the SLOPE estimator \cite{bu2021algorithmic}, $\ell_1$-minimum norm interpolator \cite{li2021minimum} in the linear model, and the maximum likelihood estimator in logistic regression \cite{sur2019likelihood,sur2019modern}. Here we provide non-asymptotic $\ell_2$ error controls in the Ridge case under a general variance profile on the design matrix. 

With Proposition \ref{prop:ridge_theta_error} in hand, we may then use our AMP theory in Theorem \ref{thm:AMP_asym_avg} to derive the distribution of $\hat{\mu}$. The details of the proofs for both Proposition \ref{prop:ridge_theta_error} and Theorem \ref{thm:ridge_dist} can be found in Section \ref{subsection:proof_ridge_dist}.

\subsection{Conjectured fixed point equations for general regularized least squares}
The above method of analysis for the Ridge estimator $\hat{\mu}$ in (\ref{def:ls_estimator}) via our AMP theory suggests a broader paradigm for a general class of regularized least squares estimators. 

Let $\mathsf{f}:\R\to \R_{\geq 0}$ be a non-negative, proper and closed convex function, and let $\mathsf{f}_n(\mu)\equiv \sum_{j \in [n]} \mathsf{f}(\mu_j)$. We are interested in the behavior of the regularized least squares estimator $
\hat{\mu}_{\mathsf{f}}\in \argmin_{\mu \in \R^n} \big\{2^{-1}\pnorm{Y-A\mu}{}^2+\mathsf{f}_n(\mu)\big\}$. We conjecture that $\hat{\mu}_{\mathsf{f}}\stackrel{d}{\approx} \prox_{\mathsf{f}_n}(\mu_0+\gamma^\ast_{\mathsf{f}}\circ Z_n;\tau_{b^\ast_{\mathsf{f}}} )$ in a similar sense to Theorem \ref{thm:ridge_dist}, where $(\gamma^\ast_{\mathsf{f}},b^\ast_{\mathsf{f}}) \in \R_{\geq 0}^n\times [0,1)^m$ is now the solution to the following high dimensional system of equations (whenever exists uniquely and is suitably bounded):
\begin{align}\label{eqn:rs_general_fpe}
\begin{cases}
\gamma^2 =\mathfrak{D}_{\tau_b}^2\mathscr{V}_m^\top \mathfrak{D}_{1_m-b}^2\Big[\xi^2+ \mathscr{V}_m \E^\xi \big(\prox_{\mathsf{f}_n}(\mu_0+\gamma\circ Z_n;\tau_b)-\mu_0\big)^2\Big] \in \R^n,\\
b\circ (1_m-b)^{-1}= \mathscr{V}_m \mathfrak{D}_{\tau_b} \E^\xi \prox_{\mathsf{f}_n}'(\mu_0+\gamma\circ Z_n;\tau_b) \in \R^m.
\end{cases}
\end{align}
In the Gaussian i.i.d. design case $V=1_{m\times n}$ (i.e., $\mathscr{V}_m=1_{m\times n}/m$), (\ref{eqn:rs_general_fpe}) reduces to a system of two equations in two unknowns $(\gamma,\tau) \in \R_{\geq 0}\times [1,\infty)$: with $\sigma_\xi^2\equiv \pnorm{\xi}{}^2/m$,
\begin{align}\label{eqn:rs_iid_fpe}
\begin{cases}
\gamma^2 = \sigma_\xi^2+ \frac{1}{m} \E^\xi  \pnorm{\prox_{\mathsf{f}_n}(\mu_0+\gamma  Z_n;\tau)-\mu_0}{}^2,\\
1-1/\tau = \frac{1}{m} \E^\xi \dv_{\prox_{\mathsf{f}_n}(\cdot;\tau)}(\mu_0+\gamma  Z_n).
\end{cases}
\end{align}
The above system of equations integrate its limiting version in \cite{elkaroui2018impact,thrampoulidis2018precise}, the Lasso case in \cite{miolane2021distribution} with $\mathsf{f}_n(\mu)\equiv \lambda \pnorm{\mu}{1}$, and a class of general non-separable, non-smooth projection-type regularizers in \cite{han2023noisy}.

Due to the apparent complications in the analysis of the fixed point equations (\ref{eqn:rs_general_fpe}), it remains an outstanding open question to formally establish the prescribed distributional properties of $\hat{\mu}_{\mathsf{f}}$ via (\ref{eqn:rs_general_fpe}) for general regularizers $\mathsf{f}_n$.

\section{Proof of Theorem \ref{thm:AMP_sym_loo}: main steps}\label{section:proof_AMP_loo}

In the proofs, we write $\E[\cdot]=\E[\cdot|z^{(0)}]$ for simplicity. For notational simplicity, $t$'s appearing in the estimates are understood as $1\vee t$.

\subsection{Decomposition via leave-one-out AMP}

Recall the leave-one-out AMP $\{z_{[-k]}^{(t)}\}$ defined in (\ref{def:AMP_loo}). Let us define some further notation:
\begin{itemize}
	\item $\mathsf{B}_t\equiv \mathfrak{D}_{b_t}\in \R^{n\times n}$.
	\item $\Delta z_{[-k]}^{(t)}\equiv z^{(t)}-z_{[-k]}^{(t)}$.
	\item $z_{[-k]}^{(t)}(u) \equiv (1-u) z_{[-k]}^{(t)}+u  z^{(t)}$, where $u \in [0,1]$.
	\item $\Delta A_{[-k]}\equiv A-A_{[-k]}=e_k (A_k-A_{kk} e_k)^\top +A_k e_k^\top$.
	\item $(\mathsf{F}')^{(t)}\equiv \mathsf{F}'_t\big(z^{(t)}\big) \in \R^n$, $\mathsf{D}^{(t)}\equiv \mathfrak{D}_{(\mathsf{F}')^{(t)}}\in \R^{n\times n}$.
	\item $(\mathsf{F}')_{[-k]}^{(t)}\equiv  \mathsf{F}'_t\big(z_{[-k]}^{(t)}\big) \in \R^n$, $\mathsf{D}_{[-k]}^{(t)}\equiv \mathfrak{D}_{(\mathsf{F}')_{[-k]}^{(t)}}\in \R^{n\times n}$.
	\item $\overline{\mathsf{F}'}_{[-k]}^{(t)}\equiv \E_U \mathsf{F}'_t\big(z_{[-k]}^{(t)}(U)\big) \in \R^n$, $\overline{\mathsf{D}}_{[-k]}^{(t)}\equiv \mathfrak{D}_{\overline{\mathsf{F}'}_{[-k]}^{(t)}}\in \R^{n\times n}$, with $\E_U$ being the expectation taken over $U\sim \mathrm{Unif}[0,1]$ independent of $A$.
\end{itemize}
Using these notation, a first-order Taylor expansion yields that for $r=1,2,\ldots$,
\begin{align}\label{eqn:delta_z_recursion}
\Delta z_{[-k]}^{(r)}& = A_{[-k]} \overline{\mathsf{D}}_{[-k]}^{(r-1)} \Delta z_{[-k]}^{(r-1)}-\mathsf{B}_{r-1} \overline{\mathsf{D}}_{[-k]}^{(r-2)}\Delta z_{[-k]}^{(r-2)}+\Delta A_{[-k]} \mathsf{F}_{r-1}(z^{(r-1)}).
\end{align}
For any given $t$, we define a sequence of matrices $\overline{M}_{0,[-k]}^{(t)} ,\overline{M}_{1,[-k]}^{(t)},\ldots, \overline{M}_{t,[-k]}^{(t)}  \in \R^{n\times n}$ recursively as follows: let  $\overline{M}_{-1,[-k]}^{(t)} =0_{n\times n}$, $\overline{M}_{0,[-k]}^{(t)} \equiv \overline{\mathsf{D}}_{[-k]}^{(t)}$, and for $s \in [t]$, 
\begin{align}\label{def:M_bar_recursion}
\overline{M}_{s,[-k]}^{(t)} \equiv \big(\overline{M}_{s-1,[-k]}^{(t)}  A_{[-k]} - \overline{M}_{s-2,[-k]}^{(t)}  \mathsf{B}_{t-s+1}\big) \overline{\mathsf{D}}_{[-k]}^{(t-s)}. 
\end{align}
The next proposition provides an exact error decomposition for the leave-one-out representation in Theorem \ref{thm:AMP_sym_loo}, via the matrix sequence defined above. 
\begin{proposition}\label{prop:AMP_loo_decom}
	Suppose $\mathsf{F}_{s,\ell} \in C^1(\R)$ for all $s \in [0:t]$ and $\ell \in [n]$. Then for any $k \in [n]$,
	\begin{align*}
	z^{(t+1)}_k-\bigiprod{A_k}{\mathsf{F}_t (z_{[-k]}^{(t)})} = \sum_{s \in [t]} \rem_k^{(s)},
	\end{align*}
	where for $s \in [t]$,
	\begin{align*}
	\rem_k^{(s)}&\equiv \iprod{A_k}{ \overline{M}_{t-s,[-k]}^{(t)}  e_k}\cdot  \Big(z^{(s)}_k+b_{s-1,k} \mathsf{F}_{s-2,k}(z^{(s-2)}_k)-A_{kk}\mathsf{F}_{s-1,k}(z^{(s-1)}_k)\Big)\\
	&\qquad + \big(\iprod{A_k}{\overline{M}_{t-s,[-k]}^{(t)}  A_k }- b_{t,k}\bm{1}_{s=t}\big)\cdot \mathsf{F}_{s-1,k}(z^{(s-1)}_k).
	\end{align*}
	Further suppose (S1) in Theorem \ref{thm:AMP_sym_loo} holds for some $K\geq 2$, and (S2) is replaced by $\max_{s \in [0:t]}\max_{\ell \in [n]}\big\{\abs{\mathsf{F}_{s,\ell}(0)}+\pnorm{\mathsf{F}_{s,\ell}' }{\infty}\big\}\leq \Lambda$ for some $\Lambda\geq 2$. Then there exists some universal constant $c_0>0$ such that 
	\begin{align*}
	\bigabs{z^{(t+1)}_k-\bigiprod{A_k}{\mathsf{F}_t (z_{[-k]}^{(t)})}}&\leq (K\Lambda)^{c_0 t}\cdot (1+\abs{A_{kk}})\cdot \max_{s \in [0:t]}\big(1+\abs{z^{(s)}_k}\big)\\
	&\quad \times \Big(\max_{s \in [t]} \bigabs{ \big(\iprod{A_k}{\overline{M}_{t-s,[-k]}^{(t)}  A_k }- b_{t,k}\bm{1}_{s=t}\big)   }+\abs{A_{kk}}\Big).
	\end{align*}
\end{proposition}

\begin{proof}
	First, using (\ref{eqn:delta_z_recursion}) for $r=t$,
	\begin{align*}
	z^{(t+1)}_k &= \bigiprod{ A \mathsf{F}_t(z^{(t)})-\mathsf{B}_{t} \mathsf{F}_{t-1}(z^{(t-1)}) }{e_k} \\
	&= \bigiprod{A_k}{\mathsf{F}_t (z_{[-k]}^{(t)})}+\bigiprod{A_k}{ \overline{M}_{0,[-k]}^{(t)} \Delta z_{[-k]}^{(t)} }- b_{t,k} \mathsf{F}_{t-1,k}(z^{(t-1)}_k)\\
	& = \bigiprod{A_k}{\mathsf{F}_t (z_{[-k]}^{(t)})}+ \bigiprod{A_k}{\overline{M}_{1,[-k]}^{(t)}  \Delta z^{(t-1)}_{[-k]} - \overline{M}_{0,[-k]}^{(t)}  \mathsf{B}_{t-1} \overline{\mathsf{D}}_{[-k]}^{(t-2)} \Delta z_{[-k]}^{(t-2)} }+\rem_k^{(t)}.
	\end{align*}
	Next, using (\ref{eqn:delta_z_recursion}) for $r=t-1$, 
	\begin{align*}
	&\bigiprod{A_k}{\overline{M}_{1,[-k]}^{(t)}  \Delta z^{(t-1)}_{[-k]} - \overline{M}_{0,[-k]}^{(t)}  \mathsf{B}_{t-1} \overline{\mathsf{D}}_{[-k]}^{(t-2)} \Delta z_{[-k]}^{(t-2)} }\\
	& = \bigiprod{A_k}{\overline{M}_{2,[-k]}^{(t)}  \Delta z^{(t-2)}_{[-k]} - \overline{M}_{1,[-k]}^{(t)}  \mathsf{B}_{t-2} \overline{\mathsf{D}}_{[-k]}^{(t-3)} \Delta z_{[-k]}^{(t-3)} }+ \rem_k^{(t-1)}.
	\end{align*}
	The recursion terminates when (\ref{eqn:delta_z_recursion}) is used until $r=1$, where by using the terminating condition $\Delta z_{[-k]}^{(0)}=0$, we arrive at
	\begin{align*}
	z^{(t+1)}_k & = \bigiprod{A_k}{\mathsf{F}_t (z_{[-k]}^{(t)})}+\rem_k^{(t)}+\cdots+ \rem_k^{(1)}.
	\end{align*}
	The claimed identity follows. The claimed estimate follows by noting that $\abs{\iprod{A_k}{ \overline{M}_{t-s,[-k]}^{(t)} e_k}}\leq (K\Lambda)^{2t}\abs{A_{kk}}$.  
\end{proof}

In view of the error bound in Proposition \ref{prop:AMP_loo_decom}, in the next two subsections we will then (i) provide coordinate-wise controls for the AMP iterate $\{z^{(t)}\}$, as well as (ii) compute the normalized trace of the matrix $\overline{M}_{s,[-k]}^{(t)}$ defined via the recursion (\ref{def:M_bar_recursion}), as $\overline{M}_{s,[-k]}^{(t)}$ and $A_k$ are approximately independent and thus the estimate of the quadratic forms such as $\langle A_k, \overline{M}_{t-s, [-k]}^{(t)} A_k\rangle$ boils down to the estimate of tracial quantities of $\overline{M}_{t-s, [-k]}^{(t)}$ via Hanson-Wright inequality.

\subsection{Some apriori estimates}

The following proposition provides delocalization estimates for the AMP iterate $\{z^{(t)}\}$.

\begin{proposition}\label{prop:AMP_iterate_infty}
	Suppose (S1) in Theorem \ref{thm:AMP_sym_loo} holds for some $K\geq 2$, and (S2) is replaced by $\max_{s \in [0:t]}\max_{\ell \in [n]}\big\{\abs{\mathsf{F}_{s,\ell}(0)}+\pnorm{\mathsf{F}_{s,\ell}' }{\infty}\big\}\leq \Lambda$ for some $\Lambda\geq 2$. Then there exists some universal constant $c_0>0$ such that for $x\geq 1$, 
	\begin{align*}
	\max_{k \in [n]}\Prob\Big(\pnorm{z^{(t)}}{\infty}\vee \pnorm{z^{(t)}-z^{(t)}_{[-k]}}{}\geq (K\Lambda)^{c_0 t} x\cdot (1+\pnorm{z^{(0)}}{\infty})\big| z^{(0)}\Big)\leq c_0\cdot tn e^{-(x^2\wedge n)/c_0}. 
	\end{align*}
	Moreover, for any $p\geq 1$, there exists some $C_p>0$ such that for $t\leq n/(C_p\log n)$,
	\begin{align*}
	\E^{1/p}\big[\pnorm{z^{(t)}}{\infty}^p|z^{(0)}\big]\vee \max_{k \in [n]}\E^{1/p}\big[\pnorm{z^{(t)}-z^{(t)}_{[-k]}}{}^p|z^{(0)}\big]\leq (K\Lambda)^{C_p t}\sqrt{\log n} \cdot (1+\pnorm{z^{(0)}}{\infty}).
	\end{align*}
\end{proposition}

For $i, j \in [n]$, let $\partial_{ij}=\partial/\partial A_{ij}$, and
\begin{align}\label{def:Delta_ij}
\Delta_{ij}\equiv 2^{-1}\big(1+\bm{1}_{i\neq j}\big)\big(e_ie_j^\top+e_je_i^\top\big)=\partial_{ij} A.
\end{align}
The following proposition provides delocalization estimates for $\{\partial_{ij} z^{(t)}\}$.

\begin{proposition}\label{prop:z_der_l2}
	Suppose (S1)-(S2) in Theorem \ref{thm:AMP_sym_loo} hold for some $K,\Lambda\geq 2$. For any $p\geq 1$, there exists some $C_p>0$ such that for $t\leq n/(C_p \log n)$,
	\begin{align*}
	\E^{1/p}\big[\pnorm{\partial_{ij} z^{(t)} }{}^p|z^{(0)}\big]&\leq (K\Lambda)^{C_p t}\sqrt{\log n} \cdot (1+\pnorm{z^{(0)}}{\infty}),
	\end{align*}
	and for $t\leq \log n/C_p$, 
	\begin{align*}
	\E^{1/p} \Big[\max_{i,j,\ell\in [n],\ell \neq i,j} \big(\sqrt{n} \abs{ \partial_{ij} z_{\ell}^{(t)} }\big)^p\big| z^{(0)}\Big] &\leq \big((K\Lambda)^t \log n \cdot (1+\pnorm{z^{(0)}}{\infty})\big)^{C_p t}.
	\end{align*}
\end{proposition}

Proofs for both of the above propositions can be found in Section \ref{section:proof_technical_results_proof_main_thm}.

\subsection{The trace control}

Consider a sequence of matrices $M_{0}^{(t)}, M_{1}^{(t)},\ldots, M_{t}^{(t)} \in \R^{n\times n}$ defined recursively as follows: $M_{-1}^{(t)}=0_{n\times n}$, $M_{0}^{(t)} =  \mathsf{D}^{(t)}$, and for $s \in [t]$,
\begin{align}\label{eqn:M_recursion}
M_{s}^{(t)}\equiv \big(M_{s-1}^{(t)} A - M_{s-2}^{(t)} \mathsf{B}_{t-s+1}\big) \mathsf{D}^{(t-s)}. 
\end{align}
The following proposition provides the key trace control for a generalized version of $M_s^{(t)}$.
\begin{proposition}\label{prop:trace_M}
	Suppose (S1)-(S2) in Theorem \ref{thm:AMP_sym_loo} hold for some $K,\Lambda\geq 2$. Then there exists some universal constant $c_0>0$ such that for $t\leq \log n/c_0$,
	\begin{align*}
	\max_{d_0 \in \{V_{[1:n]}^2, 1_n\}} \max_{s \in [t]} \frac{1}{n}\bigabs{\E \big[\tr(\mathfrak{D}_{d_0}M_s^{(t)})|z^{(0)}\big] } \leq \big(K\Lambda\log n\cdot (1+\pnorm{z^{(0)}}{\infty})\big)^{c_0 t^3} \cdot n^{-1/2}.
	\end{align*}
	
\end{proposition}

\subsubsection{The GOE and $\mathsf{F}_\cdot = \mathrm{id}$ case}
Before presenting the proof for the above proposition, let us examine the simplest possible case, where $V=1_{n\times n}$, and $\mathsf{F}_{t}=\mathrm{id}$ is the coordinate-wise identity map. In this case, $\mathsf{D}^{(\cdot)}=\mathsf{B}_\cdot= I_n$, and the matrix recursion (\ref{eqn:M_recursion}) reduces to 
\begin{align}\label{eqn:GOE_M_recursion}
M_s = 2M_{s-1} A_0- M_{s-2},\quad  M_0=I_n, \quad M_{-1} = 0_{n\times n},
\end{align}
where $A_0\equiv A/2$. Let $U_s(\cdot)$ be the Chebyshev polynomials of the second kind, defined recursively via
\begin{align}\label{eqn:GOE_chebyshev_poly}
U_s(x)=2x U_{s-1}(x)- U_{s-2}(x),\quad U_0(x)\equiv 1,\quad U_{-1}(x)\equiv 0.
\end{align}
The structural similarity between (\ref{eqn:GOE_M_recursion}) and (\ref{eqn:GOE_chebyshev_poly}) leads to the following  observation. Let $\mu_n\equiv n^{-1}\sum_{i \in [n]} \delta_{\lambda_i (A_0)}$ be the empirical spectral distribution of $A_0$. As $A_0$ is approximately (scaled) GOE, $\mu_n$ is close to the (scaled) semicircle law $\mu_\infty$ on $[-1,1]$: $\mu_n(\d x)\approx \mu_\infty(\d x) = (2/\pi)\sqrt{1-x^2}\,\d{x}$. Now using the fact that $\{U_s\}_{s \in \mathbb{Z}_{\geq 0}}$ are orthogonal polynomials in $L_2(\mu_\infty)$, we have for any $s\geq 1$,
\begin{align*}
\frac{1}{n} \tr (M_s) = \frac{1}{n}\sum_{i \in [n]} U_s\big(\lambda_i(A_0)\big) = \int_{-1}^1 U_s(x)\,\mu_n(\d x)\approx \int_{-1}^1 U_s(x)U_0(x)\,\mu_\infty(\d x) = 0.
\end{align*}
Unfortunately, the essential use of the linearity of $\mathsf{F}_t$ and the GOE property of $A$ in the above (rough) argument makes it less amenable to further generalizations to cover the general cases in Proposition \ref{prop:trace_M}. In the proof below, instead of estimating the trace of this complicated matrix polynomial at once, we eliminate the high complexity randomness in $M_s^{(t)}$ step by step via Gaussian integration by parts, and eventually arrive at simple linear and quadratic forms which can be estimated via Gaussian concentration directly.

\subsubsection{Some further notation and preliminary facts}  

For generic matrices $\mathsf{M}_0,\mathsf{M}_1,\mathsf{M}_2,\mathsf{M}_3 \in \R^{n\times n}$, we write 
\begin{align*}
\bigiprod{\mathsf{M}_1 \partial_{\cdot\cdot} \mathsf{M}_0 \mathsf{M}_2 }{\mathsf{M}_3}\equiv \sum_{i,j \in [n] } \big(\mathsf{M}_1 \partial_{ij} \mathsf{M}_0  \mathsf{M}_2\big)_{ij} (\mathsf{M}_{3})_{ij}.
\end{align*} 
Recall $\Delta_{ij}=\partial_{ij} A$ defined in (\ref{def:Delta_ij}). We have the derivative formula: for $i, j \in [n]$, and $r \in [t]$, 
\begin{align}\label{eqn:M_der_recursion}
\partial_{ij} M_r^{(t)}&=\big(\partial_{ij} M_{r-1}^{(t)} A - \partial_{ij} M_{r-2}^{(t)} \mathsf{B}_{t-r+1}\big)  \mathsf{D}^{(t-r)}+ M_{r-1}^{(t)}\Delta_{ij}  \mathsf{D}^{(t-r)}\nonumber\\
&\qquad + \big(M_{r-1}^{(t)} A - M_{r-2}^{(t)} \mathsf{B}_{t-r+1}\big) \mathfrak{D}_{\mathsf{F}^{(2)}_{t-r}(z^{(t-r)})}\mathfrak{D}_{\partial_{ij} z^{(t-r)}}.
\end{align}
For a fixed $s \in [t]$, we consider another sequence of matrices $N_0^{(s)},\ldots, N_s^{(s)} \in \R^{n\times n}$, defined recursively as follows: $N_{-1}^{(s)}\equiv 0_{n\times n}, N_0^{(s)}\equiv I_n$, and for $u \in [s]$,
\begin{align}\label{eqn:N_recursion}
N_u^{(s)}\equiv  A  \mathsf{D}^{(t-s+u)} N_{u-1}^{(s)} - \mathsf{B}_{t-s+u} \mathsf{D}^{(t-s+u-1)} N_{u-2}^{(s)}.
\end{align}
Then following derivative formula holds:
\begin{align}\label{eqn:N_der_recursion}
\partial_{ij} N_u^{(s)}&= \big(A  \mathsf{D}^{(t-s+u)} \partial_{ij} N_{u-1}^{(s)} - \mathsf{B}_{t-s+u}  \mathsf{D}^{(t-s+u-1)} \partial_{ij}N_{u-2}^{(s)}\big)+ \Delta_{ij} \mathsf{D}^{(t-s+u)} N_{u-1}^{(s)}\nonumber\\
&\qquad + A \mathfrak{D}_{\mathsf{F}^{(2)}_{t-s+u}(z^{(t-s+u)})}\mathfrak{D}_{\partial_{ij} z^{(t-s+u)}} N_{u-1}^{(s)} \nonumber\\
&\qquad- \mathsf{B}_{t-s+u}  \mathfrak{D}_{\mathsf{F}^{(2)}_{t-s+u-1}(z^{(t-s+u-1)})}\mathfrak{D}_{\partial_{ij} z^{(t-s+u-1)}} N_{u-2}^{(s)}.
\end{align}
The sequence of matrices $\{N_\cdot^{(\cdot)}\}$ can be viewed, in a certain sense, a reversed recursion to $\{M_\cdot^{(\cdot)}\}$, the role of which will be clear in the proof in the next subsection. 

The following estimates on $\{ M^{(\cdot)}_{\cdot} \}$ and $\{ N^{(\cdot)}_{\cdot} \}$ are proved in Section \ref{section:proof_technical_results_proof_main_thm}.
\begin{lemma}\label{lem:M_N_moment}
	Suppose (S1)-(S2) in Theorem \ref{thm:AMP_sym_loo} hold for some $K,\Lambda\geq 2$. There exists some universal $c_0>0$ such that for $t\leq n/(c_0\log n)$ and any $p \geq 1$,
	\begin{align}\label{ineq:trace_pop_M_N_moment}
	&n^{-1/2}\Big\{ \E^{1/p}\pnorm{M^{(t)}_\cdot}{F}^p\vee \E^{1/p}\pnorm{N^{(s)}_\cdot}{F}^p \Big\}\nonumber\\
	&\qquad \qquad + \big( \E^{1/p}\pnorm{M^{(t)}_\cdot}{\op}^p\vee \E^{1/p}\pnorm{N^{(s)}_\cdot}{\op}^p\big) \leq C_p (K\Lambda)^{c_0t}.
	\end{align}
	Here the constant $C_p>0$ may depend on $p$. Moreover, for any two symmetric matrices $A_1,A_2 \in \R^{n\times n}$,
	\begin{align}\label{ineq:trace_pop_M_N_difference}
	&\pnorm{M_\cdot^{(t)}(A_1)-M_\cdot^{(t)}(A_2)}{F} \vee \pnorm{N_\cdot^{(s)}(A_1)-N_\cdot^{(s)}(A_2)}{F}\nonumber\\
	& \leq \big(K\Lambda (1+\pnorm{z^{(0)}}{\infty}) (1+ \pnorm{A_1}{\op}+ \pnorm{A_2}{\op})\big)^{c_0 t^2}\cdot \sqrt{n}  \pnorm{A_1-A_2}{\op}.
	\end{align}
\end{lemma}

\subsubsection{Proof of Proposition \ref{prop:trace_M}}

First, for $3\leq s\leq t$, we show that the normalized trace of $M_s^{(t)}$ can be reduced to that of $\{N_\cdot^{(\cdot)}\}$, the most complicated one being $N_{s-2}^{(s)}$. The following formulation also allows for the base case $s=1,2$.

\begin{lemma}\label{lem:trace_cal_M_to_N}
	Suppose (S1)-(S2) in Theorem \ref{thm:AMP_sym_loo} hold for some $K,\Lambda\geq 2$. There exists some universal constant $c_0>0$ such that for any fixed $d_0 \in \R^n$ and an index set $\mathcal{T}\subset [-1:t]$ allowing for multiplicity, and any $1\leq s\leq t\leq \log n/c_0$, 
	\begin{align}\label{ineq:trace_pop_claim_1}
	&\frac{1}{n}\bigabs{\E \tr\big(\mathfrak{D}_{d_0}M_s^{(t)}\mathsf{D}^{(\mathcal{T})}\big)}\leq \big(K\Lambda\log n\cdot (1+\pnorm{z^{(0)}}{\infty})\big)^{c_0 t^2}\cdot n^{-1/2} \\
	& + (K\Lambda)^{c_0 t^2}\sum_{m\in [s-2]}\max_{ \substack{d_{[c_0 m]} \in \{V_{[1:n]}^2,b_{[0:t]},1_n\},\\ \mathcal{S}\in [-1:t]^{c_0 m} }}  \frac{1}{n}\bigabs{\E \tr\Big(\mathfrak{D}_{d_0\circ \otimes_{q \in [c_0 m]}d_q} N_{s-m-1}^{(s-m+1)} \mathsf{D}^{(\mathcal{T})} \mathsf{D}^{(\mathcal{S})}\Big)}.\nonumber
	\end{align}
	Here we write $\mathsf{D}^{(\mathcal{T})}\equiv \prod_{s \in \mathcal{T}}\mathsf{D}^{(s)}$ and $\mathsf{D}^{(-1)}\equiv I_{n}$ for notational convenience.
\end{lemma}
\begin{proof}
	\noindent (\textbf{Step 1}). First consider the case $s=1,2$. Using the recursion (\ref{eqn:M_recursion}),
	\begin{align*}
	M^{(t)}_1= \mathsf{D}^{(t)} A \mathsf{D}^{(t-1)},\quad M^{(t)}_2= \mathsf{D}^{(t)}\big(A \mathsf{D}^{(t-1)} A- \mathsf{B}_{t-1}\big) \mathsf{D}^{(t-2)}.
	\end{align*}
	For $s=1$, a trivial bound suffices: 
	\begin{align*}
	\frac{1}{n} \bigabs{\E \tr (M_1^{(t)})} \leq  \frac{1}{n}\sum_{i \in [n]} \E \bigabs{\mathsf{F}'_{t,i}(z^{(t)}_i)\mathsf{F}'_{t-1,i}(z^{(t-1)}_i) A_{ii}}\leq \frac{(K \Lambda)^c}{\sqrt{n}}.
	\end{align*}
	For $s=2$, note that
	\begin{align}\label{ineq:trace_pop_quad_1}
	\frac{1}{n} \bigabs{\E \tr (M_2^{(t)})} \leq \frac{\Lambda^2}{n}\sum_{k \in [n]} \E\abs{(A\mathsf{D}^{(t-1)}A)_{kk}-b_{t-1,k}}. 
	\end{align}
	For each $k \in [n]$,
	\begin{align}\label{ineq:trace_pop_quad_2}
	&\E \abs{(A\mathsf{D}^{(t-1)}A)_{kk}-b_{t-1,k}}\nonumber\\
	&\leq \E\biggabs{\sum_{j \in [n]} \big(A_{kj}^2-{V_{kj}^2}/{n}\big)  \mathsf{F}'_{t-1,j}(z^{(t-1)}_j)}+ \E\biggabs{\frac{1}{n}\sum_{j \in [n]} V_{kj}^2\big(\mathrm{id}-\E\big) \mathsf{F}'_{t-1,j}(z^{(t-1)}_j)}\nonumber\\
	&\equiv \mathfrak{Q}_1^{(t)}+  \mathfrak{Q}_2^{(t)}.
	\end{align}
	For $\mathfrak{Q}_1^{(t)}$, we have for $t\leq \log n/c$,
	\begin{align}\label{ineq:trace_pop_quad_3}
	\mathfrak{Q}_1^{(t)}&\leq \E \biggabs{\sum_{j \in [n]} \big(A_{kj}^2-{V_{kj}^2 }/{n}\big)  \big(\mathsf{F}'_{t-1,j}(z^{(t-1)}_j)- \mathsf{F}'_{t-1,j}(z^{(t-1)}_{[-k],j})\big)}\nonumber\\
	& \qquad + \E \biggabs{\sum_{j \in [n]} \big(A_{kj}^2-{V_{kj}^2}/{n}\big)  \mathsf{F}'_{t-1,j}(z^{(t-1)}_{[-k],j})}\nonumber\\
	&\leq (K\Lambda)^c \cdot n^{-1/2}\big(\E^{1/2}\pnorm{z^{(t-1)}-z^{(t-1)}_{[-k]}}{}^2+1\big)\nonumber\\
	&\leq  (K\Lambda)^{ct}(1+\pnorm{z^{(0)}}{\infty})\sqrt{\log n}\cdot n^{-1/2}.
	\end{align}
	Here in the last line we used Proposition \ref{prop:AMP_iterate_infty}.
	
	For $\mathfrak{Q}_2^{(t)}$, let $H_0(A)\equiv n^{-1}\sum_{j \in [n]} V_{kj}^2 \mathsf{F}'_{t-1,j}(z^{(t-1)}_j(A))$. Then for any two symmetric matrices $A_1,A_2 \in \R^{n\times n}$,
	\begin{align*}
	\abs{H_0(A_1)-H_0(A_2)}&\leq (K\Lambda)^c \cdot n^{-1/2} \pnorm{z^{(t-1)}(A_1)- z^{(t-1)}(A_2)}{}\\
	&\leq \big(K\Lambda (1+\pnorm{A_1}{\op}+\pnorm{A_2}{\op})\big)^{ct}\cdot \pnorm{A_1-A_2}{\op}.
	\end{align*}
	Moreover, 
	\begin{align*}
	\abs{H_0(A)}\leq (K\Lambda)^c\cdot \big(1+n^{-1/2}\pnorm{z^{(t-1)}}{}\big)\leq \big(K\Lambda  (1+\pnorm{A}{\op})\big)^{ct} (1+\pnorm{z^{(0)}}{\infty}).
	\end{align*}
	By applying the Gaussian concentration Lemma \ref{lem:gaussian_conc}, we have for $t\leq \log n/c$,
	\begin{align}\label{ineq:trace_pop_quad_4}
	\mathfrak{Q}_2^{(t)}&\leq \var^{1/2}(H_0(A))\leq (K\Lambda)^{ct}(1+\pnorm{z^{(0)}}{\infty})\log n\cdot n^{-1/2}.
	\end{align}
	Combining (\ref{ineq:trace_pop_quad_1})-(\ref{ineq:trace_pop_quad_4}), we have
	\begin{align*}
	\frac{1}{n} \bigabs{\E \tr (M_2^{(t)})} &\leq (K\Lambda)^{ct} (1+\pnorm{z^{(0)}}{\infty})\log n\cdot n^{-1/2}.
	\end{align*}
	The claimed, slightly more general form follows from minor modifications.

	\noindent (\textbf{Step 2}). Next assume $s\geq 3$. For $r \in [s-1]$, let
	\begin{align*}
	R^{(s)}_{r}&\equiv \frac{1}{n^2}\E\biggiprod{\mathsf{D}^{(t-s)} M_{r-1}^{(t)}\Delta_{\cdot\cdot} \mathsf{D}^{(t-r )} N_{s-1-r}^{(s)} }{V\circ V  }\\
	&+ \frac{1}{n^2}\E\biggiprod{\mathsf{D}^{(t-s)}   \big(M_{r-1}^{(t)} A - M_{r-2}^{(t)} \mathsf{B}_{t-r+1}\big) \mathfrak{D}_{\mathsf{F}^{(2)}_{t-r}(z^{(t-r)})}\mathfrak{D}_{\partial_{\cdot\cdot} z^{(t-r)}} N_{s-1-r}^{(s)} }{V\circ V  }.
	\end{align*}
	By Gaussian integration by parts,
	\begin{align*}
	\frac{1}{n}\E \tr(M_s^{(t)})&= \frac{1}{n}\sum_{i,j \in [n]} \E \mathsf{D}^{(t-s)}_{ii}  M_{s-1,ij}^{(t)} A_{ij}-\frac{1}{n}\E \tr\big( \mathsf{D}^{(t-s)} M_{s-2}^{(t)} \mathsf{B}_{t-s+1}\big)\\
	& = \frac{1}{n^2}\bigiprod{ \E\mathsf{D}^{(t-s)}  \partial_{\cdot\cdot} M_{s-1}^{(t)}  }{V\circ V }-\frac{1}{n}\E \tr\big(\mathsf{D}^{(t-s)} M_{s-2}^{(t)} \mathsf{B}_{t-s+1}\big)\\
	&\qquad + \frac{1}{n^2}\sum_{i,j \in [n] } \E  \mathsf{F}^{(2)}_{t-s,i}(z^{(t-s)}_{i})\partial_{ij} z^{(t-s)}_{i}  M_{s-1,ij}^{(t)} V_{ij}^2.
	\end{align*}
	Using (\ref{eqn:M_der_recursion}) with $r=s-1$, the first term above equals 
	\begin{align*}
	&\frac{1}{n^2} \E \biggiprod{ \mathsf{D}^{(t-s)} \Big(\partial_{\cdot\cdot} M_{s-2}^{(t)} N_1^{(s)} - \partial_{\cdot\cdot} M_{s-3}^{(t)}  \mathsf{B}_{t-s+2}  \mathsf{D}^{(t-s+1)} N_0^{(s)} \Big) }{V\circ V}+R^{(s)}_{s-1}.
	\end{align*}
	Using (\ref{eqn:M_der_recursion}) with $r=s-2$, the first term of the above display equals 
	\begin{align*}
	&\frac{1}{n^2} \E \biggiprod{\mathsf{D}^{(t-s)} \Big(\partial_{\cdot\cdot} M_{s-3}^{(t)} N_2^{(s)} - \partial_{\cdot\cdot} M_{s-4}^{(t)}  \mathsf{B}_{t-s+3}  \mathsf{D}^{(t-s+2)} N_1^{(s)} \Big) }{V\circ V}+R^{(s)}_{s-2}.
	\end{align*}
	Consequently, iterating this procedure until (\ref{eqn:M_der_recursion}) is applied with $r=1$, we obtain 
	\begin{align}\label{ineq:trace_pop_1}
	\frac{1}{n}\E \tr(M_s^{(t)})&=R^{(s)}_{s-1}+\cdots+R^{(s)}_{1} -\frac{1}{n}\E \tr\big(\mathsf{D}^{(t-s)} M_{s-2}^{(t)} \mathsf{B}_{t-s+1}\big)\nonumber\\
	&\qquad + \frac{1}{n^2}\sum_{i,j \in [n] } \E \mathsf{F}_{t-s,i}'(z_i^{(t-s)}) \mathsf{F}_{t,i}^{(2)}(z_i^{(t)}) \partial_{ij} z_i^{(t)} N_{s-1,ij}^{(s)} V_{ij}^2 \nonumber\\
	&\qquad + \frac{1}{n^2}\sum_{i,j \in [n]} \E  \mathsf{F}^{(2)}_{t-s,i}(z^{(t-s)}_{i})\partial_{ij} z^{(t-s)}_{i}  M_{s-1,ij}^{(t)} V_{ij}^2.
	\end{align}
	\noindent (\textbf{Step 2-(a)}). We provide a bound for the last two terms in (\ref{ineq:trace_pop_1}): using the moment estimates in (\ref{ineq:trace_pop_M_N_moment}) and Proposition \ref{prop:z_der_l2}, for $t\leq \log n/c$, 
	\begin{align}
	&\hbox{Last two terms of (\ref{ineq:trace_pop_1})}\nonumber\\
	&\leq (K\Lambda )^c\cdot \max_{u \in [0:t]}\E^{1/2} \max_{i,j \in [n]} \pnorm{\partial_{ij} z^{(u)}}{\infty}^2\cdot n^{-1}\big(\E^{1/2} \pnorm{M_{s-1}^{(t)}}{F}^2+ \E^{1/2} \pnorm{N_{s-1}^{(s)}}{F}^2\big)\nonumber\\
	&\leq \frac{(K\Lambda)^{c t}}{\sqrt{n}}\log^c n\cdot (1+\pnorm{z^{(0)}}{\infty}).
	\end{align}

	\noindent (\textbf{Step 2-(b)}). We provide a bound for $R_r^{(s)}$ for $r \in [s-1]$. Note that
	\begin{align}\label{ineq:trace_pop_2}
	\abs{R_r^{(s)}}&\leq \frac{1}{n^2}\biggabs{\E \sum_{i,j \in [n]} \mathsf{D}_{ii}^{(t-s)} M_{r-1,ii}^{(t)}  \mathsf{D}_{jj}^{(t-r)} N^{(s)}_{s-1-r,jj} V_{ij}^2 }\nonumber\\
	&\quad + \frac{1}{n^2}\biggabs{\E \sum_{i\neq j \in [n]}  \mathsf{D}_{ii}^{(t-s)} M_{r-1,ij}^{(t)}  \mathsf{D}_{ii}^{(t-r)} N^{(s)}_{s-1-r,ij} V_{ij}^2 }\nonumber\\
	&\quad + \biggabs{\frac{1}{n^2}\E \sum_{i,j,\ell} \mathsf{D}_{ii}^{(t-s)} \big(M_{r-1}^{(t)} A - M_{r-2}^{(t)} \mathsf{B}_{t-r+1}\big)_{i\ell} \mathsf{F}_{t-r,\ell}^{(2)}(z_{\ell}^{(t-r)}) \partial_{ij} z^{(t-r)}_{\ell} N_{s-1-r,\ell j}^{(s)}V_{ij}^2 }\nonumber\\
	&\equiv \mathfrak{R}_{r;1}^{(s)}+\mathfrak{R}_{r;2}^{(s)}+\mathfrak{R}_{r;3}^{(s)}.
	\end{align}
	We shall handle the three terms $\mathfrak{R}_{r;1}^{(s)},\mathfrak{R}_{r;2}^{(s)},\mathfrak{R}_{r;3}^{(s)}$ separately.

	\noindent \textbf{\emph{Term $\mathfrak{R}_{r;3}^{(s)}$}}: By writing $P_r^{(t)}\equiv M_{r-1}^{(t)} A - M_{r-2}^{(t)} \mathsf{B}_{t-r+1}$ for notational simplicity, 
	\begin{align*}
	\mathfrak{R}_{r;3}^{(s)}&\leq \frac{(K\Lambda)^c}{n^2} \E \sum_{i,j,\ell} \abs{P^{(t)}_{r,i\ell}}\abs{\partial_{ij} z_{\ell}^{(t-r)}}\abs{N_{s-1-r,\ell j}^{(s)}}\\
	&\leq \frac{(K\Lambda)^c}{n^2} \E \bigg[\max_{i,j\in [n], \ell \neq i,j} \abs{ \partial_{ij} z_{\ell}^{(t-r)} } \cdot n \pnorm{P_r^{(t)}}{F} \pnorm{N_{s-1-r}^{(s)}}{F}\\
	&\qquad + \max_{i,j\in [n] } \pnorm{\partial_{ij} z^{(t-r)} }{\infty}\cdot  n\Big(\pnorm{P^{(t)}_r}{\op}\pnorm{N_{s-1-r}^{(s)}}{F}+\pnorm{P^{(t)}_r}{F}\pnorm{N_{s-1-r}^{(s)}}{\op}   \Big) \bigg].
	\end{align*}
	Using the moment estimates in (\ref{ineq:trace_pop_M_N_moment}) and Proposition \ref{prop:z_der_l2}, for $t\leq \log n/c$,
	\begin{align}
	\mathfrak{R}_{r;3}^{(s)}&\leq  \big(K\Lambda\log n\cdot (1+\pnorm{z^{(0)}}{\infty})\big)^{c t^2}\cdot n^{-1/2}.
	\end{align}
	\noindent \textbf{\emph{Term $\mathfrak{R}_{r;2}^{(s)}$}}: By a simple Cauchy-Schwarz inequality and the moment estimates in (\ref{ineq:trace_pop_M_N_moment}), for $t\leq \log n/c$,
	\begin{align}\label{ineq:trace_pop_3}
	\mathfrak{R}_{r;2}^{(s)}\leq (K\Lambda)^2 n^{-2} \cdot\E^{1/2}\pnorm{M_{r-1}^{(t)}}{F}^2\cdot\E^{1/2}\pnorm{N_{s-1-r}^{(s)}}{F}^2\leq \frac{(K\Lambda)^{c t}}{n}.
	\end{align}
	\noindent \textbf{\emph{Term $\mathfrak{R}_{r;1}^{(s)}$} with $r=1$}: Note that
	\begin{align*}
	\mathfrak{R}_{1;1}^{(s)}&\leq \Lambda \E^{1/2} \pnorm{M_0^{(t)}}{\op}^2\cdot  \frac{1}{n} \max_{\ell \in [n]}\E^{1/2} \tr^2\Big(\mathfrak{D}_{V_\ell^2} N_{s-2}^{(s)}  \mathsf{D}^{(t-1)}\Big).
	\end{align*}
	Clearly $\E^{1/2} \pnorm{M_0^{(t)}}{\op}^2\leq \Lambda$. Let $H(A)\equiv n^{-1} \tr\big(\mathfrak{D}_{V_\ell^2} N_{s-2}^{(s)}  \mathsf{D}^{(t-1)}\big)$. Then using (\ref{ineq:trace_pop_M_N_difference}), for any two symmetric $A_1,A_2\in \R^{n\times n}$,
	\begin{align*}
	&\abs{H(A_1)-H(A_2)}\leq n^{-1}\bigabs{ \tr\Big(\mathfrak{D}_{V_\ell^2} \big(N_{s-2}^{(s)}(A_1)-N_{s-2}^{(s)}(A_2)\big)  \mathsf{D}^{(t-1)}(A_1)\Big) }\\
	&\qquad + n^{-1}\bigabs{ \tr\Big(\mathfrak{D}_{V_\ell^2} N_{s-2}^{(s)}(A_2)  \big(\mathsf{D}^{(t-1)}(A_1)-\mathsf{D}^{(t-1)}(A_2)\big)\Big)  }\\
	&\leq (K\Lambda)^c\cdot n^{-1/2}\pnorm{N^{(s)}_{s-2}(A_1)-N^{(s)}_{s-2}(A_2)}{F}\\
	&\qquad +(K\Lambda)^c\cdot \big(1+\pnorm{N^{(s)}_{s-2}(A_1)}{\op}+\pnorm{N^{(s)}_{s-2}(A_2)}{\op}\big)\cdot n^{-1/2}\pnorm{z^{(t-1)}(A_1)-z^{(t-1)}(A_2)}{}\\
	&\leq \big(K\Lambda (1+\pnorm{z^{(0)}}{\infty}) (1+\pnorm{A_1}{\op}+\pnorm{A_2}{\op})\big)^{ct^2}\cdot \pnorm{A_1-A_2}{\op}.
	\end{align*}
	Moreover,
	\begin{align*}
	\abs{H(A)}\leq (K\Lambda)^2 \pnorm{N^{(s)}_{s-2}}{\op}\leq \big(K\Lambda (1+\pnorm{A}{\op})\big)^{ct^2}.
	\end{align*}
	Now applying the Gaussian concentration Lemma \ref{lem:gaussian_conc}, we have
	\begin{align*}
	\var(H(A))\leq n^{-1}\big(K\Lambda \log n\cdot (1+\pnorm{z^{(0)}}{\infty}\big)^{ct^2}.
	\end{align*}
	Combining the above two displays, we have for $3\leq s\leq t\leq \log n/c$,
	\begin{align}\label{ineq:trace_pop_4}
	\mathfrak{R}_{1;1}^{(s)}&\leq \frac{\big(K\Lambda \log n\cdot (1+\pnorm{z^{(0)}}{\infty})\big)^{c t^2} }{\sqrt{n}} + (K\Lambda)^{ct}\max_{\ell \in [n]} \frac{1}{n}\bigabs{\E \tr\big(\mathfrak{D}_{V_\ell^2} N_{s-2}^{(s)} \mathsf{D}^{(t-1)}\big)}.
	\end{align}
	\noindent \textbf{\emph{Term $\mathfrak{R}_{r;1}^{(s)}$} with $r \in [2:s-1]$}: Note that
	\begin{align*}
	\mathfrak{R}_{r;1}^{(s)}& \leq \frac{1}{n^2} \biggabs{\E \sum_{j \in [n]}   \mathsf{D}_{jj}^{(t-r)} N^{(s)}_{s-1-r,jj}\cdot \tr\big(\mathfrak{D}_{V_j^2} M_{r-1}^{(t)}  \mathsf{D}^{(t-s)}\big) }\\
	&\leq \Lambda\E^{1/2}\pnorm{N^{(s)}_{s-1-r} }{\op}^2\cdot \frac{1}{n} \max_{\ell \in [n]}\E^{1/2} \tr^2\big(\mathfrak{D}_{V_\ell^2} M_{r-1}^{(t)}  \mathsf{D}^{(t-s)}\big).
	\end{align*}
	From here, we may use arguments similarly as above to conclude that, for $3\leq s\leq t\leq \log n/c$,
	\begin{align}\label{ineq:trace_pop_5}
	\mathfrak{R}_{r;1}^{(s)}&\leq \frac{\big(K\Lambda \log n(1+\pnorm{z^{(0)}}{\infty})\big)^{c t^2}}{\sqrt{n}}+ (K\Lambda)^{ct} \max_{\ell \in [n]} \frac{1}{n}\bigabs{\E \tr\big(\mathfrak{D}_{V_\ell^2} M_{r-1}^{(t)} \mathsf{D}^{(t-s)}\big)   }.
	\end{align}
	Combining (\ref{ineq:trace_pop_2})-(\ref{ineq:trace_pop_5}), for $3\leq s\leq t\leq \log n/c$, 
	\begin{align*}
	\abs{R_r^{(s)}}&\leq \big(K\Lambda\log n\cdot (1+\pnorm{z^{(0)}}{\infty})\big)^{c t^2}\cdot n^{-1/2}\\
	&\qquad +  (K\Lambda)^{ct}\max_{\ell \in [n]}  \frac{1}{n}\bigabs{\E \tr\big(\mathfrak{D}_{V_\ell^2} N_{s-2}^{(s)} \mathsf{D}^{(t-1)}\big)}  \bm{1}_{r=1}\\
	&\qquad + (K\Lambda)^{ct}\max_{\ell \in [n]}  \frac{1}{n}\bigabs{\E \tr\big(\mathfrak{D}_{V_\ell^2} M_{r-1}^{(t)} \mathsf{D}^{(t-s)}\big) } \bm{1}_{r \in [2:s-1]}.
	\end{align*}
	\noindent (\textbf{Step 2-(c)}). In view of (\ref{ineq:trace_pop_1}), we then conclude that for $3\leq s\leq t\leq \log n/c$, 
	\begin{align*}
	\frac{1}{n}\bigabs{\E \tr(M_s^{(t)})}&\leq \big(K\Lambda\log n\cdot (1+\pnorm{z^{(0)}}{\infty})\big)^{c t^2}\cdot n^{-1/2} \\
	&\qquad +(K\Lambda)^{ct} \max_{d \in \{V_{[1:n]}^2,b_{[0:t]}\}}  \frac{1}{n}\bigabs{\E \tr\big(\mathfrak{D}_{d} N_{s-2}^{(s)} \mathsf{D}^{(t-1)}\big)}\\
	&\qquad + (K\Lambda)^{ct}\sum_{r \in [s-2]}\max_{d \in \{V_{[1:n]}^2,b_{[0:t]}\}}  \frac{1}{n}\bigabs{\E \tr\big(\mathfrak{D}_{d} M_{r}^{(t)} \mathsf{D}^{(t-s)}\big) }.
	\end{align*}
	Now repeating the above arguments for the term $r=s-2,s-3,\ldots, r\geq 3$ in the first summation, it follows that for $3\leq s\leq t\leq \log n/c$, 
	\begin{align*}
	&\frac{1}{n}\bigabs{\E \tr(M_s^{(t)})}\leq \big(K\Lambda\log n\cdot (1+\pnorm{z^{(0)}}{\infty})\big)^{c t^2}\cdot n^{-1/2} \\
	& +  (K\Lambda)^{ct}\sum_{r =1,2}\max_{d \in \{V_{[1:n]}^2,b_{[0:t]}\}} \frac{1}{n}\bigabs{\E \tr\big(\mathfrak{D}_{d} M_{r}^{(t)} \mathsf{D}^{(t-s)}\big) }\\
	&+ (K\Lambda)^{ct^2}\sum_{m\in [s-2]}\max_{ \substack{d_{[cm]} \in \{V_{[1:n]}^2,b_{[0:t]},1_n\},\\ \mathcal{S}\in [-1:t]^{cm} }} \frac{1}{n}\bigabs{\E \tr\big(\mathfrak{D}_{\otimes_{q \in [cm]}d_q} N_{s-m-1}^{(s-m+1)}  \mathsf{D}^{(\mathcal{S})}\big)}.
	\end{align*}
	The first summation term can be assimilated into the first term by similar arguments in Step 1. This proves the desired estimate (\ref{ineq:trace_pop_claim_1}) (the slightly more general form claimed there follows from minor modifications).
\end{proof}

Next, for $3\leq u\leq s\leq t$, we show that the normalized trace of $N_u^{(s)}$ can be reduced back to those of $\{M_\cdot^{(s)}\}$, the most complicated one being $M^{(t-s+u)}_{u-2}$. Again the formulation below allows for $u=1,2$.

\begin{lemma}\label{lem:trace_cal_N_to_M}
	Suppose (S1)-(S2) in Theorem \ref{thm:AMP_sym_loo} hold for some $K,\Lambda\geq 2$. There exists some universal constant $c_0>0$ such that for any fixed $d_0 \in \R^n$ and an index set $\mathcal{T}\subset [-1:t]$ allowing for multiplicity, and any $1\leq u\leq s\leq t\leq \log n/c_0$, 
	\begin{align}\label{ineq:trace_pop_claim_2}
	&\frac{1}{n}\bigabs{\E \tr\big(\mathfrak{D}_{d_0}N_u^{(s)}\mathsf{D}^{(\mathcal{T})}\big)} \leq \big(K\Lambda\log n\cdot (1+\pnorm{z^{(0)}}{\infty})\big)^{c_0 t^2}\cdot n^{-1/2}  \\
	& + (K\Lambda)^{c_0 t^2}\sum_{m \in [u-2]}\max_{ \substack{d_{[c_0 m]} \in \{V_{[1:n]}^2, b_{[0:t]},1_n\},\\ \mathcal{S}\in [-1:t]^{c_0 m}} } \frac{1}{n}\biggabs{\E \tr\Big(\mathfrak{D}_{d_0\circ \otimes_{q \in [c_0 m]} d_q} M^{(t-s+u-m+1)}_{u-m-1}  \mathsf{D}^{(\mathcal{T})} \mathsf{D}^{(\mathcal{S})}\Big)}.\nonumber
	\end{align}
	Here recall $\mathsf{D}^{(\mathcal{T})}\equiv \prod_{s \in \mathcal{T}}\mathsf{D}^{(s)}$ and $\mathsf{D}^{(-1)}\equiv I_{n}$.
\end{lemma}

The proof of the above lemma follows largely that of Lemma \ref{lem:trace_cal_M_to_N}; some details can be found in Section \ref{section:proof_technical_results_proof_main_thm}. Lemmas \ref{lem:trace_cal_M_to_N} and \ref{lem:trace_cal_N_to_M} taken together, the normalized traces of $\{M_\cdot^{(t)}\}$ may be reduced to those of $\{M_\cdot^{(s)}\}$ for $s\leq t-2$, and therefore eventually to those of $\{M_\cdot^{(1)}, M_\cdot^{(2)}\}$. We make this precise below.

\begin{proof}[Proof of Proposition \ref{prop:trace_M}]
	For any $t \in \N$ and $k_\ell,k_u \in \mathbb{Z}_{\geq 0}$, let
	\begin{align*}
	\mathcal{M}_t(k_\ell,k_u)\equiv \max_{1\leq r\leq s \leq t}\max_{d_{[k_\ell]} \in \{V_{[1:n]}^2,b_{[0:t]},1_n\} }  \max_{w_{[k_u]} \in [-1:t]^{k_u} } \frac{1}{n}\bigabs{\E \tr\big(\mathfrak{D}_{\otimes_{q \in [k_\ell]} d_{q}} M_{r}^{(s)} \mathsf{D}^{(w_{[k_u]})}\big)}. 
	\end{align*}
	Here $\mathsf{D}^{(-1)}=I_n$, and we write $\mathsf{D}^{(w_{[k_u]})}\equiv \prod_{q \in [k_u]}\mathsf{D}^{(w_q)}$ for notational simplicity. Using (\ref{ineq:trace_pop_claim_1}) in Lemma \ref{lem:trace_cal_M_to_N} and (\ref{ineq:trace_pop_claim_2}) in Lemma \ref{lem:trace_cal_N_to_M}, with some calculations, there exists some universal constant $c_0>0$ such that
	\begin{align*}
	\mathcal{M}_t(k_\ell,k_u)&\leq \big(K\Lambda\log n\cdot (1+\pnorm{z^{(0)}}{\infty})\big)^{c_0 t^2}\cdot n^{-1/2}+(K\Lambda)^{c_0 t^2}\cdot \mathcal{M}_{t-2}(k_\ell+c_0 t,k_u+c_0 t).
	\end{align*}
	Iterating the bound with the following initial condition (which can be proved in a similar way as in Step 1 in Lemma \ref{lem:trace_cal_M_to_N}),
	\begin{align*}
	\mathcal{M}_1(k_\ell,k_u)\vee \mathcal{M}_2(k_\ell,k_u)\leq (K\Lambda)^{c_1(k_\ell+k_u)}\big(\log n\cdot (1+\pnorm{z^{(0)}}{\infty})\big)^{c_1}\cdot n^{-1/2},
	\end{align*}
	we conclude that 
	\begin{align*}
	\mathcal{M}_t(1,0)\leq \big(K\Lambda\log n\cdot (1+\pnorm{z^{(0)}}{\infty})\big)^{c_2 t^3}\cdot n^{-1/2},
	\end{align*}
	as desired.
\end{proof}

\subsection{Proof of Theorem \ref{thm:AMP_sym_loo}}

We shall now relate the estimate for the trace calculations in Proposition \ref{prop:trace_M} to the original sequence $\{\overline{M}_{\cdot,[-k]}^{(t)}\}$ defined in (\ref{def:M_bar_recursion}). To this end, consider an auxiliary sequence of matrices $M_{0,[-k]}^{(t)}, M_{1,[-k]}^{(t)},\ldots, M_{t,[-k]}^{(t)} \in \R^{n\times n}$, defined recursively as follows: $M_{-1,[-k]}^{(t)}=0_{n\times n}$, $M_{0,[-k]}^{(t)} =  \mathsf{D}_{[-k]}^{(t)}$, and for $s \in [t]$,
\begin{align}\label{eqn:M_recursion_loo}
M_{s,[-k]}^{(t)}\equiv \big(M_{s-1,[-k]}^{(t)} A_{[-k]} - M_{s-2,[-k]}^{(t)} \mathsf{B}_{t-s+1}\big) \mathsf{D}_{[-k]}^{(t-s)}. 
\end{align}
The sequence $\{M_{\cdot,[-k]}^{(t)}\}$ is different from $\{M_{\cdot}^{(t)}(V_{[-k]})\}$ in (\ref{eqn:M_recursion}) due to the use of $\mathsf{B}_{\cdot}$ rather than $\mathsf{B}_{\cdot,[-k]}\equiv \mathrm{diag}\big(\big\{n^{-1}\sum_{\ell \in [n]} V_{[-k],j\ell}^2 \E \mathsf{F}_{\cdot,\ell}'(z_{[-k],\ell}^{(\cdot)})\big\}_{j\in [n]}\big)$. However, their trace differences are negligible, as shown in the proof of the following result.

\begin{corollary}\label{cor:trace_M_loo}
	Suppose (S1)-(S2) in Theorem \ref{thm:AMP_sym_loo} hold for some $K,\Lambda\geq 2$. Then the trace estimate in Proposition \ref{prop:trace_M} also holds for $\{M_{\cdot,[-k]}^{(t)}\}$ defined in (\ref{eqn:M_recursion_loo}), i.e., there exists some universal constant $c_0>0$ such that for $t\leq \log n/c_0$,
	\begin{align*}
	\max_{d_0 \in \{V_{[1:n]}^2,1_n\} }\max_{s \in [t]} \frac{1}{n}\bigabs{\E \big[\tr(\mathfrak{D}_{d_0}M_{s,[-k]}^{(t)})|z^{(0)}\big] } \leq \big(K\Lambda\log n\cdot (1+\pnorm{z^{(0)}}{\infty})\big)^{c_0 t^3}\cdot n^{-1/2}.
	\end{align*}
\end{corollary}

More importantly, the sequence $\{M_{\cdot,[-k]}^{(t)}\}$ can be related to $\{\overline{M}_{\cdot,[-k]}^{(t)}\}$ in (\ref{def:M_bar_recursion}) (defined via a different multiplicative matrix factor $\mathsf{D}_{[-k]}^{(t-s)}$) in the quadratic form of interest in the leave-one-out decomposition in Proposition \ref{prop:AMP_loo_decom}. 

\begin{lemma}\label{lem:M_bar_loo_quad_diff}
	Suppose (S1)-(S2) in Theorem \ref{thm:AMP_sym_loo} hold for some $K,\Lambda\geq 2$. Then for any $D>0$, there exist some universal constant $c_0>0$ and another constant $c_1=c_1(D)>0$ such that for $0\leq s\leq t\leq \log n/c_0$, with $\Prob(\cdot|z^{(0)})$-probability at least $1-c_1 n^{-D}$,
	\begin{align*}
	\bigabs{\iprod{A_k}{\overline{M}_{s,[-k]}^{(t)}  A_k }-\iprod{A_k}{{M}_{s,[-k]}^{(t)}  A_k }}\leq c_1 (K\Lambda)^{c_0 t}\log^{c_0} n\cdot (1+\pnorm{z^{(0)}}{\infty})\cdot n^{-1/2}.
	\end{align*}
\end{lemma}

The proofs of both Corollary \ref{cor:trace_M_loo} and Lemma \ref{lem:M_bar_loo_quad_diff} above can be found in Section \ref{section:proof_technical_results_proof_main_thm}. We are now finally in a position to prove Theorem \ref{thm:AMP_sym_loo}.

\begin{proof}[Proof of Theorem \ref{thm:AMP_sym_loo}]
	By Propositions \ref{prop:AMP_loo_decom}, \ref{prop:AMP_iterate_infty} and Lemma \ref{lem:M_bar_loo_quad_diff}, for $t\leq \log n/c$, with probability at least $1-cn^{-D}$, 
	\begin{align*}
	\bigabs{z^{(t+1)}_k-\bigiprod{A_k}{\mathsf{F}_t (z_{[-k]}^{(t)})}}&\leq (K\Lambda)^{ct}\log n\cdot (1+\pnorm{z^{(0)}}{\infty})\\
	&\quad \times \Big(\max_{s \in [t]} \bigabs{ \iprod{A_k}{M_{t-s,[-k]}^{(t)}  A_k }- b_{t,k}\bm{1}_{s=t}   }+n^{-1/2}\Big).
	\end{align*}
	For $s\leq t-1$, by Hanson-Wright inequality and Corollary \ref{cor:trace_M_loo}, for $t\leq \log n/c$, with probability at least $1-cn^{-D}$, 
	\begin{align*}
	\max_{s \in [t-1]} \bigabs{\iprod{A_k}{M_{t-s,[-k]}^{(t)}  A_k }   }\leq \big(K\Lambda\log n\cdot (1+\pnorm{z^{(0)}}{\infty})\big)^{c t^3}\cdot n^{-1/2}.
	\end{align*}
	For $s=t$, by Hanson-Wright inequality and using Proposition \ref{prop:AMP_iterate_infty} again, for $t\leq \log n/c$, with probability at least $1-cn^{-D}$, 
	\begin{align*}
	\bigabs{ \iprod{A_k}{M_{0,[-k]}^{(t)}  A_k }- b_{t,k} }&\leq \frac{1}{n}\sum_{\ell \in [n]} V_{k\ell}^2 \bigabs{\E \mathsf{F}'_{t,\ell}(z^{(t)}_{[-k],\ell})- \E \mathsf{F}'_{t,\ell}(z^{(t)}_\ell)}+ \frac{(K\Lambda)^c \log n}{\sqrt{n}}\\
	&\leq (K\Lambda)^{ct}\log^c n\cdot (1+\pnorm{z^{(0)}}{\infty})\cdot n^{-1/2}. 
	\end{align*}
	Combining the above three displays proves that for $t\leq \log n/c$, with probability at least $1-cn^{-D}$, 
	\begin{align}\label{ineq:AMP_sym_proof_loo_represent}
	\bigabs{z^{(t+1)}_k-\bigiprod{A_k}{\mathsf{F}_t (z_{[-k]}^{(t)})} }&\leq \big(K\Lambda\log n\cdot (1+\pnorm{z^{(0)}}{\infty})\big)^{c t^3}\cdot n^{-1/2}.
	\end{align}
	The proof is complete. 
\end{proof}

\section{Proofs of technical results in Section \ref{section:proof_AMP_loo}}\label{section:proof_technical_results_proof_main_thm}

\subsection{Proof of Proposition \ref{prop:AMP_iterate_infty}}

Using the AMP recursion (\ref{def:AMP}), we have
\begin{align*}
\pnorm{\Delta z_{[-k]}^{(t+1)}}{}&\leq \pnorm{A \mathsf{F}_{t}(z^{(t)})- A_{[-k]} \mathsf{F}_{t}(z^{(t)}_{[-k]}) }{}+\pnorm{b_t}{\infty} \pnorm{\Delta z_{[-k]}^{(t-1)}}{}\\
&\leq (K\Lambda)^2 (1+\pnorm{A}{\op})\big(\pnorm{\Delta z_{[-k]}^{(t)}}{}\vee \pnorm{\Delta z_{[-k]}^{(t-1)}}{}\big)+\pnorm{\Delta A_{[-k]} \mathsf{F}_{t}(z^{(t)}_{[-k]}) }{},
\end{align*}
where
\begin{align*}
\pnorm{\Delta A_{[-k]} \mathsf{F}_{t}(z^{(t)}_{[-k]}) }{}&\leq \abs{A_k^\top \mathsf{F}_{t}(z^{(t)}_{[-k]})}+ 2\Lambda\pnorm{A_k}{}\big(1+\abs{z_{k,[-k]}^{(t)}}\big).
\end{align*}
On the other hand, using (\ref{def:AMP_loo}),
\begin{itemize}
	\item $\abs{z_{k,[-k]}^{(t)}}\leq (K\Lambda)^2(1+\abs{z_{k,[-k]}^{(t-2)}})\leq \cdots\leq (K\Lambda)^{ct}(1+\abs{z_k^{(0)}})$,
	\item $\pnorm{z_{[-k]}^{(t)}}{}\leq (K\Lambda)^2 (1+\pnorm{A_{[-k]}}{\op})\big(\sqrt{n}+\pnorm{z_{[-k]}^{(t-1)}}{}\vee \pnorm{z_{[-k]}^{(t-2)}}{}\big)\leq\cdots \leq \big(K\Lambda(1+\pnorm{A_{[-k]}}{\op})\big)^{ct}\big(\sqrt{n}+\pnorm{z^{(0)}}{}\big)$. 
\end{itemize}
Using the independence of $A_k$ and $\mathsf{F}_{t}(z^{(t)}_{[-k]})$ and the standard subgaussian tail estimate for $\pnorm{A}{\op},\pnorm{A_{[-k]}}{\op}$ (cf. Lemma \ref{lem:A_op_subgaussian_tail}), for $x\geq 1$, with probability at least $1-c e^{-(x^2\wedge n)/c}$, both $\sqrt{n}\abs{A_k^\top \mathsf{F}_t(z_{[-k]}^{(t)} )}\leq cKx\cdot \pnorm{\mathsf{F}_t(z_{[-k]}^{(t)})}{}$ and $\pnorm{A}{\op}\vee \pnorm{A_{[-k]}}{\op}\leq cK$ hold. So with the same probability, 
\begin{align*}
\pnorm{\Delta z_{[-k]}^{(t+1)}}{}&\leq c(K\Lambda)^2 \big(\pnorm{\Delta z_{[-k]}^{(t)}}{}\vee \pnorm{\Delta z_{[-k]}^{(t-1)}}{}\big)\\
&\qquad +(K\Lambda)^{ct}(1+\abs{z_k^{(0)}})+(K\Lambda)^{ct} x \cdot \big(1+n^{-1/2}\pnorm{z^{(0)}}{}\big)\\
&\leq c(K\Lambda)^2 \big(\pnorm{\Delta z_{[-k]}^{(t)}}{}\vee \pnorm{\Delta z_{[-k]}^{(t-1)}}{}\big)+ (K\Lambda)^{ct}x\cdot (1+\pnorm{z^{(0)}}{\infty}).
\end{align*}
Iterating the bound, with probability at least $1-ct e^{-(x^2\wedge n)/c}$,
\begin{align*}
\pnorm{\Delta z_{[-k]}^{(t+1)}}{}\leq (K\Lambda)^{ct} x\cdot (1+\pnorm{z^{(0)}}{\infty}),
\end{align*}
and therefore with the same probability,
\begin{align*}
\abs{z^{(t+1)}_k}\leq \pnorm{\Delta z_{[-k]}^{(t+1)}}{}+\abs{ z^{(t+1)}_{k,[-k]}  }\leq (K\Lambda)^{ct} x\cdot (1+\pnorm{z^{(0)}}{\infty}).
\end{align*}
Consequently, with probability at least $1-c\cdot t n e^{-(x^2\wedge n)/c}$,
\begin{align*}
\pnorm{z^{(t+1)}}{\infty}\leq (K\Lambda)^{ct} x\cdot (1+\pnorm{z^{(0)}}{\infty}).
\end{align*}
The moment estimate for $\pnorm{z^{(t+1)}}{\infty}$ follows by combining the above display and the simple apriori estimate $\pnorm{z^{(t+1)}}{\infty}\leq \pnorm{z^{(t+1)}}{}\leq  \big(\Lambda(1+\pnorm{A}{\op})\big)^{ct}\big(\sqrt{n}+\pnorm{z^{(0)}}{}\big)$. \qed

\subsection{Proof of Proposition \ref{prop:z_der_l2}}

\noindent (1). By the definition of the AMP recursion (\ref{def:AMP}), we have
\begin{align*}
\partial_{ij} z^{(t+1)}&=\Delta_{ij} \mathsf{F}_t(z^{(t)})+A\big(\mathsf{F}_t'(z^{(t)})\circ \partial_{ij} z^{(t)}\big)-\mathsf{B}_t \big(\mathsf{F}_{t-1}'(z^{(t-1)})\circ \partial_{ij} z^{(t-1)}\big).
\end{align*}
Consequently,
\begin{align*}
\pnorm{\partial_{ij} z^{(t+1)} }{}&\leq \abs{\mathsf{F}_{t,i}(z^{(t)}_i)}+\abs{\mathsf{F}_{t,j}(z^{(t)}_j)}+ \Lambda \pnorm{A}{\op} \pnorm{\partial_{ij} z^{(t)}}{}+(K\Lambda)^2 \pnorm{\partial_{ij} z^{(t-1)}}{}\\
&\leq (K\Lambda)^2(1+\pnorm{A}{\op})\big(\pnorm{\partial_{ij} z^{(t)}}{}\vee \pnorm{\partial_{ij} z^{(t-1)}}{} \big)+ 2\Lambda(1+\pnorm{z^{(t)}}{\infty}).
\end{align*}
Iterating the bound, for any $p\geq 1$, there exists some $C_p>0$ such that 
\begin{align*}
\E^{1/p}\pnorm{\partial_{ij} z^{(t+1)} }{}^p\leq C_p \E^{1/2p}\big(K \Lambda(1+\pnorm{A}{\op}) \big)^{C_p t} \max_{s \in [0:t]} \E^{1/2p}(1+\pnorm{z^{(s)}}{\infty})^{2p}.
\end{align*}
The claim follows from Proposition \ref{prop:AMP_iterate_infty}. 

\noindent (2).
For any $u \in [t]$, we let  $\mathfrak{A}^{(u)}_1 \equiv A$, $\mathfrak{A}^{(u)}_2 \equiv \mathsf{B}_{u+1}$.
Note that for $\ell\neq i,j$, $\partial_{ij} z^{(t+1)}_\ell$ consists of at most $(K\Lambda)^{ct}$ terms of the form
\begin{align}\label{eq:eiA_Aej}
\biggiprod{e_i}{\prod_{k \in [s]}\big(\mathfrak{A}^{(t_{k})}_{\tau_{t_k}} \mathsf{D}^{(t_{k})}\big) A e_j},
\end{align}
where $s \in [0:t]$, $1\leq t_s<t_{s-1}<\cdots<t_1\leq t$, $\tau_{t_k} \in \{1,2\}$ with $\tau_{t_1}=1$, and the product $\prod_{k \in [s]}$ is taken in the order from $1$ to $s$. Consequently, it suffices to derive the moment estimate for the above terms.  

Suppose $s\geq 1$ (the case $s=0$ is trivial). For $r \in [s]$, let 
\begin{align*}
\mathcal{A}_r&\equiv \max_{i\neq j \in [n]}\max_{w \in [r]} \max_{\substack{1\leq t_w<t_{w-1}<\cdots<t_1\leq t,\\ \tau_{t_k} \in \{1,2\},\tau_{t_1}=1 }} \biggabs{\biggiprod{e_i}{\prod_{k \in [w]}\big(\mathfrak{A}^{(t_{k})}_{\tau_{t_k}} \mathsf{D}^{(t_{k})}\big) A e_j}}.
\end{align*}
Let us fix $i\neq j \in [n]$, and let $\mathcal{P}\equiv \{i,j\}$. We first replace $\mathsf{D}^{(\cdot)}$ by $\mathsf{D}^{(\cdot)}_{[-\mathcal{P}]}$, where $\mathsf{D}^{(\cdot)}_{[-\mathcal{P}]}\equiv \mathfrak{D}_{\mathsf{F}'_{\cdot}(z^{(\cdot)}_{[-\mathcal{P}]})}$. Let $\Delta \mathsf{D}^{(\cdot)}_{[-\mathcal{P}]}=\mathsf{D}^{(\cdot)}-\mathsf{D}^{(\cdot)}_{[-\mathcal{P}]}$. Then for any sequence $1\leq t_s<t_{s-1}<\cdots<t_1\leq t$ and $\tau_{t_k} \in \{1,2\}$, 
\begin{align*}
&\bigabs{ \bigiprod{e_i}{A \mathsf{D}^{(t_1)}\mathfrak{A}^{(t_2)}_{\tau_{t_2}} \mathsf{D}^{(t_2)}\cdots \mathfrak{A}^{(t_s)}_{\tau_{t_s}} \mathsf{D}^{(t_{s})} A e_j}-\bigiprod{e_i}{A \mathsf{D}^{(t_1)}_{[-\mathcal{P}]}\mathfrak{A}^{(t_2)}_{\tau_{t_2}} \mathsf{D}^{(t_2)}\cdots \mathfrak{A}^{(t_s)}_{\tau_{t_s}} \mathsf{D}^{(t_{s})} A e_j}    }\\
&= \biggabs{\sum_{k \in [n]} A_{ik} \Delta \mathsf{D}^{(t_1)}_{[-\mathcal{P}],kk} \big(\mathfrak{A}^{(t_2)}_{\tau_{t_2}} \mathsf{D}^{(t_2)}\cdots \mathfrak{A}^{(t_s)}_{\tau_{t_s}} \mathsf{D}^{(t_{s})} A \big)_{kj}}\\
&\leq \pnorm{A_i}{}\cdot \Lambda \pnorm{z^{(t_1)}-z^{(t_1)}_{[-\mathcal{P}]}}{}\cdot  (K\Lambda)^{ct}\cdot  \mathcal{A}_{s-1}+\abs{A_{ij}} (K\Lambda \pnorm{A}{\op})^{ct}.
\end{align*}
Repeating the procedure of replacing one $\mathsf{D}^{(\cdot)}$ factor by a corresponding $\mathsf{D}_{[-\mathcal{P}]}^{(\cdot)}$ factor till the end, with $\pnorm{A}{\infty}\equiv \max_{i,j \in [n]} \abs{A_{ij}}$, we have 
\begin{align}\label{ineq:z_der_deloc_1}
&\bigabs{ \bigiprod{e_i}{A \mathsf{D}^{(t_1)}\mathfrak{A}^{(t_2)}_{\tau_{t_2}} \mathsf{D}^{(t_2)}\cdots \mathfrak{A}^{(t_s)}_{\tau_{t_s}} \mathsf{D}^{(t_{s})} A e_j}-\bigiprod{e_i}{A \mathsf{D}^{(t_1)}_{[-\mathcal{P}]}\mathfrak{A}^{(t_2)}_{\tau_{t_2}} \mathsf{D}^{(t_2)}_{[-\mathcal{P}]}\cdots \mathfrak{A}^{(t_s)}_{\tau_{t_s}} \mathsf{D}^{(t_{s})}_{[-\mathcal{P}]} A e_j}    }\nonumber\\
&\leq \big(K\Lambda (1+\pnorm{A}{\op})\big)^{c t}\cdot \Big( \max_{s \in [t]}\pnorm{z^{(s)}-z^{(s)}_{[-\mathcal{P}]}}{}\cdot  \mathcal{A}_{s-1}+\pnorm{A}{\infty}\Big).
\end{align}
Next we replace $\mathfrak{A}^{(u)}_{\tau_u}$ with $\mathfrak{A}^{(u)}_{\tau_u,[-\mathcal{P}]}$, where  $A_{[-\mathcal{P}]}\equiv (A_{k\ell}\bm{1}_{k,\ell \notin \mathcal{P}})$ and $\mathsf{B}_{\cdot,[-\mathcal{P}]}\equiv \mathsf{B}_\cdot$. Then it is easy to see
\begin{align*}
\Delta \mathfrak{A}^{(u)}_{\tau_u,[-\mathcal{P}]}&\equiv \mathfrak{A}^{(u)}_{\tau_u}-\mathfrak{A}^{(u)}_{\tau_u,[-\mathcal{P}]}= \bm{1}_{\tau_u=1}\sum_{(a,b)} \big(e_{a}e_a^\top A+Ae_ae_a^\top-A_{ab} e_ae_b^\top-A_{aa} e_ae_a^\top \big),
\end{align*}
where the summation is running over $(a,b)=\{(i,j),(j,i)\}$. Then	
\begin{align*}
&\bigabs{ \bigiprod{e_i}{A \mathsf{D}^{(t_1)}_{[-\mathcal{P}]}\mathfrak{A}^{(t_2)}_{\tau_{t_2}} \mathsf{D}^{(t_2)}_{[-\mathcal{P}]}\cdots \mathfrak{A}^{(t_s)}_{\tau_{t_s}} \mathsf{D}^{(t_{s})}_{[-\mathcal{P}]} A e_j}-\bigiprod{e_i}{A \mathsf{D}^{(t_1)}_{[-\mathcal{P}]}\mathfrak{A}^{(t_2)}_{\tau_{t_2},[-\mathcal{P}] } \mathsf{D}^{(t_2)}_{[-\mathcal{P}]}\cdots \mathfrak{A}^{(t_s)}_{\tau_{t_s}} \mathsf{D}^{(t_{s})}_{[-\mathcal{P}]} A e_j}}\\
&\leq  \big(K\Lambda(1+\pnorm{A}{\op})\big)^{ct}\cdot \bigg(\pnorm{A}{\infty}+\max_{w \in [s-1]} \max_{\substack{1\leq t_w<t_{w-1}<\cdots<t_1\leq t,\\ \tau_{t_k} \in \{1,2\},\tau_{t_1}=1 }} \biggabs{\biggiprod{e_i}{\prod_{k \in [w]}\big(\mathfrak{A}^{(t_{k})}_{\tau_{t_k}} \mathsf{D}^{(t_{k})}_{[-\mathcal{P}]}\big) A e_j}}\bigg)\\
&\leq \big(K\Lambda(1+\pnorm{A}{\op})\big)^{ct}\cdot \Big[ \big(1+\max_{s \in [t]}\pnorm{z^{(s)}-z^{(s)}_{[-\mathcal{P}]}}{}\big)\cdot  \mathcal{A}_{s-1}+\pnorm{A}{\infty}\Big].
\end{align*}
Here in the last inequality we used (\ref{ineq:z_der_deloc_1}) to reduce the estimate to $\mathcal{A}_{s-1}$. Repeating the procedure of replacing one $\mathfrak{A}^{(\cdot)}_{\cdot}$ factor by a corresponding $\mathfrak{A}^{(\cdot)}_{\cdot,[-\mathcal{P}]}$ factor till the end, we have 
\begin{align}\label{ineq:z_der_deloc_2}
&\bigabs{ \bigiprod{e_i}{A \mathsf{D}^{(t_1)}_{[-\mathcal{P}]}\mathfrak{A}^{(t_2)}_{\tau_{t_2}} \mathsf{D}^{(t_2)}_{[-\mathcal{P}]}\cdots \mathfrak{A}^{(t_s)}_{\tau_{t_s}} \mathsf{D}^{(t_{s})}_{[-\mathcal{P}]} A e_j} \nonumber\\
	&-\bigiprod{e_i}{A \mathsf{D}^{(t_1)}_{[-\mathcal{P}]}\mathfrak{A}^{(t_2)}_{\tau_{t_2},[-\mathcal{P}]} \mathsf{D}^{(t_2)}_{[-\mathcal{P}]}\cdots \mathfrak{A}^{(t_s)}_{\tau_{t_s},[-\mathcal{P}]} \mathsf{D}^{(t_{s})}_{[-\mathcal{P}]} A e_j} }\nonumber\\
&\leq  \big(K\Lambda(1+\pnorm{A}{\op})\big)^{c t}\cdot \Big[ \big(1+\max_{s \in [t]}\pnorm{z^{(s)}-z^{(s)}_{[-\mathcal{P}]}}{}\big)\cdot  \mathcal{A}_{s-1}+\pnorm{A}{\infty}\Big].
\end{align}	
Combining (\ref{ineq:z_der_deloc_1})-(\ref{ineq:z_der_deloc_2}), Proposition \ref{prop:AMP_iterate_infty}, and standard concentration estimates, for $t\leq n$, uniformly in $i\neq j \in [n]$, with probability at least $1-c n^{-D}$, 
\begin{align}\label{ineq:z_der_deloc_3}
&\bigabs{ \bigiprod{e_i}{A \mathsf{D}^{(t_1)}\mathfrak{A}^{(t_2)}_{\tau_{t_2}} \mathsf{D}^{(t_2)}\cdots \mathfrak{A}^{(t_s)}_{\tau_{t_s}} \mathsf{D}^{(t_{s})} A e_j} \nonumber\\
	&-\bigiprod{e_i}{A \mathsf{D}^{(t_1)}_{[-\mathcal{P}]}\mathfrak{A}^{(t_2)}_{\tau_{t_2},[-\mathcal{P}]} \mathsf{D}^{(t_2)}_{[-\mathcal{P}]}\cdots \mathfrak{A}^{(t_s)}_{\tau_{t_s},[-\mathcal{P}]} \mathsf{D}^{(t_{s})}_{[-\mathcal{P}]} A e_j} }\nonumber\\
&\leq (K\Lambda)^{ct}\sqrt{\log n}\cdot (1+\pnorm{z^{(0)}}{\infty})\cdot (\mathcal{A}_{s-1}+n^{-1/2}). 
\end{align}
On the other hand, as $A_i,A_j$ are independent of $\mathfrak{A}^{(\cdot)}_{\cdot,[-\mathcal{P}]}$ and $\mathsf{D}_{[-\mathcal{P}]}^{(\cdot)}$, using Hanson-Wright inequality, with probability at least $1-cn^{-D}$,
\begin{align}\label{ineq:z_der_deloc_4}
&\bigabs{\bigiprod{e_i}{A \mathsf{D}^{(t_1)}_{[-\mathcal{P}]}\mathfrak{A}^{(t_2)}_{\tau_{t_2},[-\mathcal{P}]} \mathsf{D}^{(t_2)}_{[-\mathcal{P}]}\cdots \mathfrak{A}^{(t_s)}_{\tau_{t_s},[-\mathcal{P}]} \mathsf{D}^{(t_{s})}_{[-\mathcal{P}]} A e_j} }\nonumber\\
& = \bigabs{A_i^\top \big(\mathsf{D}^{(t_1)}_{[-\mathcal{P}]}\mathfrak{A}^{(t_2)}_{\tau_{t_2},[-\mathcal{P}]} \mathsf{D}^{(t_2)}_{[-\mathcal{P}]}\cdots \mathfrak{A}^{(t_s)}_{\tau_{t_s},[-\mathcal{P}]} \mathsf{D}^{(t_{s})}_{[-\mathcal{P}]}\big)A_j} \leq (K\Lambda)^{ct} \log n\cdot n^{-1/2}. 
\end{align}
Combining (\ref{ineq:z_der_deloc_3}) and (\ref{ineq:z_der_deloc_4}), for $t\leq n$, uniformly in $i\neq j \in [n]$, with probability at least $1-c n^{-D}$,
\begin{align*}
\bigabs{\bigiprod{e_i}{A \mathsf{D}^{(t_1)}\mathfrak{A}^{(t_2)}_{\tau_{t_2}} \mathsf{D}^{(t_2)}\cdots \mathfrak{A}^{(t_s)}_{\tau_{t_s}} \mathsf{D}^{(t_{s})} A e_j}}\leq (K\Lambda)^{ct}\log n\cdot (1+\pnorm{z^{(0)}}{\infty})\cdot (\mathcal{A}_{s-1}+n^{-1/2}).
\end{align*}
Now by a union bound via cardinality calculations, we conclude that for $t\leq \log n/c$, with probability at least $1-c n^{-D}$,
\begin{align*}
\mathcal{A}_s\leq (K\Lambda)^{ct}\log n\cdot (1+\pnorm{z^{(0)}}{\infty})\cdot (\mathcal{A}_{s-1}+n^{-1/2}).
\end{align*}
Iterating the above bound shows that for $t\leq \log n/c$, with probability at least $1-c n^{-D}$,
\begin{align*}
\mathcal{A}_s\leq \big((K\Lambda)^{t}\log n\cdot (1+\pnorm{z^{(0)}}{\infty})\big)^{ct}\cdot n^{-1/2}.
\end{align*}
The moment estimate for $\E^{1/p}[ |(\ref{eq:eiA_Aej})|^p ]$ follows from the above high probability control, and the simple apriori estimate $\abs{(\ref{eq:eiA_Aej})}\leq (K\Lambda \pnorm{A}{\op})^{ct}$. \qed

\subsection{Proof of Lemma \ref{lem:M_N_moment}}

\noindent (1). In view of the recursions in (\ref{eqn:M_recursion}) and (\ref{eqn:N_recursion}), we have an easy estimate 
\begin{align*}
n^{-1/2}\big\{\pnorm{M^{(t)}_\cdot}{F}\vee \pnorm{N^{(s)}_\cdot}{F} \big\} + \big( \pnorm{M^{(t)}_\cdot}{\op}\vee \pnorm{N^{(s)}_\cdot}{\op}\big) \leq \big(K\Lambda (1+\pnorm{A}{\op})\big)^{ct}.
\end{align*}
The first claim follows by using the subgaussian tail of $\pnorm{A}{\op}$, cf. Lemma \ref{lem:A_op_subgaussian_tail}. 

\noindent (2). We shall only prove the case for $M^{(t)}_s$; the other cases follow from minor modifications. Using the recursion (\ref{eqn:M_recursion}), we have
\begin{align*}
&\pnorm{ M^{(t)}_s(A_1)-M^{(t)}_s(A_2)}{F}\leq \bigpnorm{M^{(t)}_{s-1} (A_1) A_1 \mathsf{D}^{(t-s)}(A_1)-M^{(t)}_{s-1} (A_2) A_2 \mathsf{D}^{(t-s)}(A_2)  }{F}\\
&\qquad + \bigpnorm{M^{(t)}_{s-2} (A_1) \mathsf{B}_{t-s+1} \mathsf{D}^{(t-s)}(A_1)-M^{(t)}_{s-2} (A_2) \mathsf{B}_{t-s+1} \mathsf{D}^{(t-s)}(A_2)  }{F}\equiv (I)+(II).
\end{align*}
Let $\Gamma(A_1,A_2;c)\equiv \big(K\Lambda (1+\pnorm{z^{(0)}}{\infty})(1+\pnorm{A_1}{\op}+ \pnorm{A_2}{\op})\big)^c$. For $(I)$, we have
\begin{align*}
(I)&\leq \pnorm{M^{(t)}_{s-1} (A_1) -M^{(t)}_{s-1} (A_2)  }{F}\cdot \pnorm{A_1 \mathsf{D}^{(t-s)}(A_1) }{\op}\\
&\qquad + \pnorm{M^{(t)}_{s-1}(A_2)}{\op}\cdot \pnorm{ A_1 \mathsf{D}^{(t-s)}(A_1)-A_2 \mathsf{D}^{(t-s)}(A_2) }{F}\\
&\leq \Gamma(A_1,A_2;ct) \cdot \Big( \pnorm{M^{(t)}_{s-1} (A_1) -M^{(t)}_{s-1} (A_2)  }{F} \\
&\qquad \qquad + \pnorm{A_1-A_2}{F}+ \pnorm{\mathsf{D}^{(t-s)}(A_1)-\mathsf{D}^{(t-s)}(A_2) }{F} \Big)\\
&\leq \Gamma(A_1,A_2;c_1t) \cdot \Big( \pnorm{M^{(t)}_{s-1} (A_1) -M^{(t)}_{s-1} (A_2)  }{F} + \sqrt{n}\pnorm{A_1-A_2}{\op}\Big).
\end{align*}
Here in the last inequality we used the estimate $
\pnorm{\mathsf{D}^{(t-s)}(A_1)-\mathsf{D}^{(t-s)}(A_2) }{F}\leq \Lambda\cdot \pnorm{z^{(t-s)}(A_1)-z^{(t-s)}(A_2) }{}\leq \Gamma(A_1,A_2;ct)\cdot \sqrt{n} \pnorm{A_1-A_2}{\op}$.  Using the same argument, we may estimate $(II)$:
\begin{align*}
(II)&\leq \Gamma(A_1,A_2;c_2t) \cdot \Big( \pnorm{M^{(t)}_{s-2} (A_1) -M^{(t)}_{s-2} (A_2)  }{F} + \sqrt{n}\pnorm{A_1-A_2}{\op}\Big).
\end{align*}
Combining the estimates, there exists some universal $c_0>0$ such that
\begin{align*}
&\max_{u \in [0:s]}\pnorm{ M^{(t)}_u(A_1)-M^{(t)}_u(A_2)}{F}\\
&\leq \Gamma(A_1,A_2;c_0 t) \cdot \Big( \max_{u \in [0:s-1]}\pnorm{M^{(t)}_{u} (A_1) -M^{(t)}_{u} (A_2)  }{F} + \sqrt{n}\pnorm{A_1-A_2}{\op}\Big).
\end{align*}
Iterating the bound and using the initial condition $\pnorm{M^{(t)}_{0} (A_1) -M^{(t)}_{0} (A_2)  }{F}= \pnorm{\mathsf{D}^{(t)}(A_1)-\mathsf{D}^{(t)}(A_2) }{F}\leq \Gamma(A_1,A_2;c't)\cdot \sqrt{n} \pnorm{A_1-A_2}{\op}$ to conclude. \qed

\subsection{Proof of Lemma \ref{lem:trace_cal_N_to_M}}

The proof of (\ref{ineq:trace_pop_claim_2}) is largely similar to that in (\ref{ineq:trace_pop_claim_1}) in Lemma \ref{lem:trace_cal_M_to_N}. We only sketch the key reduction. Let for $v \in [u-1]$,
\begin{align*}
Q_{v}^{(u)}&\equiv \frac{1}{n^2} \biggiprod{\E M_{u-1-v}^{(t-s+u)}\Delta_{\cdot\cdot}  \mathsf{D}^{(t-s+v)} N_{v-1}^{(s)} }{V\circ V}\\
&\qquad + \frac{1}{n^2}\E\biggiprod{ A \mathfrak{D}_{\mathsf{F}^{(2)}_{t-s+v}(z^{(t-s+v)})}\mathfrak{D}_{\partial_{\cdot\cdot} z^{(t-s+v)}} N_{v-1}^{(s)} }{V\circ V} \nonumber\\
&\qquad- \frac{1}{n^2}\E\biggiprod{\mathsf{B}_{t-s+v}  \mathfrak{D}_{\mathsf{F}^{(2)}_{t-s+v-1}(z^{(t-s+v-1)})}\mathfrak{D}_{\partial_{\cdot\cdot} z^{(t-s+v-1)}} N_{v-2}^{(s)}}{V\circ V}.
\end{align*}
By Gaussian integration by parts, 
\begin{align*}
\frac{1}{n}\E \tr(N_u^{(s)})&=\frac{1}{n}\sum_{i,j \in [n]} \E \mathsf{D}^{(t-s+u)}_{ii} N^{(s)}_{u-1,ij} A_{ij}- \frac{1}{n}\E \tr\Big(\mathsf{B}_{t-s+u} \mathsf{D}^{(t-s+u-1)} N_{u-2}^{(s)}\Big)\\
& = \frac{1}{n^2}\E \bigiprod{\mathsf{D}^{(t-s+u)} \partial_{\cdot\cdot} N_{u-1}^{(s)} }{V\circ V}- \frac{1}{n}\E \tr\Big(\mathsf{B}_{t-s+u}  \mathsf{D}^{(t-s+u-1)} N_{u-2}^{(s)}\Big)\\
&\qquad + \frac{1}{n^2}\sum_{i,j \in [n] } \E  \mathsf{F}_{t-s+u,i}^{(2)}(z_i^{(t-s+u)}) \partial_{ij}z_i^{(t-s+u)} N_{u-1,ij}^{(s)} V_{ij}^2.
\end{align*}
Using (\ref{eqn:N_der_recursion}) with $u$ replaced by $u-1$, the first term above equals
\begin{align*}
\frac{1}{n^2}\biggiprod{\E M_1^{(t-s+u)} \partial_{\cdot\cdot} N_{u-2}^{(s)}- \E M_0^{(t-s+u)} \mathsf{B}_{t-s+u-1}  \mathsf{D}^{(t-s+u-2)} \partial_{\cdot\cdot} N_{u-3}^{(s)} }{V\circ V}+Q^{(u)}_{u-1}.
\end{align*}
Iterating this procedure until (\ref{eqn:N_der_recursion}) is applied with $u$ replaced by $1$, we have 
\begin{align*}
\frac{1}{n}\E \tr(N_u^{(s)})&=Q^{(u)}_{u-1}+\cdots+Q^{(u)}_1- \frac{1}{n}\E \tr\Big(\mathsf{B}_{t-s+u}  \mathsf{D}^{(t-s+u-1)} N_{u-2}^{(s)}\Big)\\
&\qquad + \frac{1}{n^2}\sum_{i,j \in [n] } \E  \mathsf{F}_{t-s+u,i}^{(2)}(z_i^{(t-s+u)}) \partial_{ij}z_i^{(t-s+u)} N_{u-1,ij}^{(s)} V_{ij}^2.
\end{align*}
Note that compared to (\ref{ineq:trace_pop_1}) that has two remainder terms in the second and third lines therein, due to $\partial_{\cdot\cdot} N^{(s)}_0=0$ we only have one remainder term in the second line of the above display. From here we may use almost the same arguments as in the proof of Lemma \ref{lem:trace_cal_M_to_N} to prove the claim (\ref{ineq:trace_pop_claim_2}). We omit these repetitive details.  \qed

\subsection{Proof of Corollary \ref{cor:trace_M_loo}}

We write $\{M_{\cdot}^{(t)}(V_{[-k]})\}\equiv \{M_{\cdot,V_{[-k]}}^{(t)}\}$ in the proof. Then with $\Delta M_{s,[-k]}^{(t)}\equiv M_{s,[-k]}^{(t)}-M_{s,V_{[-k]}}^{(t)}$, we have
\begin{align*}
\pnorm{\Delta M_{s,[-k]}^{(t)}}{F}&\leq \Lambda \pnorm{A_{[-k]}}{\op} \pnorm{\Delta M_{s-1,[-k]}^{(t)}}{F}+(K\Lambda)^2 \pnorm{\Delta M_{s-2,[-k]}^{(t)}}{F}\\
&\qquad + \Lambda \pnorm{\mathsf{B}_{t-s+1}-\mathsf{B}_{t-s+1,[-k]}}{F} \pnorm{M_{s-2,[-k]}^{(t)}}{\op}.
\end{align*}
Using $\pnorm{M_{s,[-k]}^{(t)}}{\op}\leq \big(K\Lambda (1+\pnorm{A_{[-k]}}{\op})\big)^{ct}$ and
\begin{align*}
\pnorm{\mathsf{B}_{\cdot}-\mathsf{B}_{\cdot,[-k]}}{F}^2&\leq 2\sum_{j \in [n]}\biggabs{\frac{1}{n}\sum_{\ell \in [n]} V_{[-k],j\ell}^2 \Big(\E \mathsf{F}_{\cdot,\ell}'(z_{\ell}^{(\cdot)})-\E \mathsf{F}_{\cdot,\ell}'(z_{[-k],\ell}^{(\cdot)})\Big)}^2\\
&\qquad + 2\sum_{j \in [n]}\biggabs{\frac{1}{n}\sum_{\ell \in [n]} \big(V_{j\ell}^2-V_{[-k],j\ell}^2\big) \E \mathsf{F}_{\cdot,\ell}'(z_{\ell}^{(\cdot)})}^2 \\
&\leq (K\Lambda)^c\cdot \big(1+\pnorm{z^{(\cdot)}-z^{(\cdot)}_{[-k]}}{}^2\big),
\end{align*}
we obtain via the above recursion and Proposition \ref{prop:AMP_iterate_infty} that
\begin{align*}
\max_{s \in [t]}\E\pnorm{\Delta M_{s,[-k]}^{(t)}}{F} \leq (K\Lambda)^{ct}(1+\pnorm{z^{(0)}}{\infty})\sqrt{\log n}. 
\end{align*}
The claim follows by using the simple estimate $n^{-1}\max_{d_0 \in \{V_{[1:n]}^2,1_n\} }\abs{\E [\tr(\mathfrak{D}_{d_0}M_{s,[-k]}^{(t)})]- \E [\tr(\mathfrak{D}_{d_0}M_{s,V_{[-k]}}^{(t)})]}\leq K^2 n^{-1/2}\E\pnorm{\Delta M_{s,[-k]}^{(t)}}{F}$ and noting that Proposition \ref{prop:trace_M} applies to $n^{-1} \max_{d_0 \in \{V_{[1:n]}^2,1_n\} }\E [\tr(\mathfrak{D}_{d_0}M_{s,V_{[-k]}}^{(t)})]$. \qed

\subsection{Proof of Lemma \ref{lem:M_bar_loo_quad_diff}}

Recall the definition of $\mathfrak{A}^{(u)}_{\tau_u}$ in the proof of Proposition \ref{prop:z_der_l2}, where $\tau_u \in \{1,2\}$ and $\mathfrak{A}^{(u)}_1 \equiv A$, $\mathfrak{A}^{(u)}_2 \equiv \mathsf{B}_{u+1}$. Recall also $\mathfrak{A}^{(u)}_{\tau_u,[-k]}=A_{[-k]}\bm{1}_{\tau_u=1}+\mathsf{B}_{u+1}\bm{1}_{\tau_u=2}$.

Fix $s\in [t]$. Note that $\iprod{A_k}{\overline{M}_{s,[-k]}^{(t)}  A_k }$ consists of at most $c^t$ terms of the form $\bigiprod{A_k}{\prod_{w \in [1:r-1]}\big(\overline{\mathsf{D}}_{[-k]}^{(t_w)}\mathfrak{A}_{\tau_{t_w},[-k]}^{(t_w)}\big)  \overline{\mathsf{D}}_{[-k]}^{(t_{r})} A_k}$, where $r \in [s]$. 
Let $\Delta \overline{\mathsf{D}}_{[-k]}^{(\cdot)}\equiv \mathsf{D}_{[-k]}^{(\cdot)}-\overline{\mathsf{D}}_{[-k]}^{(\cdot)}$. Then we have $\pnorm{\Delta \overline{\mathsf{D}}_{[-k]}^{(\cdot)}}{F}\leq \Lambda \pnorm{z^{(\cdot)}-z^{(\cdot)}_{[-k]}}{}$. Consequently,
\begin{align}\label{ineq:M_bar_loo_quad_diff_1}
&\bigabs{ \bigiprod{A_k}{ \mathsf{D}_{[-k]}^{(t_1)}\mathfrak{A}_{\tau_{t_1},[-k]}^{(t_1)} \mathsf{D}_{[-k]}^{(t_2)} \mathfrak{A}_{\tau_{t_2},[-k]}^{(t_2)}\cdots  {\mathsf{D}}_{[-k]}^{(t_{r})}A_k} \nonumber \\
	&\qquad - \bigiprod{A_k}{\overline{\mathsf{D}}_{[-k]}^{(t_1)}\mathfrak{A}_{\tau_{t_1},[-k]}^{(t_1)} {\mathsf{D}}_{[-k]}^{(t_2)} \mathfrak{A}_{\tau_{t_2},[-k]}^{(t_2)}\cdots  {\mathsf{D}}_{[-k]}^{(t_{r})} A_k} }\nonumber\\
&= \biggabs{\sum_{\ell \in [n]\setminus \{k\}} A_{k\ell} \Delta \overline{\mathsf{D}}_{[-k],\ell\ell}^{(t_1)} \cdot \bigiprod{\big(\mathfrak{A}_{\tau_{t_1},[-k]}^{(t_1)}\big)_{\ell\cdot} }{ \mathsf{D}_{[-k]}^{(t_2)} \mathfrak{A}_{\tau_{t_2},[-k]}^{(t_2)} \cdots  \mathsf{D}_{[-k]}^{(t_{r})} A_k} }\nonumber\\
& \qquad +  \bigabs{ A_{kk} \Delta \overline{\mathsf{D}}_{[-k],kk}^{(t_1)} \cdot \bigiprod{\big(\mathfrak{A}_{\tau_{t_1},[-k]}^{(t_1)}\big)_{k\cdot} }{ \mathsf{D}_{[-k]}^{(t_2)} \mathfrak{A}_{\tau_{t_2},[-k]}^{(t_2)} \cdots  \mathsf{D}_{[-k]}^{(t_{r})} A_k}    }\nonumber\\
&\leq \pnorm{A_k}{}\cdot \Lambda \pnorm{z^{(t_1)}-z^{(t_1)}_{[-k]}}{}\cdot  \max_{\ell \in [n]\setminus \{k\}} \bigabs{\bigiprod{\big(\mathfrak{A}_{\tau_{t_1},[-k]}^{(t_1)}\big)_{\ell\cdot}}{ \mathsf{D}_{[-k]}^{(t_2)}\mathfrak{A}_{\tau_{t_2},[-k]}^{(t_2)}\cdots \mathsf{D}_{[-k]}^{(t_{r})} A_k}}\nonumber\\
&\qquad + \abs{A_{kk}}\cdot \big(K\Lambda(1+\pnorm{A}{\op})\big)^{ct}. 
\end{align}
Note that for $\ell \in [n]\setminus \{k\}$, $\big(\mathfrak{A}_{\tau_{t_1},[-k]}^{(t_1)}\big)_{\ell\cdot}$ and $\mathsf{D}_{[-k]}^{(t_2)} \mathfrak{A}_{\tau_{t_2},[-k]}^{(t_2)}\cdots  \mathsf{D}_{[-k]}^{(t_{r})}$ are independent of $A_k$, so by subgaussian inequality first conditional on $A_{[-k]}$ and then unconditionally, with probability at least $1-cn^{-D}$, 
\begin{align*}
\max_{\ell \in [n]\setminus \{k\}} \bigabs{\bigiprod{\big(\mathfrak{A}_{\tau_{t_1},[-k]}^{(t_1)}\big)_{\ell\cdot}}{ \mathsf{D}_{[-k]}^{(t_2)}\mathfrak{A}_{\tau_{t_2},[-k]}^{(t_2)}\cdots \mathsf{D}_{[-k]}^{(t_{r})} A_k}}\leq c_D (K\Lambda)^{ct} \sqrt{\log n}\cdot n^{-1/2}. 
\end{align*}
Combined with Proposition \ref{prop:AMP_iterate_infty}, for $t\leq \log n/c$, with probability at least $1-cn^{-D}$, 
\begin{align*}	
\hbox{LHS of }(\ref{ineq:M_bar_loo_quad_diff_1})&\leq c_D (K\Lambda)^{ct}\log^c n\cdot (1+\pnorm{z^{(0)}}{\infty})\cdot n^{-1/2}. 
\end{align*}
Iterating the bound proves the claim. \qed

\section{Proofs of remaining results in Section \ref{section:main_results}}\label{section:proof_remaining_main}

\subsection{Proof of Theorem \ref{thm:AMP_sym}}\label{subsection:proof_AMP_sym}
By apriori estimates for $(z^{(s+1)}_k)_{s \in [0:t]}$ (cf. Proposition \ref{prop:AMP_iterate_infty}) and $ \big(\iprod{A_k}{\mathsf{F}_s (z_{[-k]}^{(s)})}\big)_{s \in [0:t]}$, (\ref{ineq:AMP_sym_proof_loo_represent}) also holds for moment estimates. Consequently, for $t\leq \log n/c$,
\begin{align}\label{ineq:AMP_sym_proof_1}
&\bigabs{\E \psi\big[ \big(z^{(s+1)}_k\big)_{s \in [0:t]}\big]-\E \psi\big[\big(\bigiprod{A_k}{\mathsf{F}_s (z_{[-k]}^{(s)})}\big)_{s \in [0:t]}\big]}\nonumber\\
&\leq \Lambda\cdot \E^{1/2} \Big(1+\bigpnorm{ (z^{(s+1)}_k)_{s \in [0:t]} }{}+\bigpnorm{ \big(\bigiprod{A_k}{\mathsf{F}_s (z_{[-k]}^{(s)})}\big)_{s \in [0:t]} }{}\Big)^{2\mathfrak{p}}\nonumber\\
&\qquad\qquad \times \E^{1/2} \bigpnorm{ \Big(z^{(s+1)}_k-\bigiprod{A_k}{\mathsf{F}_s (z_{[-k]}^{(s)})}\Big)_{s\in [0:t]}  }{}^2\nonumber\\
&\leq  \big(K\Lambda\log n\cdot (1+\pnorm{z^{(0)}}{\infty})\big)^{c_{\mathfrak{p}} t^3}\cdot n^{-1/2}. 
\end{align}
Below we shall replace $\E \psi\big[\big(\bigiprod{A_k}{\mathsf{F}_s (z_{[-k]}^{(s)})}\big)_{s\in [0:t]}\big]$ in the left hand side of the above display by the quantity $\E \psi\big[(Z^{(s+1)}_k)_{s\in [0:t]}\big]$ and its variant.

\noindent (\textbf{Step 1}).  For $s_1,s_2 \in [0:t]$, let for $k \in [n]$
\begin{align}\label{ineq:AMP_sym_proof_alpha}
\begin{cases}
\alpha_k^{(s_1+1,s_2+1)}\equiv \frac{1}{n}\sum_{\ell \in [n]} V_{k\ell}^2 \E \mathsf{F}_{s_1,\ell}(z^{(s_1)}_\ell)\mathsf{F}_{s_2,\ell}(z^{(s_2)}_\ell),\\
\alpha_{[-k],k}^{(s_1+1,s_2+1)}\equiv \frac{1}{n}\sum_{\ell \in [n]} V_{k\ell}^2  \mathsf{F}_{s_1,\ell}(z^{(s_1)}_{[-k],\ell})\mathsf{F}_{s_2,\ell}(z^{(s_2)}_{[-k],\ell}).
\end{cases}
\end{align}
Note that $\alpha_{[-k],k}^{(s_1+1,s_2+1)}$ is random (without expectation). Let $Z_{k;\alpha}' \sim \mathcal{N}(0,I_{t+1})$ be a standard Gaussian vector, independent of all other random variables. Let
\begin{align*}
\Sigma_{k;\alpha}\equiv \big(\alpha_{k}^{(s_1+1,s_2+1)}\big)_{s_1,s_2 \in [0:t]},\quad \Sigma_{[-k],k;\alpha}\equiv \big(\alpha_{[-k],k}^{(s_1+1,s_2+1)}\big)_{s_1,s_2 \in [0:t]}.
\end{align*}
It is easy to see that $\Sigma_{k;\alpha},\Sigma_{[-k],k;\alpha} \in \R^{(t+1)\times (t+1)}$ are covariance matrices, so we may further define
\begin{align*}
\big(Z_{k;\alpha}^{(s+1)}\big)_{s \in [0:t]}\equiv \Sigma_{k;\alpha}^{1/2} Z_{k;\alpha}',\quad \big(Z_{[-k],k;\alpha}^{(s+1)}\big)_{s \in [0:t]}\equiv \Sigma_{[-k],k;\alpha}^{1/2} Z_{k;\alpha}'\in \R^{t+1}.
\end{align*}	
Using $\E \psi\big[\big(\bigiprod{A_k}{\mathsf{F}_s (z_{[-k]}^{(s)})}\big)_{s \in [0:t]}\big]= \E \psi\big[\big(Z^{(s+1)}_{[-k],k;\alpha}\big)_{s \in [0:t]}\big]$, we then have
\begin{align}\label{ineq:AMP_sym_proof_2}
&\bigabs{\E \psi\big[\big(\bigiprod{A_k}{\mathsf{F}_s (z_{[-k]}^{(s)})}\big)_{s \in [0:t]}\big]- \E \psi \big[  \big(Z^{(s+1)}_{k;\alpha}\big)_{s \in [0:t]} \big]}\nonumber\\
&= \bigabs{\E \psi \big[  \big(Z^{(s+1)}_{k;\alpha}\big)_{s \in [0:t]} \big]- \E \psi \big[  \big(Z^{(s+1)}_{[-k],k;\alpha}\big)_{s \in [0:t]} \big]}   \nonumber\\ 
&\leq (K\Lambda)^{c_{\mathfrak{p}}t}\cdot \E^{1/2} \bigpnorm{\Sigma_{k;\alpha}^{1/2}-\Sigma_{[-k],k;\alpha}^{1/2} }{F}^2 \nonumber\\
&\stackrel{(\ast)}{\leq} (K\Lambda)^{c_{\mathfrak{p}}t }\cdot  \big(\log n\cdot (1+\pnorm{z^{(0)}}{\infty})\big)^c\cdot n^{-1/4}.
\end{align}
In $(\ast)$ above, we used the following estimate: By Lemma \ref{lem:cov_sqrt_diff},
\begin{align*}
&\E \bigpnorm{\Sigma_{k;\alpha}^{1/2}-\Sigma_{[-k],k;\alpha}^{1/2} }{F}^2\leq (t+1)^{1/2}\cdot \E^{1/2} \pnorm{\Sigma_{k;\alpha}-\Sigma_{[-k],k;\alpha} }{F}^2\\
&\leq (K\Lambda)^{ct} \cdot \max_{s_1,s_2 \in [0:t]}\bigg\{\max_{u,v \in \{s_1,s_2\}}\E^{1/2} \bigg(\frac{1}{n}\sum_{\ell \in [n]}  \big(1+\abs{z_\ell^{(u)}}+\abs{z_{[-k],\ell}^{(u)}}\big)\abs{ z_\ell^{(v)}-z_{[-k],\ell}^{(v)}  }\bigg)^2  \\
&\qquad\qquad + \E^{1/2} \bigg(\frac{1}{n} \sum_{\ell \in [n]} V_{k\ell}^2\cdot (\mathrm{id}-\E)\mathsf{F}_{s_1,\ell}(z^{(s_1)}_\ell)\mathsf{F}_{s_2,\ell}(z^{(s_2)}_\ell)  \bigg)^2\bigg\}\\
&\leq (K\Lambda)^{ct} \cdot \big(\log n\cdot (1+\pnorm{z^{(0)}}{\infty})\big)^c\cdot n^{-1/2}. 
\end{align*}
Here in the last inequality of the above display, we bound the second term in the bracket by applying the Gaussian concentration Lemma \ref{lem:gaussian_conc} to $H_0(A)\equiv n^{-1}\sum_{\ell \in [n]} V_{k\ell}^2 \mathsf{F}_{s_1,\ell}(z^{(s_1)}_\ell)\mathsf{F}_{s_2,\ell}(z^{(s_2)}_\ell)$, upon computing, for any symmetric matrices $A,A_1,A_2$, the following estimates: (i) $\abs{H_0(A_1)-H_0(A_2)}\leq \big(K\Lambda (1+\pnorm{A_1}{\op}+\pnorm{A_2}{\op})\big)^{ct}(1+\pnorm{z^{(0)}}{\infty})^c\cdot \pnorm{A_1-A_2}{\op}$, and (ii) $\abs{H_0(A)}\leq \big(K\Lambda (1+\pnorm{A}{\op}) \big)^{ct} (1+\pnorm{z^{(0)}}{\infty})^c$.

Combining (\ref{ineq:AMP_sym_proof_1}) and (\ref{ineq:AMP_sym_proof_2}), we conclude 
\begin{align}\label{ineq:AMP_sym_proof_3}
&\max_{k \in [n]}\bigabs{\E \psi\big[ \big(z^{(s+1)}_k\big)_{s \in [0:t]}\big]-  \E \psi \big[\big(Z_{k;\alpha}^{(s+1)}\big)_{s \in [0:t]}\big] }\nonumber\\
&\leq  \big(K\Lambda\log n\cdot (1+\pnorm{z^{(0)}}{\infty})\big)^{c_{\mathfrak{p}} t^3}\cdot n^{-1/4}\equiv \delta_n^{(t+1)}. 
\end{align}
\noindent (\textbf{Step 2}). Recall the definition of $Z^{(1)},Z^{(2)},\ldots$ with correlation structures specified in Definition \ref{def:AMP_se}. For $k\in [n]$, let $\Sigma_k\equiv \big(\cov(Z_k^{(s_1+1)},Z_k^{(s_2+1)})\big)_{s_1,s_2 \in [0:t]}$. Then using Lemma \ref{lem:cov_sqrt_diff} again,
\begin{align*}
&\bigabs{\E \psi \big[\big(Z_{k;\alpha}^{(s+1)}\big)_{s \in [0:t]}\big]-\E \psi \big[\big(Z_{k}^{(s+1)}\big)_{s \in [0:t]}\big]}\\
&\leq (K\Lambda)^{c_{\mathfrak{p}} t}\cdot  \pnorm{\Sigma_{k;\alpha}^{1/2}-\Sigma_{k}^{1/2} }{F}\leq (K\Lambda)^{c_{\mathfrak{p}} t}\cdot  \pnorm{\Sigma_{k;\alpha}-\Sigma_{k} }{F}^{1/2}\\
&\leq  (K\Lambda)^{c_{\mathfrak{p}}' t}\cdot\max_{s_1,s_2 \in [0:t]} \bigg\{\frac{1}{n}\sum_{\ell \in [n]} \bigabs{ \E \mathsf{F}_{s_1,\ell}(z^{(s_1)}_\ell)\mathsf{F}_{s_2,\ell}(z^{(s_2)}_\ell) - \E \mathsf{F}_{s_1,\ell}(Z^{(s_1)}_\ell)\mathsf{F}_{s_2,\ell}(Z^{(s_2)}_\ell)     } \bigg\}^{1/2}.
\end{align*}
Here we adopt the convention that $Z^{(0)}_\cdot = z^{(0)}_\cdot$. Now with $\mathscr{F}_{t+1}(L,\mathfrak{p})\equiv \{\psi: \R^{t+1}\to \R, \abs{\psi(x)-\psi(y)}\leq L\cdot (1+\pnorm{x}{}+\pnorm{y}{})^{\mathfrak{p}}\cdot \pnorm{x-y}{}\hbox{ for all }x, y \in \R^{t+1}\}$ and 
\begin{align*}
\gamma_{t+1}(L,\mathfrak{p})\equiv \max_{k \in [n]} \sup_{\psi \in \mathscr{F}_{t+1}(L,\mathfrak{p})}\bigabs{\E \psi \big[\big(Z_{k;\alpha}^{(s+1)}\big)_{s \in [0:t]}\big]-\E \psi \big[\big(Z_{k}^{(s+1)}\big)_{s \in [0:t]}\big]},
\end{align*}
using (\ref{ineq:AMP_sym_proof_3}) at iteration $t$, we obtain the recursion: 
\begin{align*}
\gamma_{t+1}(\Lambda,\mathfrak{p})&\leq  (K\Lambda)^{c_{\mathfrak{p}} t}\cdot \big[ \big(\delta_n^{(t)}\big)^{1/2}+\gamma_t^{1/2}(2\Lambda,\mathfrak{p})\big].
\end{align*}	
Now iterating the bound and using the initial condition $\gamma_1\equiv 0$ (equivalently, $\E \psi \big[Z_{k;\alpha}^{(1)}\big]=\E \psi \big[Z_{k}^{(1)}\big]$ for all $k \in [n]$), we have 
\begin{align}\label{ineq:AMP_sym_proof_4}
\gamma_{t+1}(\Lambda,\mathfrak{p})\leq  \big(K\Lambda\log n\cdot (1+\pnorm{z^{(0)}}{\infty})\big)^{c_{\mathfrak{p}} t^3}\cdot n^{-1/c_0^t},
\end{align}
where $c_0>0$ is a universal constant. The claimed error estimate  now follows from (\ref{ineq:AMP_sym_proof_3}) and (\ref{ineq:AMP_sym_proof_4}), and noting that $t\leq \log n/c$ can be removed for free as otherwise we may use the (trivial) estimate in Proposition  \ref{prop:AMP_iterate_infty} instead. 

\noindent (\textbf{Step 3}). Consider the case $\psi\equiv \psi_1:\R\to \R$, and we assume without loss of generality that $\mathfrak{p}\geq 2$. Then (\ref{ineq:AMP_sym_proof_1}) reduces to 
\begin{align}\label{ineq:AMP_sym_proof_single_1}
\bigabs{\E \psi_1 \big(z^{(t+1)}_k\big)-\E \psi_1\big(\bigiprod{A_k}{\mathsf{F}_s (z_{[-k]}^{(s)})}\big) }\leq  \big(K\Lambda\log n\cdot (1+\pnorm{z^{(0)}}{\infty})\big)^{c_{\mathfrak{p}} t^3}\cdot n^{-1/2}. 
\end{align}
Moreover, mimicking the proof of (\ref{ineq:AMP_sym_proof_2}), with $\sigma_{\alpha}^{(s)}\equiv  1\wedge \min_{k \in [n]} \var^{1/2}(Z_{k;\alpha}^{(s)})$,
\begin{align}\label{ineq:AMP_sym_proof_single_2}
&\bigabs{\E \psi_1\big(\bigiprod{A_k}{\mathsf{F}_t (z_{[-k]}^{(t)})}\big)- \E \psi_1 \big(Z^{(t+1)}_{k;\alpha}\big) }\nonumber\\
&= \bigabs{\E \psi_1   \big(Z^{(t+1)}_{k;\alpha}\big)- \E \psi_1  \big(Z^{(t+1)}_{[-k],k;\alpha}\big) }   \nonumber\\ 
&\leq (K\Lambda)^{c_{\mathfrak{p}}t}\cdot \E^{1/2} \bigabs{e_{t+1}^\top \big(\Sigma_{k;\alpha}^{1/2}-\Sigma_{[-k],k;\alpha}^{1/2}\big) e_{t+1} }^2\nonumber\\
&\leq (K\Lambda)^{c_{\mathfrak{p}}t}\cdot (\sigma_{\alpha}^{(t+1)})^{-1}\cdot \E^{1/2} \bigabs{e_{t+1}^\top \big(\Sigma_{k;\alpha}-\Sigma_{[-k],k;\alpha}\big) e_{t+1} }^2\nonumber\\
&\leq (K\Lambda)^{c_{\mathfrak{p}}t}\cdot (\sigma_{\alpha}^{(t+1)})^{-1} \cdot \big(\log n\cdot (1+\pnorm{z^{(0)}}{\infty})\big)^c\cdot n^{-1/2}. 
\end{align}
Combining (\ref{ineq:AMP_sym_proof_single_1})-(\ref{ineq:AMP_sym_proof_single_2}), we arrive the following version of (\ref{ineq:AMP_sym_proof_3}):
\begin{align}\label{ineq:AMP_sym_proof_single_3}
&\max_{k \in [n]} \bigabs{\E \psi_1 \big(z^{(t+1)}_k\big)-\E \psi_1 \big(Z^{(t+1)}_{k;\alpha}\big) }\nonumber\\
&\leq  \big(K\Lambda\log n\cdot (1+\pnorm{z^{(0)}}{\infty})\big)^{c_{\mathfrak{p}} t^3}\cdot (\sigma_{\alpha}^{(t+1)})^{-1} \cdot  n^{-1/2}\equiv \bar{\delta}_n^{(t+1)}. 
\end{align}
Using a similar argument as in the beginning of Step 2, we have 
\begin{align*}
&\bigabs{\E \psi \big(Z_{k;\alpha}^{(t+1)}\big)-\E \psi \big(Z_{k}^{(t+1)}\big)}\\
&\leq (K\Lambda)^{c_{\mathfrak{p}} t}\cdot  \bigabs{e_{t+1}^\top\big(\Sigma_{k;\alpha}^{1/2}-\Sigma_{k}^{1/2}\big) e_{t+1} } \leq (K\Lambda)^{c_{\mathfrak{p}} t}\cdot (\sigma_{\alpha}^{(t+1)})^{-1}\cdot  \bigabs{e_{t+1}^\top\big(\Sigma_{k;\alpha}-\Sigma_{k}\big) e_{t+1} }\\
&\leq  (K\Lambda)^{c_{\mathfrak{p}}' t}\cdot (\sigma_{\alpha}^{(t+1)})^{-1} \cdot \frac{1}{n}\sum_{\ell \in [n]} \bigabs{ \E \mathsf{F}_{t,\ell}^2(z^{(t)}_\ell)- \E \mathsf{F}_{t,\ell}^2(Z^{(t)}_\ell)     } .
\end{align*}
Consequently, by letting
\begin{align*}
\bar{\gamma}_{t+1}(L,\mathfrak{p})\equiv \max_{k \in [n]} \sup_{\psi_1 \in \mathscr{F}_{1}(L,\mathfrak{p})}\bigabs{\E \psi_1 \big(Z_{k;\alpha}^{(t+1)}\big)-\E \psi \big(Z_{k}^{(t+1)}\big)},
\end{align*}
and using (\ref{ineq:AMP_sym_proof_single_3}) at iteration $t$, we arrive at the recursion
\begin{align*}
\bar{\gamma}_{t+1}(\Lambda,\mathfrak{p})&\leq  (K\Lambda)^{c_{\mathfrak{p}} t}\cdot (\sigma_{\alpha}^{(t+1)})^{-1}\cdot \big[ \bar{\delta}_n^{(t)}+\gamma_t(2\Lambda,\mathfrak{p})\big].
\end{align*}	
Now iterating the bound to obtain 
\begin{align}\label{ineq:AMP_sym_proof_single_4}
\bar{\gamma}_{t+1}(\Lambda,\mathfrak{p})\leq  \Big(K\Lambda\min_{s \in [t+1]} (\sigma_{\alpha}^{(s)})^{-1}\cdot\log n\cdot (1+\pnorm{z^{(0)}}{\infty})\Big)^{c_{\mathfrak{p}} t^3} \cdot n^{-1/2}.
\end{align}
Combining (\ref{ineq:AMP_sym_proof_single_3}) and (\ref{ineq:AMP_sym_proof_single_4}), we conclude that
\begin{align}\label{ineq:AMP_sym_proof_single_5}
&\max_{k \in [n]} \bigabs{\E \psi_1 \big(z^{(t+1)}_k\big)-\E \psi_1 \big(Z^{(t+1)}_{k}\big) }\nonumber\\
&\leq  \Big(K\Lambda\min_{s \in [t+1]} (\sigma_{\alpha}^{(s)})^{-1}\cdot \log n\cdot (1+\pnorm{z^{(0)}}{\infty})\Big)^{c_{\mathfrak{p}} t^3} \cdot n^{-1/2}.
\end{align}
On the other hand, for the special case $\psi_1(x)=x^2$, all the above estimates remain valid without the term involving $\{(\sigma_{\alpha}^{(s+1)})^{-1}\}_{s \in [t]}$. This means, with $\sigma^{(s)}\equiv  1\wedge \min_{k \in [n]} \var^{1/2}(Z_{k}^{(s)})$, we have
\begin{align*}
\max_{s \in [t+1]}\abs{(\sigma_{\alpha}^{(s)})^2-(\sigma^{(s)})^2}\leq \big(K\Lambda\log n\cdot (1+\pnorm{z^{(0)}}{\infty})\big)^{c_{\mathfrak{p}} t^3}\cdot  n^{-1/2}\equiv \err_{n,t}.
\end{align*}
Consequently, as $\sigma_\ast^{[t+1]}=\min_{s \in [t+1]} \sigma^{(s)}$, if $\err_{n,t}\leq \sigma_\ast^{[t+1]}/2$, we have $\min_{s \in [t+1]} (\sigma_{\alpha}^{(s)})^{-1}\leq C_0 \min_{s \in [t+1]} (\sigma^{(s)})^{-1}$ for some universal $C_0>0$. In other words, if $\err_{n,t}\leq \sigma_\ast^{[t+1]}/2$, (\ref{ineq:AMP_sym_proof_single_5}) entails that 
\begin{align*}
&\max_{k \in [n]} \bigabs{\E \psi_1 \big(z^{(t+1)}_k\big)-\E \psi_1 \big(Z^{(t+1)}_{k}\big) }\leq  \big(K\Lambda \sigma_\ast^{[t+1],-1}\log n\cdot (1+\pnorm{z^{(0)}}{\infty})\big)^{c_{\mathfrak{p}}'' t^3} \cdot n^{-1/2}.
\end{align*}
The case for $\err_{n,t}> \sigma_\ast^{[t+1]}/2$ can be assimilated into the above estimate by possibly enlarging $c_{\mathfrak{p}}'' >0$ (upon using apriori bounds due to the delocalization estimates in Proposition \ref{prop:AMP_iterate_infty}).
\qed

\subsection{Proof of Theorem \ref{thm:AMP_sym_avg}}\label{subsection:proof_AMP_sym_avg}

Recall the state evolution and $Z^{(1)},Z^{(2)},\ldots$ in Definition \ref{def:AMP_se}. For $k \in [n]$, let $\sigma_{t,k}^2\equiv \var(Z^{(t)}_k)$ and recall $
\overline{b}_{t,k}= n^{-1}\sum_{\ell \in [n]} V_{k\ell}^2 \E [\mathsf{F}_{t,\ell}'(Z_\ell^{(t)})|z^{(0)}]$. 
The modified AMP iterate associated with $\{\overline{b}_t\}$ is denoted $\{\overline{z}^{(t)}\}$. For simplicity of notation, we work with the case $\psi_1\equiv \cdots\equiv \psi_n\equiv \psi$; the general case follows by minor, cosmetic modifications.

\begin{lemma}\label{lem:smoothed_AMP_moment}
	Assume the same conditions as in Theorem \ref{thm:AMP_sym_avg}. Fix any $\Lambda$-pseudo-Lipschitz function $\psi: \R^{t+1}\to \R$ of order 2. Then there exists some universal constant $c_0>0$, such that
	\begin{align*}
	&\biggabs{ \frac{1}{n}\sum_{k \in [n]} \Big(\E \big[\psi\big( \overline{z}^{(t+1)}_{k},\overline{z}^{(t)}_{k},\ldots,\overline{z}^{(1)}_{k}\big)|z^{(0)}\big]- \E \big[\psi\big( Z^{(t+1)}_{k},Z^{(t)}_{k},\ldots,Z^{(1)}_{k}\big)|z^{(0)}\big]\Big) }\\
	&\leq \big(K\Lambda\sigma_{*,\psi}^{-1}\cdot \log n\cdot (1+\pnorm{z^{(0)}}{\infty})\big)^{c_0 t^3}\cdot n^{-1/c_0^t}.
	\end{align*}
\end{lemma}
\begin{proof}
	We assume for notational simplicity that $\psi$ is full dimensional (i.e., $\psi$ depends on all of its arguments), and write $\sigma_{*,\psi}$ as $\sigma_{*}$. Let $\varphi \in C^\infty(\R)$ be a mollifier supported in $[-1,1]$ such that $\varphi\geq 0$ and $\int \varphi =1$, and $\varphi_\delta(\cdot)\equiv \delta^{-1}\varphi(\cdot/\delta)$. Let $\mathsf{F}_{\delta;s,\ell}\equiv \mathsf{F}_{s,\ell}*\varphi_\delta$. Then:
	\begin{itemize}
		\item For $q=1$, $
		\abs{\mathsf{F}_{\delta;s,\ell}'(x)} = \bigabs{  \int \mathsf{F}_{s,\ell}'(x-z) \varphi_\delta(z)\,\d{z}  }\leq \Lambda$. 
		\item For $q\geq 2, q \in \N$ and $\delta \in (0,1]$,
		\begin{align*}
		\bigabs{\mathsf{F}_{\delta;s,\ell}^{(q)}(x)}= \frac{1}{\delta^q} \biggabs{\int_{[x\pm \delta]} \mathsf{F}_{s,\ell}'(z) \varphi^{(q-1)}\big((x-z)/\delta\big)\,\d{z}}\leq 2\Lambda \pnorm{\varphi^{(q-1)}}{\infty}  \delta^{-(q-1)}.
		\end{align*}
		\item $
		\abs{\mathsf{F}_{\delta;s,\ell}(x)-\mathsf{F}_{s,\ell}(x)}=\bigabs{\int_{-1}^{1} \big(\mathsf{F}_{s,\ell}(x-\delta z)- \mathsf{F}_{s,\ell}(x)\big)\varphi(z)\,\d{z}  }\leq \Lambda \pnorm{\varphi}{\infty} \delta$. 
	\end{itemize}

	\noindent (\textbf{Step 1}). Let $\{z_\delta^{(\cdot)}\}$ be the AMP iterate associated with the non-linearities $\{\mathsf{F}_{\delta;s}\}_{s\in [0:t]}$ and the same initialization $z^{(0)}$. We define the state evolution associated with $z^{(\cdot)}_\delta$ as follows. Let $Z^{(1)}_\delta,Z^{(2)}_\delta,\ldots \in \R^n$ be a sequence of correlated $n$-dimensional centered Gaussian vectors with independent entries, whose correlation structure along the iteration path is determined with the initial condition $\var(Z^{(1)}_{\delta;k})\equiv n^{-1}\sum_{\ell \in [n]} V_{k\ell}^2 \mathsf{F}_{\delta;0,\ell}^2(z^{(0)}_\ell)$ for $k \in [n]$, and the subsequent recursive updates: 
	\begin{align}\label{ineq:smooth_AMP_step1_0}
	\cov\big(Z^{(s_1+1)}_{\delta;k},Z^{(s_2+1)}_{\delta;k}\big)\equiv \frac{1}{n}\sum_{\ell \in [n]} V_{k\ell}^2 \E \big[\mathsf{F}_{\delta;s_1,\ell}\big(Z^{(s_1)}_{\delta;\ell}\big) \mathsf{F}_{\delta;s_2,\ell}\big(Z^{(s_2)}_{\delta;\ell}\big)\big],\quad k \in [n]. 
	\end{align}
	For notational convenience, we shall identify $Z_{0;\cdot}^{(\cdot)}=Z_{\cdot}^{(\cdot)}$ as defined in Definition \ref{def:AMP_se} for the (unsmoothed) AMP iterate. Using almost the same proof as in Step 2 of Theorem \ref{thm:AMP_sym}, there exists some universal constant $c_0>0$ such that
	\begin{align}\label{ineq:smooth_AMP_step1_1}
	&\max_{k \in [n]}\bigabs{ \E \psi\big( Z^{(t+1)}_k,\ldots,Z^{(1)}_k\big)- \E \psi\big( Z^{(t+1)}_{\delta;k},\ldots,Z^{(1)}_{\delta;k}\big) }\leq  \big(K\Lambda\log n\big)^{c_0 t^2} \delta^{1/c_0^t}.
	\end{align}
	In particular, with $\sigma_{\delta;t,k}^2\equiv \var(Z_{\delta;k}^{(t)})$, we have 
	\begin{align}\label{ineq:smooth_AMP_step1_2}
	\max_{k \in [n]} \bigabs{\sigma_{\delta;t,k}^2-\sigma_{t,k}^2  }\leq
	\big(K\Lambda\log n\big)^{c_0 t^2}\cdot \delta^{1/c_0^t}\equiv \err_{n,t}(\delta).
	\end{align}

	\noindent (\textbf{Step 2}). Let $\{\overline{z}_\delta^{(\cdot)}\}$ be the modified AMP iterate using the Onsager correction vector $\{\overline{b}_{\delta;t}\}$ with the above state evolution, i.e., $\overline{b}_{\delta;t,k}= n^{-1}\sum_{\ell \in [n]} V_{k\ell}^2 \E [\mathsf{F}_{\delta;t,\ell}'(Z_{\delta;\ell}^{(t)})|z^{(0)}]$. In this step, we prove that for some universal constant $c_0>0$, if $t\leq \log n/c_0$, 
	\begin{align}\label{ineq:smooth_AMP_step2}
	&\biggabs{\frac{1}{n}\sum_{k \in [n]} \Big(\E \psi\big( z^{(t+1)}_{\delta;k},z^{(t)}_{\delta;k},\ldots,z^{(1)}_{\delta;k}\big)-  \E \psi\big( \overline{z}^{(t+1)}_{\delta;k},\overline{z}^{(t)}_{\delta;k},\ldots,\overline{z}^{(1)}_{\delta;k}\big) \Big) }\nonumber\\
	&\leq \big(K\Lambda \delta^{-1}\log n\cdot (1+\pnorm{z^{(0)}}{\infty})\big)^{c_0 t^3}\cdot n^{-1/c_0^t}.
	\end{align}
	To this end, note that by applying Theorem \ref{thm:AMP_sym} to the smoothed AMP iterate $\{z^{(\cdot)}_\delta\}$, 
	\begin{align*}
	&\pnorm{z_\delta^{(t+1)}-\overline{z}^{(t+1)}_\delta}{}\\
	&\leq \pnorm{A\mathsf{F}_{\delta;t}(z_\delta^{(t)})- A \mathsf{F}_{\delta;t}(\overline{z}^{(t)}_\delta )}{}+K^2\Lambda\cdot  \pnorm{\mathsf{F}_{\delta;t-1}(z_\delta^{(t-1)})- \mathsf{F}_{\delta;t-1}(\overline{z}^{(t-1)}_\delta )}{}\\
	&\qquad + K^2\cdot \frac{1}{n} \sum_{\ell \in [n]} \bigabs{\E \mathsf{F}_{\delta;t,\ell}'(z^{(t)}_{\delta;\ell})- \E \mathsf{F}_{\delta;t,\ell}'(Z^{(t)}_{\delta;\ell})}  \cdot \pnorm{\mathsf{F}_{\delta;t-1}(\overline{z}^{(t-1)}_\delta)}{}\\
	&\leq \big(\pnorm{A}{\op}+K^2\Lambda\big)\cdot\Big(\max_{s=t,t-1}\pnorm{\mathsf{F}_{\delta;s}(z_\delta^{(s)})- \mathsf{F}_{\delta;s}(\overline{z}^{(s)}_\delta)}{}\\
	&\qquad + \big(K\Lambda \delta^{-1}\log n\cdot (1+\pnorm{z^{(0)}}{\infty})\big)^{c t^3}\cdot n^{-1/c^t}\cdot (\sqrt{n}+\pnorm{\overline{z}^{(t-1)}_\delta}{}) \Big).
	\end{align*}
	Consequently,
	\begin{align*}
	&\max_{s \in [0:t+1]}\pnorm{z_\delta^{(s)}-\overline{z}^{(s)}_\delta}{}\leq (K\Lambda)^2 (1+\pnorm{A}{\op})\cdot \max_{s \in [0:t]}\pnorm{z_\delta^{(s)}-\overline{z}^{(s)}_\delta}{}\\
	&\qquad + (1+\pnorm{A}{\op})\big(K\Lambda \delta^{-1}\log n\cdot (1+\pnorm{z^{(0)}}{\infty})\big)^{c t^3}\cdot n^{-1/c^t}\cdot (\sqrt{n}+\max_{s \in [0:t]}\pnorm{\overline{z}^{(s)}_\delta}{}). 
	\end{align*}
	Iterating the bound, we arrive at
	\begin{align}\label{ineq:smooth_AMP_step2_1}
	\max_{s \in [0:t+1]} n^{-1/2}\pnorm{z_\delta^{(s)}-\overline{z}^{(s)}_\delta}{} &\leq  (1+\pnorm{A}{\op})^{c t}\cdot \big(K\Lambda \delta^{-1}\log n\cdot (1+\pnorm{z^{(0)}}{\infty})\big)^{c t^3}\nonumber\\
	&\qquad \times n^{-1/c^t}\cdot \Big(1+\max_{s \in [0:t]}\pnorm{\overline{z}^{(s)}_\delta}{\infty}\Big). 
	\end{align}
	On the other hand, using the same leave-one-out argument as in the proof of Proposition \ref{prop:AMP_iterate_infty} (where for the current modified AMP iterate, we only need to require $\mathsf{F}_{\cdot,\cdot}$ to be Lipschitz), with probability at least $1-(nt)^{c_0} e^{-(x^2\wedge n)/c_0}$,
	\begin{align}\label{ineq:smooth_AMP_step2_2}
	\max_{s \in [t+1]} \big(\pnorm{\overline{z}^{(s)}_\delta}{\infty}\vee \pnorm{\overline{z}^{(s)}}{\infty}\big)\leq (K\Lambda)^{c_0 t}x\cdot (1+\pnorm{z^{(0)}}{\infty}).
	\end{align}
	Finally note that with $\psi_{0}(z^{(\cdot)}_\delta)\equiv n^{-1}\sum_{k \in [n]} \psi\big(z_{\delta;k}^{(t+1)},\ldots,z_{\delta;k}^{(1)}\big) $ and similarly $\psi_{0}(\overline{z}^{(\cdot)}_\delta)\equiv n^{-1}\sum_{k \in [n]} \psi\big(\overline{z}_{\delta;k}^{(t+1)},\ldots,\overline{z}_{\delta;k}^{(1)}\big) $, we have 
	\begin{align}\label{ineq:smooth_AMP_step2_3}
	&\bigabs{ \psi_0(z^{(\cdot)}_\delta)-\psi_0(\overline{z}^{(\cdot)}_\delta)}\nonumber\\
	&\leq \Lambda t^{c_0}\cdot \max_{s \in [0:t+1]} \big(1+\pnorm{z^{(s)}_\delta}{\infty}+\pnorm{\overline{z}^{(s)}_\delta}{\infty}\big)^{c_0 }\cdot n^{-1/2} \pnorm{z_\delta^{(s)}-\overline{z}^{(s)}_\delta}{}.
	\end{align}
	Combining (\ref{ineq:smooth_AMP_step2_1})-(\ref{ineq:smooth_AMP_step2_3}), for any $D>0$, there exists some $c_1=c_1(D)>0$, such that for $t\leq  \log n/c_0$, on an event with probability at least $1-c_1 n^{-D}$,
	\begin{align*}
	\bigabs{ \psi_0(z^{(\cdot)}_\delta)-\psi_0(\overline{z}^{(\cdot)}_\delta)}\leq \big(c_1 K\Lambda \delta^{-1}\log n\cdot (1+\pnorm{z^{(0)}}{\infty})\big)^{c t^3}\cdot n^{-1/c^t}.
	\end{align*}
	The moment estimate in (\ref{ineq:smooth_AMP_step2}) follows by the following simple apriori estimates:  by assuming $\psi(0)=0$ without loss of generality, using Proposition \ref{prop:AMP_iterate_infty} we have
	\begin{align*}
	\E^{1/2}\abs{\psi_0(z^{(\cdot)}_\delta)}^2&\leq \Lambda \E^{1/2}\max_{s \in [t+1]}(1\vee \pnorm{z^{(s)}_\delta}{\infty})^{c}\cdot \big(n^{-1/2} \pnorm{z^{(s)}_\delta}{}\big)^2\\
	& \leq \big(K\Lambda \log n\cdot  (1+\pnorm{z^{(0)}}{\infty})\big)^{c t}.
	\end{align*}
	A completely similar estimate holds for $\E^{1/2}\abs{\psi_0(\overline{z}^{(\cdot)}_\delta)}^2$. 
	
	\noindent (\textbf{Step 3}). In this step, we prove that for $t\leq \log n/c_0$,
	\begin{align}\label{ineq:smooth_AMP_step3}
	&\biggabs{\frac{1}{n}\sum_{k \in [n]} \Big(\E \psi\big( \overline{z}^{(t+1)}_k,\overline{z}^{(t)}_k,\ldots,\overline{z}^{(1)}_k\big)-  \E \psi\big( \overline{z}^{(t+1)}_{\delta;k},\overline{z}^{(t)}_{\delta;k},\ldots,\overline{z}^{(1)}_{\delta;k}\big) \Big) }\nonumber\\
	&\leq \big(K\Lambda \sigma_{*}^{-1} \log n\cdot  (1+\pnorm{z^{(0)}}{\infty})\big)^{c_0 t^2}\cdot \big(\err_{n,t}(\delta)+\delta+n^{-100}\big).
	\end{align}
	The proof has certain similarities to Step 2. For $\err_{n,t}(\delta)\leq \sigma_{*}^2/2$, 
	\begin{align*}
	&\pnorm{\overline{z}_\delta^{(t+1)}-\overline{z}^{(t+1)}}{}\\
	&\leq \pnorm{A\mathsf{F}_{\delta;t}(\overline{z}_\delta^{(t)})- A \mathsf{F}_t(\overline{z}^{(t)})}{}+K^2\Lambda\cdot  \pnorm{\mathsf{F}_{\delta;t-1}(\overline{z}_\delta^{(t-1)})- \mathsf{F}_{t-1}(\overline{z}^{(t-1)})}{}\\
	&\qquad + K^2\cdot \frac{1}{n} \sum_{\ell \in [n]} \bigabs{\E \mathsf{F}_{\delta;t,\ell}'(Z^{(t)}_{\delta;\ell})- \E \mathsf{F}_{t,\ell}'(Z^{(t)}_{\ell})}  \cdot \pnorm{\mathsf{F}_{t-1}(\overline{z}^{(t-1)})}{}\\
	&\stackrel{(*)}{\leq} \big(\pnorm{A}{\op}+K^2\Lambda \big)\cdot \Big(\max_{s=t,t-1}\pnorm{\mathsf{F}_{\delta;s}(\overline{z}_\delta^{(s)})- \mathsf{F}_{\delta;s}(\overline{z}^{(s)})}{}+\max_{s=t,t-1} \pnorm{\mathsf{F}_{\delta;s}(\overline{z}^{(s)})- \mathsf{F}_{s}(\overline{z}^{(s)})}{}\\
	&\qquad + (K\Lambda)^{ct}\sigma_{*}^{-c}\cdot  \big(\err_{n,t}(\delta)+ \delta\big)\cdot \big(\sqrt{n}+\pnorm{\overline{z}^{(t-1)}}{}\big) \Big)\\
	&\leq \big(\pnorm{A}{\op}+1\big) \big(K\Lambda {\sigma}_{*}^{-1}\big)^{ct} \cdot \Big(\max_{s=t,t-1} \pnorm{\overline{z}_\delta^{(s)}-\overline{z}^{(s)}}{}+ \big(\err_{n,t}(\delta)+ \delta\big)\cdot \big(\sqrt{n}+\pnorm{\overline{z}^{(t-1)}}{}\big) \Big). 
	\end{align*}
	Here in $(\ast)$ we used the following estimate: for $\err_{n,t}(\delta)\leq \sigma_{*}^2/2$, by (\ref{ineq:smooth_AMP_step1_2}), 
	\begin{align*}
	&\frac{1}{n} \sum_{\ell \in [n]} \bigabs{\E \mathsf{F}_{\delta;t,\ell}'(Z^{(t)}_{\delta;\ell})- \E \mathsf{F}_{t,\ell}'(Z^{(t)}_{\ell})}  \\
	& = \frac{1}{n} \sum_{\ell \in [n]} \bigabs{\sigma_{\delta;t,k}^{-2}\E Z^{(t)}_{\delta;\ell}\mathsf{F}_{\delta;t,\ell}(Z^{(t)}_{\delta;\ell})- \sigma_{t,k}^{-2}\E Z^{(t)}_{\ell}\mathsf{F}_{t,\ell}(Z^{(t)}_{\ell})}  \\
	&\leq \frac{(K\Lambda)^{ct} }{\sigma_{*}^2 (\sigma_{*}^2 -\err_{n,t}(\delta))_+ }\cdot  \err_{n,t}(\delta)+\frac{1}{\sigma_{*}^{2} }\cdot  \frac{1}{n} \sum_{\ell \in [n]} \bigabs{\E Z^{(t)}_{\delta;\ell}\mathsf{F}_{\delta;t,\ell}(Z^{(t)}_{\delta;\ell})- \E Z^{(t)}_{\ell}\mathsf{F}_{t,\ell}(Z^{(t)}_{\ell})}\\
	&\leq (K\Lambda)^{ct}\sigma_{*}^{-c}\cdot  \big(\err_{n,t}(\delta)+ \delta\big). 
	\end{align*}
	Consequently,
	\begin{align*}
	&\max_{s \in [0:t+1]}\pnorm{\overline{z}_\delta^{(s)}-\overline{z}^{(s)}}{}\leq (\pnorm{A}{\op}+1) \big(K\Lambda \sigma_{*}^{-1}\big)^{ct}\\
	&\qquad \times \Big(\max_{s \in [0:t]} \pnorm{\overline{z}_\delta^{(s)}-\overline{z}^{(s)}}{}+ \big(\err_{n,t}(\delta)+ \delta\big)\cdot \big(\sqrt{n}+\max_{s \in [0:t]}\pnorm{\overline{z}^{(s)}}{}\big) \Big).
	\end{align*}
	Iterating the bound, we obtain
	\begin{align*}
	&\max_{s \in [0:t+1]}n^{-1/2}\pnorm{\overline{z}_\delta^{(s)}-\overline{z}^{(s)}}{}\nonumber\\
	&\leq \big(K\Lambda\sigma_{*}^{-1}\cdot (1\vee \pnorm{A}{\op})\big)^{c_0 t^2}\cdot \big(\err_{n,t}(\delta)+ \delta\big)\cdot \big(1+\max_{s \in [0:t]}\pnorm{\overline{z}^{(s)}}{\infty}\big).  
	\end{align*}
	From here the proof proceeds as in Step 2. The case $\err_{n,t}(\delta)> \sigma_{*}^2/2$ follows by the delocalization estimates in (\ref{ineq:smooth_AMP_step2_2}).

	\noindent (\textbf{Step 4}). Applying Theorem \ref{thm:AMP_sym} to the smoothed AMP iterate $\{z_\delta^{(\cdot)}\}$, we obtain 
	\begin{align*}
	&\biggabs{ \frac{1}{n}\sum_{k \in [n]} \Big(\E \psi\big( z^{(t+1)}_{\delta;k},z^{(t)}_{\delta;k},\ldots,z^{(1)}_{\delta;k}\big)- \E \psi\big( Z^{(t+1)}_{\delta;k},Z^{(t)}_{\delta;k},\ldots,Z^{(1)}_{\delta;k}\big)\Big) }\\
	&\leq \big(K\Lambda \delta^{-1}\log n\cdot (1+\pnorm{z^{(0)}}{\infty})\big)^{c_0 t^3}\cdot n^{-1/c_0^t}. 
	\end{align*}
	Combined with (the averaged version of) (\ref{ineq:smooth_AMP_step1_1}), (\ref{ineq:smooth_AMP_step2}) and (\ref{ineq:smooth_AMP_step3}), we have for $t\leq \log n/c_0$,
	\begin{align*}
	&\biggabs{ \frac{1}{n}\sum_{k \in [n]} \Big(\E \psi\big( \overline{z}^{(t+1)}_{k},\overline{z}^{(t)}_{k},\ldots,\overline{z}^{(1)}_{k}\big)- \E \psi\big( Z^{(t+1)}_{k},Z^{(t)}_{k},\ldots,Z^{(1)}_{k}\big)\Big)  }\\
	&\leq \big(K\Lambda\sigma_{*}^{-1}\cdot \log n\cdot (1+\pnorm{z^{(0)}}{\infty})\big)^{c_0 t^3}\cdot \big(\delta^{-c_0 t^3}n^{-1/c_0^t}+\delta^{1/c_0^t}\big).
	\end{align*}
	Choosing $\delta\equiv n^{-1/c_\ast^t}$ for sufficiently large $c_\ast$ to conclude.
\end{proof}

\begin{proof}[Proof of Theorem \ref{thm:AMP_sym_avg}]
	Without loss of generality we assume $\psi(0)=0$ (otherwise consider $\psi-\psi(0)$). By Lemma \ref{lem:smoothed_AMP_moment}, we only need to prove concentration of $H(A)\equiv n^{-1}\sum_{k \in [n]} \psi\big( \overline{z}^{(t+1)}_{k}(A),\overline{z}^{(t)}_{k}(A),\ldots,\overline{z}^{(1)}_{k}(A)\big)$. It is easy to verify that
	\begin{align*}
	\abs{H(A)}\leq c t\Lambda\cdot \Big(1+\max_{s\in [t+1]}n^{-1}\pnorm{\overline{z}^{(s)}}{}^2\Big)\leq \big(K\Lambda(1+\pnorm{A}{\op})\big)^{ct} \big(1+\pnorm{z^{(0)}}{\infty}\big)^c,
	\end{align*}
	and that for any two symmetric matrices $A_1,A_2 \in \R^{n\times n}$,
	\begin{align*}
	\abs{H(A_1)-H(A_2)}&\leq \big(K\Lambda(1+\pnorm{A_1}{\op}+\pnorm{A_2}{\op})\big)^{ct} \big(1+\pnorm{z^{(0)}}{\infty}\big)^c\cdot \pnorm{A_1-A_2}{\op}.
	\end{align*}
	Now we may apply the Gaussian concentration Lemma \ref{lem:gaussian_conc} to conclude.
\end{proof}

\subsection{Proofs of Theorems \ref{thm:AMP_asym_loo}, \ref{thm:AMP_asym} and \ref{thm:AMP_asym_avg}}\label{subsection:proof_AMP_asym}

We will reduce the asymmetric AMP iterate (\ref{def:AMP_asym}) into the symmetric AMP iterate (\ref{def:AMP}). The reduction scheme is well understood, cf. \cite[Section 6]{berthier2020state}; we spell out some details for the convenience of the reader. Let $z^{(-1)}\equiv \binom{u^{(0)}}{0_n}, z^{(0)}\equiv \binom{0_m}{v^{(0)}}\in \R^{m+n}$ (recall $z^{(-1)}$ is a dummy variable; the algorithm is initialized at $z^{(0)}$), and for $t=1,2,\cdots$, let 
\begin{align*}
z^{(2t-1)}\equiv \binom{u^{(t)}}{0_n},\quad   z^{(2t)}\equiv \binom{0_m}{v^{(t)}}\in \R^{m+n}.
\end{align*}
Let $\overline{\mathsf{F}}_{2t-1},\overline{\mathsf{F}}_{2t}: \R^{m+n}\to \R^{m+n}$ be defined via
\begin{align*}
\overline{\mathsf{F}}_{2t-1}(z^{(2t-1)})&\equiv \binom{\mathsf{G}_{t}(u^{(t)})}{0_n},\quad \overline{\mathsf{F}}_{2t}(z^{(2t)})\equiv \binom{0_m}{ \mathsf{F}_t(v^{(t)}) }.
\end{align*}
Furthermore, recall with $\phi=m/n$, we let
\begin{align*}
\overline{V}\equiv  (\phi^{-1}+1)^{1/2}
\begin{pmatrix}
0_{m\times m} & V\\
V^\top & 0_{n\times n}
\end{pmatrix}
\in \R^{(m+n)\times (m+n)}. 
\end{align*}
Using these notation, we may rewrite the asymmetric AMP iterate (\ref{def:AMP_asym}) into the standard symmetric AMP iterate (\ref{def:AMP}):
\begin{align*}
z^{(2t+1)}& = (\overline{V}\circ G_{m+n}) \overline{\mathsf{F}}_{2t}(z^{(2t)})- b_{2t}(\overline{V})\circ  \overline{\mathsf{F}}_{2t-1}(z^{(2t-1)}),\\
z^{(2t+2)}& = (\overline{V}\circ G_{m+n}) \overline{\mathsf{F}}_{2t+1}(z^{(2t+1)})- b_{2t+1}(\overline{V})\circ \overline{\mathsf{F}}_{2t}(z^{(2t)}). 
\end{align*}
Here for $k \in [m+n]$ and $s \in \N$, $
b_{s,k}(\overline{V})=\frac{1}{m+n}\sum_{\ell \in [m+n]} \overline{V}_{k\ell}^2 \E \overline{\mathsf{F}}_{s,\ell}'(z^{(s)}_\ell)$. It is easy to check that, when $s=2t$, $b_s(\overline{V})=\binom{b_t^{\mathsf{F}}}{0_n}$, and when $s=2t+1$, $b_s(\overline{V})=\binom{0_m}{b_{t+1}^{\mathsf{G}}}$. The state evolution in Definition \ref{def:AMP_asym_se} follows from Definition \ref{def:AMP_se} and the above identification. We may now apply Theorems \ref{thm:AMP_sym_loo}, \ref{thm:AMP_sym} and \ref{thm:AMP_sym_avg} to conclude. \qed

\subsection{Proof of Theorem \ref{thm:AMP_oracle_data_diff}}\label{subsection:proof_AMP_oracle_data_diff}

\subsubsection{Some further notation} 
Let
\begin{align*}
\Delta \hat{z}^{(t)}&\equiv \hat{z}^{(t)}-z^{(t)},\quad \hat{z}^{(t)}(u) \equiv (1-u) z^{(t)}+u  \hat{z}^{(t)},\quad u \in [0,1].
\end{align*}
Recall for any vector $x \in \R^n$, $\mathfrak{D}_x=(x_i\bm{1}_{i=j})_{i,j\in [n]}$. We let
\begin{align*}
(\hat{\mathsf{F}}')^{(t)}\equiv \E_U \mathsf{F}'_t\big(\hat{z}^{(t)}(U)\big) \in \R^n,\quad \hat{\mathsf{D}}^{(t)}\equiv \mathfrak{D}_{(\hat{\mathsf{F}}')^{(t)}}\in \R^{n\times n},
\end{align*}
with $\E_U$ being the expectation taken over $U \sim \text{Unif}[0,1]$ independent of $A$.
For notational convenience, we also write $\hat{\mathsf{B}}_t\equiv \mathfrak{D}_{\hat{b}_t}\in \R^{n\times n}$, $\Delta \hat{\mathsf{B}}_t\equiv \hat{\mathsf{B}}_t-{\mathsf{B}}_t$, and $\Delta\hat{b}_{t,k} \equiv \hat{b}_{t,k} - b_{t,k}$ for $k \in [n]$. Using these notation, a first-order Taylor expansion yields that for $r=1,2,\ldots$,
\begin{align}\label{eqn:delta_z_hatz_recursion}
\Delta \hat{z}^{(r)}& = A \hat{\mathsf{D}}^{(r-1)} \Delta \hat{z}^{(r-1)}-\mathsf{B}_{r-1} \hat{\mathsf{D}}^{(r-2)}\Delta \hat{z}^{(r-2)}-\Delta \hat{\mathsf{B}}_{r-1} \mathsf{F}_{r-2}(\hat{z}^{(r-2)}).
\end{align}
For any given $t$, we define a sequence of matrices $\hat{M}_{0}^{(t)} ,\hat{M}_{1}^{(t)},\ldots, \hat{M}_{t}^{(t)}  \in \R^{n\times n}$ recursively as follows: let  $\hat{M}_{-1}^{(t)} =0_{n\times n}$, $\hat{M}_{0}^{(t)} \equiv \hat{\mathsf{D}}^{(t)}$, and for $s \in [t]$, 
\begin{align}\label{def:M_hat_recursion}
\hat{M}_{s}^{(t)} \equiv \big(\hat{M}_{s-1}^{(t)}  A - \hat{M}_{s-2}^{(t)}  \mathsf{B}_{t-s+1}\big) \hat{\mathsf{D}}^{(t-s)}. 
\end{align}
For a fixed $s \in [t]$, we consider another sequence of matrices $\hat{N}_0^{(s)},\ldots, \hat{N}_s^{(s)} \in \R^{n\times n}$, defined recursively as follows: $\hat{N}_{-1}^{(s)}\equiv 0_{n\times n}, \hat{N}_0^{(s)}\equiv I_n$, and for $u \in [s]$,
\begin{align}\label{def:N_hat_recursion}
\hat{N}_u^{(s)}\equiv  A  \mathsf{\hat{D}}^{(t-s+u)} \hat{N}_{u-1}^{(s)} - \mathsf{B}_{t-s+u} \mathsf{\hat{D}}^{(t-s+u-1)} \hat{N}_{u-2}^{(s)}.
\end{align}

\subsubsection{The error decomposition}
The following lemma provides an analogue of Proposition \ref{prop:AMP_loo_decom}.
\begin{lemma}\label{lem:AMP_oracle_data_diff_loo}
	Suppose the conditions in Theorem \ref{thm:AMP_oracle_data_diff} hold.
	Then there exists some universal constant $c_0>0$ such that 
	\begin{align*}
	\abs{\Delta \hat{z}^{(t+1)}_k}\leq (K\Lambda)^{c_0 t} \cdot \bigg[\max_{1\leq r\leq s\leq t}\bigabs{ \widehat{\rem}_{r;k}^{(s)}}+\max_{s \in [0:t]}\big(1+\pnorm{\hat{z}^{(s)}}{\infty}\big)\cdot \max_{s \in [t]}\abs{\Delta \hat{b}_{s,k}}\bigg].
	\end{align*}
	Here $\widehat{\rem}_{s;k}^{(t)}\equiv \iprod{A_k}{ \hat{M}_{t-s}^{(t)} \Delta \hat{\mathsf{B}}_{s-1} \mathsf{F}_{s-2}(\hat{z}^{(s-2)})}$.
\end{lemma}
\begin{proof}
	For any $k \in [n]$, (\ref{eqn:delta_z_hatz_recursion}) with $r=t+1$ yields that
	\begin{align*}
	\Delta \hat{z}^{(t+1)}_k&=\iprod{A_k}{ \hat{\mathsf{D}}^{(t)} \Delta \hat{z}^{(t)} }- b_{t,k}\cdot \hat{\mathsf{D}}^{(t-1)}_{kk} \Delta \hat{z}^{(t-1)}_k- \Delta \hat{b}_{t,k}\cdot \mathsf{F}_{t-1,k}(\hat{z}^{(t-1)}_k).
	\end{align*}
	Next we apply (\ref{eqn:delta_z_hatz_recursion}) with $r=t$ to the first term above to obtain
	\begin{align*}
	\iprod{A_k}{ \hat{\mathsf{D}}^{(t)} \Delta \hat{z}^{(t)} }&= \bigiprod{A_k}{\hat{M}_1^{(t)} \Delta \hat{z}^{(t-1)}- \hat{M}_0^{(t)}\mathsf{B}_{t-1} \hat{\mathsf{D}}^{(t-2)}\Delta \hat{z}^{(t-2)} }-\widehat{\rem}_{t;k}^{(t)}. 
	\end{align*}
	Now iterating this procedure until (\ref{eqn:delta_z_hatz_recursion}) with $r=1$, using the initial condition $\Delta \hat{z}^{(-1)}=\Delta \hat{z}^{(0)}=0$, we have $\iprod{A_k}{ \hat{\mathsf{D}}^{(t)} \Delta \hat{z}^{(t)} }=-\sum_{s \in [t]} \widehat{\rem}_{s;k}^{(t)}$. Consequently, 
	\begin{align*}
	\abs{\Delta \hat{z}^{(t+1)}_k}&\leq \biggabs{\sum_{s \in [t]} \widehat{\rem}_{s;k}^{(t)}}+(K\Lambda)^2 \cdot \Big[\abs{\Delta \hat{z}^{(t)}_k}+ (1+\pnorm{\hat{z}^{(t-1)}}{\infty})\cdot \abs{\Delta \hat{b}_{t,k}}\Big].
	\end{align*}
	Iterating the bound to conclude. 
\end{proof}

Below we shall control each term on the right hand side of the claimed inequality in the above lemma. First we provide controls for $\pnorm{\Delta \hat{b}_t}{\infty}$.

\subsubsection{Some apriori estimates}
\begin{lemma}\label{lem:AMP_oracle_data_diff_Delta_b}
	Suppose the conditions in Theorem \ref{thm:AMP_oracle_data_diff} hold. For any $D >0$, there exist some universal constant $c_0 >0$ and another constant $c_1 = c_1(D)$, such that with $\Prob(\cdot |z^{(0)})$-probability at least $1-c_1 n^{-D}$, 
	\begin{align*}
	\pnorm{\Delta \hat{b}_t}{\infty} \leq  \big(c_1 K\Lambda \log n \cdot ( 1+ \pnorm{z^{(0)}}{\infty} )\big)^{c_0t^3} \cdot  n^{-1/2}. 
	\end{align*}
	The moment estimate $\E^{1/p}\pnorm{\Delta \hat{b}_t}{\infty}^p$ holds similarly with the constants depending further on $p$.
\end{lemma}
\begin{proof}
	Note that for $k \in [n]$, 
	\begin{align*}
	\abs{\Delta\hat{b}_{t,k}} \leq \biggabs{\frac{1}{n}\sum_{\ell \in [n]} V_{k\ell}^2 (\mathrm{id} - \E ) \mathsf{F}_{t,\ell}'(z_{\ell}^{(t)}) } + (K\Lambda)^2 \cdot n^{-1/2} \pnorm{\Delta \hat{z}^{(t)}}{ } \equiv I_1 + I_2.
	\end{align*}
	For $I_2$, using the relation (\ref{eqn:delta_z_hatz_recursion}), we have
	\begin{align*}
	\pnorm{\Delta \hat{z}^{(t)}}{}&\leq (K\Lambda)^2(1+\pnorm{A}{\op})\cdot \Big[\max_{s \in \{t-1,t-2\}} \pnorm{\Delta \hat{z}^{(s)}}{}\\
	&\qquad\qquad + \max_{s \in [t-1]} n^{1/2}\pnorm{\Delta \hat{b}_s}{\infty}\cdot (1+n^{-1/2}\pnorm{\hat{z}^{(t-2)}}{})\Big].
	\end{align*}
	Using the simple apriori estimate $\pnorm{\hat{z}^{(t)}}{}\leq \big(K\Lambda (1+\pnorm{A}{\op})\big)^{c_0 t}( \sqrt{n} + \pnorm{z^{(0)}}{} )$, we then have 
	\begin{align*}
	\pnorm{\Delta \hat{z}^{(t)}}{}&\leq \big(K \Lambda (1+\pnorm{A}{\op})\big)^{c_0 t}(1 + \pnorm{z^{(0)}}{\infty})\cdot \Big[\max_{s \in \{t-1,t-2\}} \pnorm{\Delta \hat{z}^{(s)}}{}+ \max_{s \in [t-1]} n^{1/2}\pnorm{\Delta \hat{b}_s}{\infty} \Big].
	\end{align*}
	Iterating the bound and using the initial condition $\Delta \hat{z}^{(t)}=0$, 
	\begin{align*}
	I_2=(K\Lambda)^2\cdot n^{-1/2}\pnorm{\Delta \hat{z}^{(t)}}{}&\leq \big(K\Lambda (1+\pnorm{A}{\op})(1 + \pnorm{z^{(0)}}{\infty})\big)^{c_0 t^2}\cdot  \max_{s \in [t-1]} \pnorm{\Delta \hat{b}_s}{\infty}.
	\end{align*}
	For $I_1$, let $H_k(A)\equiv n^{-1}\sum_{\ell \in [n]} V_{k\ell}^2\mathsf{F}_{t,\ell}' \big(z^{(t)}_\ell(A)\big)$. Then it is easy to estimate
	\begin{align*}
	\abs{H_k(A)}&\leq (K\Lambda)^2\big(1+n^{-1/2}\pnorm{z^{(t)}(A)}{}\big)\leq \big(K\Lambda (1+\pnorm{A}{\op})\big)^{c_0 t}( 1+ \pnorm{z^{(0)}}{\infty} ).
	\end{align*}
	Moreover, for any two symmetric matrices $A_1, A_2 \in \R^{n\times n}$,
	\begin{align*}
	\abs{H_k(A_1)-H_k(A_2)}&\leq (K\Lambda)^2\cdot n^{-1/2}\pnorm{z^{(t)}(A_1)-z^{(t)}(A_2)}{}\\
	&\leq \big(K\Lambda(1+\pnorm{A_1}{\op} +\pnorm{A_2}{\op}  )\big)^{c_0t} \cdot \pnorm{A_1 - A_2}{\op}.
	\end{align*}
	Now applying the Gaussian concentration in Lemma \ref{lem:gaussian_conc}, we have for $t\leq \log n /c_0 $, with  $\Prob(\cdot |z^{(0)})$-probability at least $1-c_1 n^{-D}$,
	\begin{align*}
	I_1 \leq \big(c_1 K\Lambda \log n \cdot ( 1+ \pnorm{z^{(0)}}{\infty} )\big)^{c_0t} \cdot n^{-1/2}.
	\end{align*}
	Combining the above estimate, by adjusting constants, with probability at least $1-c_1 n^{-D}$, 
	\begin{align*}
	\pnorm{\Delta \hat{b}_{t}}{\infty}&\leq  \big(c_1 K\Lambda \log n \cdot ( 1+ \pnorm{z^{(0)}}{\infty} )\big)^{c_0t^2} \cdot \Big(n^{-1/2}  + \max_{s \in [t-1]} \pnorm{\Delta \hat{b}_s}{\infty} \Big).
	\end{align*}
	From (\ref{eqn:delta_z_hatz_recursion}), we have $\Delta \hat{z}^{(1)} = 0$. Then it is easy to  establish that with the same probability, the initial estimate $\pnorm{\Delta \hat{b}_{1}}{\infty} \leq (c_1 K\Lambda \log n \cdot ( 1+ \pnorm{z^{(0)}}{\infty} ))^{c_0} \cdot n^{-1/2}$, and therefore the claim follows by iterating the bound. 
\end{proof}

Next we provide delocalization estimates for $\hat{z}^{(t)}$ and an $\ell_2$ control for $\pnorm{\partial_{ij} \hat{z}^{(t)}}{}$, in similar spirit to those presented in Propositions \ref{prop:AMP_iterate_infty} and \ref{prop:z_der_l2}.

\begin{lemma}\label{lem:AMP_oracle_data_diff_deloc}
	Suppose the conditions in Theorem \ref{thm:AMP_oracle_data_diff} hold.  Then there exists some universal constant $c_0>0$ such that for $x\geq 1$, 
	\begin{align*}
	\max_{k \in [n]}\Prob\Big(\pnorm{\hat{z}^{(t)}}{\infty}\vee \pnorm{\partial_{ij} \hat{z}^{(t)} }{} \geq (K\Lambda\cdot (1+\pnorm{z^{(0)}}{\infty}))^{c_0 t^2} \cdot x\big| z^{(0)}\Big)\leq c_0\cdot tn e^{-(x^2\wedge n)/c_0}. 
	\end{align*}
	Moreover, for any $p\geq 1$, there exists some $C_p>0$ such that for $t\leq n/(C_p\log n)$,
	\begin{align*}
	\E^{1/p}\big[\pnorm{\hat{z}^{(t)}}{\infty}^p|z^{(0)}\big]\vee \max_{k \in [n]}\E^{1/p}\big[\pnorm{\partial_{ij} \hat{z}^{(t)} }{}^p|z^{(0)}\big]\leq (K\Lambda\cdot (1+\pnorm{z^{(0)}}{\infty}))^{C_p t}\sqrt{\log n} .
	\end{align*}
	
\end{lemma}
\begin{proof}
	\noindent (1). The proof adopts a similar strategy to that of Proposition \ref{prop:AMP_iterate_infty}. Consider the following leave-one-out version of $\hat{z}^{(t+1)}$: for any $k \in [n]$,
	\begin{align}\label{def:AMP_loo_data}
	\hat{z}_{[-k]}^{(t+1)}& \equiv  A_{[-k]}\mathsf{F}_t(\hat{z}_{[-k]}^{(t+1)})-\hat{b}_{t,[-k]}\circ \mathsf{F}_{t-1}(\hat{z}^{(t-1)}_{[-k]}),
	\end{align}
	where 
	\begin{itemize}
		\item $\hat{z}^{(-1)}_{[-k]}=0_n$, $\hat{z}^{(0)}_{[-k]}\equiv z^{(0)}$, and 
		\item for $t\geq 1$, $\hat{b}_{t,[-k]}$ is defined by replacing $\hat{z}^{(t)}_\ell$ with $\hat{z}^{(t)}_{[-k],\ell}$ in (\ref{def:hat_b}).
	\end{itemize}
	With $\Delta \hat{z}_{[-k]}^{(t)}\equiv \hat{z}^{(t)}-\hat{z}^{(t)}_{[-k]}$, $\Delta \hat{b}_{t,[-k]}\equiv \hat{b}_{t}-\hat{b}_{t,[-k]}$, and recall $\Delta A_{[-k]}=A-A_{[-k]}$, we have the estimate
	\begin{align*}
	\pnorm{\Delta\hat{z}^{(t+1)}_{[-k]}}{}&\leq  \pnorm{\Delta A_{[-k]} \mathsf{F}_t(\hat{z}_{[-k]}^{(t)}) }{}+ \Lambda\cdot  \pnorm{A}{\op} \pnorm{\Delta\hat{z}^{(t)}_{[-k]}}{}\nonumber\\
	&\quad +\Lambda\cdot \sqrt{n}\abs{\Delta \hat{b}_{t,[-k]}}\cdot \big(1+n^{-1/2}\pnorm{\hat{z}^{(t-1)} }{}\big)+\Lambda\cdot \abs{\hat{b}_{t,[-k]}}\cdot \pnorm{\Delta\hat{z}^{(t-1)}_{[-k]}}{}.
	\end{align*}
	By noting that  $\pnorm{\hat{z}^{(t-1)}}{}\leq \big(K\Lambda (1+\pnorm{A}{\op})\big)^{c t}( \sqrt{n} + \pnorm{z^{(0)}}{} )$ and $\abs{\Delta \hat{b}_{t,[-k]}}\leq K^2\Lambda\cdot n^{-1/2}\pnorm{\Delta \hat{z}^{(t)}_{[-k]}}{}$,
	\begin{align*}
	\pnorm{\Delta\hat{z}^{(t+1)}_{[-k]}}{}&\leq  \pnorm{A_k^\top \mathsf{F}_t(\hat{z}_{[-k]}^{(t)}) }{}  +  2\Lambda \pnorm{A_k  }{} \cdot (1 + \hat{z}_{k,[-k]}^{(t)})  \\
	&\quad + \big(K\Lambda (1 +\pnorm{A}{\op}) \big)^{c_0t}(1+\pnorm{z^{(0)}}{\infty})\cdot  \max_{s \in [0:t]} \pnorm{\Delta\hat{z}^{(s)}_{[-k]}}{}.
	\end{align*}
	On the other hand, using (\ref{def:AMP_loo_data}),
	\begin{itemize}
		\item $\abs{\hat{z}_{k,[-k]}^{(t)}}\leq (K\Lambda)^2(1+\abs{\hat{z}_{k,[-k]}^{(t-2)}})\leq \cdots\leq (K\Lambda)^{c_0 t}(1+\abs{{z}_k^{(0)}})$,
		\item $\pnorm{\hat{z}_{[-k]}^{(t)}}{}\leq (K\Lambda)^2 (1+\pnorm{A_{[-k]}}{\op})\big(\sqrt{n}+\pnorm{\hat{z}_{[-k]}^{(t-1)}}{}\vee \pnorm{\hat{z}_{[-k]}^{(t-2)}}{}\big)\leq\cdots \leq \big(K\Lambda(1+\pnorm{A_{[-k]}}{\op})\big)^{c_0 t}\big(\sqrt{n}+\pnorm{z^{(0)}}{}\big)$. 
	\end{itemize}
	Using the independence of $A_k$ and $\mathsf{F}_{t}(\hat{z}^{(t)}_{[-k]})$ and the standard subgaussian tail estimate for $\pnorm{A}{\op},\pnorm{A_{[-k]}}{\op}$ (cf. Lemma \ref{lem:A_op_subgaussian_tail}), for $x\geq 1$, with probability at least $1-c_0 e^{-(x^2\wedge n)/c_0}$, both $\sqrt{n}\abs{A_k^\top \mathsf{F}_t(\hat{z}_{[-k]}^{(t)} )}\leq c_0Kx\cdot \pnorm{\mathsf{F}_t(\hat{z}_{[-k]}^{(t)})}{}$ and $\pnorm{A}{\op}\vee \pnorm{A_{[-k]}}{\op}\leq c_0K$ hold. So with the same probability,
	\begin{align*}
	\pnorm{\Delta\hat{z}^{(t+1)}_{[-k]}}{}&\leq (K\Lambda)^{c_0t}(1+\pnorm{z^{(0)}}{\infty})\cdot \Big(\max_{s \in [0:t]} \pnorm{\Delta\hat{z}^{(s)}_{[-k]}}{} + x \Big).
	\end{align*}
	Iterating the bound yields, with probability at least $1-c_0 te^{-(x^2\wedge n)/c_0}$,
	\begin{align*}
	\pnorm{\Delta\hat{z}^{(t+1)}_{[-k]}}{}&\leq (K\Lambda \cdot (1 + \pnorm{z^{(0)}}{\infty}) )^{c_0 t^2} \cdot x. 
	\end{align*}
	High probability control for $\hat{z}^{(t+1)}_k$ now follows from the simple estimate $\abs{\hat{z}^{(t+1)}_k}\leq \pnorm{\Delta\hat{z}^{(t+1)}_{[-k]}}{}+\abs{\hat{z}^{(t+1)}_{[-k],k}}$. Moment controls follows from the trivial apriori estimate. 
	
	\noindent (2). The proof shares some common structure to that of Proposition \ref{prop:z_der_l2}. Note that for $t\geq 1$, we have $
	\pnorm{\partial_{ij} \hat{b}_{t}}{\infty}\leq  K^2\Lambda\cdot n^{-1/2}\pnorm{\partial_{ij} \hat{z}^{(t)}}{}$. So we may derive
	\begin{align*}
	&\pnorm{\partial_{ij} \hat{z}^{(t+1)}}{}\leq \abs{\mathsf{F}_{t,i}(\hat{z}^{(t)}_i)}+ \abs{\mathsf{F}_{t,j}(\hat{z}^{(t)}_j)} +\Lambda \pnorm{A}{\op}\pnorm{\partial_{ij} \hat{z}^{(t)}}{} + (K\Lambda)^2\pnorm{\partial_{ij} \hat{z}^{(t-1)}}{} \\
	&\qquad\qquad\qquad  + (K\Lambda)^2 (1 + n^{-1/2} \pnorm{\hat{z}^{(t-1)}}{} )\cdot \pnorm{\partial_{ij} \hat{z}^{(t)}}{} \\
	&\leq \big(K\Lambda (1 + \pnorm{A}{\op}) \big)^{2}(1 +  n^{-1/2} \pnorm{\hat{z}^{(t-1)}}{}) \cdot (\pnorm{\partial_{ij} \hat{z}^{(t)}}{} + \pnorm{\partial_{ij} \hat{z}^{(t-1)}}{} )+2\Lambda (1+\pnorm{\hat{z}^{(t)}}{\infty}).
	\end{align*}
	Iterating the bound and using the initial condition that $\pnorm{\partial_{ij} \hat{z}^{(0)}}{} = 0$, we have
	\begin{align*}
	\pnorm{\partial_{ij} \hat{z}^{(t+1)}}{} &\leq \Big[K\Lambda (1 + \pnorm{A}{\op}) \cdot \Big(1 + n^{-1/2} \max_{s \in [t-1]}\pnorm{\hat{z}^{(s)}}{}  \Big) \Big]^{c_0t} \cdot \Big(1+ \max_{s\in [t]}\pnorm{\hat{z}^{(s)}}{\infty}\Big)\\
	&\leq \big(K\Lambda (1 + \pnorm{A}{\op}) (1 + \pnorm{z^{(0)}}{\infty} ) \big)^{c_0t^2}\cdot \Big(1+ \max_{s\in [t]}\pnorm{\hat{z}^{(s)}}{\infty}\Big). 
	\end{align*}
	The high probability control then follows from (1). The moment estimate follows easily. 
\end{proof}

\subsubsection{Error controls of $\widehat{\rem}_{s;k}^{(t)}$}
Finally, we provide the most difficult control for $\widehat{\rem}_{s;k}^{(t)}$.
\begin{lemma}\label{lem:AMP_oracle_data_diff_rem}
	Suppose the conditions in Theorem \ref{thm:AMP_oracle_data_diff} hold. Then for any $p \in \N $, there exists some universal constant $c_0>0$ such that for any $s \in [t], k \in [n]$,
	\begin{align*}
	\E^{1/2p} \big(\widehat{\rem}_{s;k}^{(t)}\big)^{2p}&\leq (K\Lambda )^{c_0 t} \cdot \hat{\Theta}_p^{(t)}\cdot n^{-1/4}.
	\end{align*}
	Here for some universal constant $c>0$, $\hat{\Theta}_p^{(t)}\equiv \E^{1/cp}\big[\hat{\mathfrak{Z}}^{(t)} \hat{\mathfrak{B}}^{(t)}\hat{\mathfrak{M}}^{(t)}\big]^{c^2 p}$ with the involved quantities defined as
	\begin{itemize}
		\item $\hat{\mathfrak{B}}^{(t)}\equiv \max\limits_{s \in [t]}n^{1/2}\pnorm{\Delta \hat{\mathsf{B}}_s }{\op}$,
		\item $\hat{\mathfrak{Z}}^{(t)}\equiv \max\limits_{s \in [t]}\max\limits_{i,j \in [n]} \big(1+\pnorm{z^{(s)}}{\infty}+\pnorm{ \hat{z}^{(s)}}{\infty}+\pnorm{\partial_{ij} z^{(s)} }{}+\pnorm{\partial_{ij} \hat{z}^{(s)}}{}\big)$,
		\item $\hat{\mathfrak{M}}^{(t)}\equiv \max\limits_{s \in [t]} \max\limits_{r \in [s]} \big(1+\pnorm{\hat{M}^{(s)}_r}{\op}+\pnorm{\hat{N}^{(s)}_r}{\op}\big)$. 
	\end{itemize}
\end{lemma}
\begin{proof}
	For notational simplicity, we write $\hat{\mathscr{R}}\equiv  \widehat{\rem}_{s;k}^{(t)}$, ${\mathsf{F}}_{s}\equiv \mathsf{F}_{s}(\hat{z}^{(s)})$, and ${\mathsf{F}}_{s}'\equiv \mathsf{F}_{s}'(\hat{z}^{(s)})$. Below the universal constant $c>0$ may change from line to line. Using Gaussian integration by parts, 
	\begin{align*}
	\E \hat{\mathscr{R}}^{2p} & = \E \iprod{A_k}{ \hat{M}_{t-s}^{(t)} \Delta \hat{\mathsf{B}}_{s-1} \mathsf{F}_{s-2} }\cdot  \hat{\mathscr{R}}^{2p-1}\\
	& = \frac{1}{n}\E\sum_{i\in [n]}V_{ki}^2 \cdot \partial_{ki}\big[\iprod{e_i}{\hat{M}_{t-s}^{(t)} \Delta \hat{\mathsf{B}}_{s-1} \mathsf{F}_{s-2}} \cdot \hat{\mathscr{R}}^{2p-1} \big]\equiv \sum_{\ell \in [4]} I_\ell. 
	\end{align*}
	
	\noindent (\textbf{\emph{Term $I_1$}}): We have the derivative formula: for $r \in [t]$,
	\begin{align}\label{eqn:AMP_practical_M_der}
	\partial_{ki} \hat{M}_r^{(t)}& = \big(\partial_{ki}\hat{M}_{r-1}^{(t)}  A - \partial_{ki}\hat{M}_{r-2}^{(t)}  \mathsf{B}_{t-r+1}\big) \hat{\mathsf{D}}^{(t-r)}+\hat{M}_{r-1}^{(t)}\Delta_{ki} \hat{\mathsf{D}}^{(t-r)}\nonumber\\
	&\qquad + \big(\hat{M}_{r-1}^{(t)}  A - \hat{M}_{r-2}^{(t)}  \mathsf{B}_{t-r+1}\big) \partial_{ki}\hat{\mathsf{D}}^{(t-r)}.
	\end{align}
	Let for $r \in [t-s]$ 
	\begin{align*}
	\hat{P}^{(t)}_r&\equiv \frac{1}{n}\E\sum_{i \in [n]} V_{ki}^2 e_i^\top  \hat{M}_{r-1}^{(t)}\Delta_{ki} \hat{\mathsf{D}}^{(t-r)} \hat{N}^{(t-s+1)}_{t-s-r} \Delta\hat{\mathsf{B}}_{s-1}\mathsf{F}_{s-2}     \cdot \hat{\mathscr{R}}^{2p-1},\\
	\hat{R}^{(t)}_r&\equiv \frac{1}{n}\E\sum_{i \in [n]} V_{ki}^2 e_i^\top \big(\hat{M}_{r-1}^{(t)}  A - \hat{M}_{r-2}^{(t)}  \mathsf{B}_{t-r+1}\big) \big(\partial_{ki}\hat{\mathsf{D}}^{(t-r)}\big) \hat{N}^{(t-s+1)}_{t-s-r} \Delta\hat{\mathsf{B}}_{s-1}\mathsf{F}_{s-2} \cdot \hat{\mathscr{R}}^{2p-1}.
	\end{align*}
	Then applying (\ref{eqn:AMP_practical_M_der}) recursively from  $r=t-s$ down to $r=1$, we have 
	\begin{align*}
	I_1& = \frac{1}{n}\E\sum_{i \in [n]} V_{ki}^2 \big(\partial_{ki} \hat{M}^{(t)}_{0} \hat{N}^{(t-s+1)}_{t-s} \Delta \hat{\mathsf{B}}_{s-1} \mathsf{F}_{s-2}\big)_{i} \cdot \hat{\mathscr{R}}^{2q-1}\\
	&\qquad + \sum_{r \in [t-s]} \hat{P}^{(t)}_r+\sum_{r \in [t-s]} \hat{R}^{(t)}_r\equiv I_{1,a}+I_{1,b}+I_{1,c}. 
	\end{align*}
	By noting that $\partial_{ki} \hat{M}^{(t)}_{0}$ is diagonal for $I_{1,a}$ and using trivial bounds for $I_{1,b},I_{1,c}$, we may produce the bound
	\begin{align*}
	\abs{I_1}&\leq (K\Lambda)^c t\cdot n^{-1/2} \E (1+\pnorm{A}{\op}) (\hat{\mathfrak{Z}}^{(t)})^2  (\hat{\mathfrak{M}}^{(t)})^2  \hat{\mathfrak{B}}^{(t)}\cdot \abs{\hat{\mathscr{R}}}^{2p-1}.
	\end{align*}
	\noindent (\textbf{\emph{Term $I_2$}}): As $\partial_{ki} \hat{b}_s = n^{-1} (V \circ V)(\mathsf{F}_s''\circ \partial_{ki} \hat{z}^{(s)})$, a simple calculation shows $\pnorm{\partial_{ki} \hat{b}_s}{}\leq K^2\Lambda\cdot \pnorm{\partial_{ki} \hat{z}^{(s)}}{}$, and therefore
	\begin{align*}
	\abs{I_2}& = \biggabs{\frac{1}{n}\E \sum_{i \in [n]} V_{ki}^2 \bigiprod{e_i}{ \hat{M}^{(t)}_{t-s} \mathfrak{D}_{\partial_{ki} \hat{b}_{s-1} } \mathsf{F}_{s-2} } \cdot \hat{\mathscr{R}}^{2p-1}}\\
	&\leq (K\Lambda)^c \cdot n^{-1/2} \E (\hat{\mathfrak{Z}}^{(t)})^2 \hat{\mathfrak{M}}^{(t)} \cdot \abs{\hat{\mathscr{R}}}^{2q-1}.
	\end{align*}
	\noindent (\textbf{\emph{Term $I_3$}}): We have
	\begin{align*}
	\abs{I_3}&= \biggabs{\frac{1}{n} \E \sum_{i \in [n]} V_{ki}^2 \bigiprod{e_i}{ \hat{M}^{(t)}_{t-s}\Delta \hat{\mathsf{B}}_{s-1} (\mathsf{F}_{s-2}'\circ \partial_{ki} \hat{z}^{s-2})  } \cdot  \hat{\mathscr{R}}^{2p-1}  }\\
	&\leq (K\Lambda)^c \cdot n^{-1/2} \E  \hat{\mathfrak{Z}}^{(t)} \hat{\mathfrak{M}}^{(t)} \hat{\mathfrak{B}}^{(t)}\cdot \abs{\hat{\mathscr{R}}}^{2p-1}.
	\end{align*}
	
	\noindent (\textbf{\emph{Term $I_4$}}): 
	Note that
	\begin{align*}
	\partial_{ki} \hat{\mathscr{R}}& = \iprod{e_i}{ \hat{M}_{t-s}^{(t)} \Delta \hat{\mathsf{B}}_{s-1} \mathsf{F}_{s-2}}+\iprod{A_k}{ (\partial_{ki}\hat{M}_{t-s}^{(t)}) \Delta \hat{\mathsf{B}}_{s-1} \mathsf{F}_{s-2}}\\
	&\quad + \iprod{A_k}{ \hat{M}_{t-s}^{(t)} (\partial_{ki} \hat{\mathsf{B}}_{s-1}) \mathsf{F}_{s-2}}+ \iprod{A_k}{ \hat{M}_{t-s}^{(t)} \Delta \hat{\mathsf{B}}_{s-1} (\mathsf{F}_{s-2}'\circ \partial_{ki} \hat{z}^{(s-2)})},
	\end{align*}
	we may write $I_4\equiv (2p-1)\big(I_{4,a}+I_{4,b}+I_{4,c}+I_{4,d}\big)$ according to the terms above. For $I_{4,a}$, we have
	\begin{align*}
	\abs{I_{4,a}}& = \biggabs{\frac{1}{n}\E \sum_{i \in [n]} V_{ki}^2 \big(\hat{M}_{t-s}^{(t)} \Delta \hat{\mathsf{B}}_{s-1} \mathsf{F}_{s-2}\big)_i^2\cdot   \hat{\mathscr{R}}^{2p-2} }\\
	&\leq (K\Lambda)^c\cdot n^{-1} \E \big(\hat{\mathfrak{Z}}^{(t)}\hat{\mathfrak{M}}^{(t)} \hat{\mathfrak{B}}^{(t)}\big)^2\cdot  \hat{\mathscr{R}}^{2p-2}. 
	\end{align*}
	For $I_{4,b}$, let for $r \in [t-s]$ 
	\begin{align*}
	\hat{P}^{(t),\ast}_r&\equiv \frac{1}{n}\E\sum_{i \in [n]} V_{ki}^2 \big(\hat{M}_{t-s}^{(t)} \Delta \hat{\mathsf{B}}_{s-1} \mathsf{F}_{s-2}\big)_i \bigiprod{A_k}{\big(\hat{M}_{r-1}^{(t)}\Delta_{ki} \hat{\mathsf{D}}^{(t-r)}\big) \hat{N}^{(t-s+1)}_{t-s-r} \Delta \hat{\mathsf{B}}_{s-1} \mathsf{F}_{s-2} }\cdot  \hat{\mathscr{R}}^{2p-2},\\
	\hat{R}^{(t),\ast}_r&\equiv \frac{1}{n}\E\sum_{i \in [n]} V_{ki}^2 \big(\hat{M}_{t-s}^{(t)} \Delta \hat{\mathsf{B}}_{s-1} \mathsf{F}_{s-2}\big)_i\\
	&\quad \times \bigiprod{A_k}{ \big(\hat{M}_{r-1}^{(t)}  A - \hat{M}_{r-2}^{(t)}  \mathsf{B}_{t-r+1}\big) \big(\partial_{ki}\hat{\mathsf{D}}^{(t-r)}\big) \hat{N}^{(t-s+1)}_{t-s-r} \Delta \hat{\mathsf{B}}_{s-1} \mathsf{F}_{s-2}}\cdot  \hat{\mathscr{R}}^{2p-2}.
	\end{align*}
	Then applying the above recursive method, 
	\begin{align*}
	I_{4,b}& = \frac{1}{n}\E\sum_{i} V_{ki}^2 \big(\hat{M}_{t-s}^{(t)} \Delta \hat{\mathsf{B}}_{s-1} \mathsf{F}_{s-2}\big)_i \bigiprod{A_k}{\big(\partial_{ki} \hat{M}_{0}^{(t)}\big)\hat{N}_{t-s}^{(t-s+1)} \Delta \hat{\mathsf{B}}_{s-1} \mathsf{F}_{s-2} }\cdot \hat{\mathscr{R}}^{2p-2}\\
	&\qquad + \sum_{r \in [t-s]} \hat{P}^{(t),\ast}_r+\sum_{r \in [t-s]} \hat{R}^{(t),\ast}_r\equiv I_{4,b;1}+I_{4,b;2}+I_{4,b;3}. 
	\end{align*}
	Using the trivial Cauchy-Schwarz for $\sum_{i\in[n]}$ in $I_{4,b;1}$ and recalling $\partial_{ki} \hat{M}_{0}^{(t)}=\partial_{ki} \hat{\mathsf{D}}^{(t)}$,
	\begin{align*}
	\abs{I_{4,b;1}}&\leq (K\Lambda)^c \cdot n^{-1} \E\Big[ \pnorm{\hat{M}_{t-s}^{(t)} \Delta \hat{\mathsf{B}}_{s-1} \mathsf{F}_{s-2}}{}\cdot \pnorm{A_k}{} \\
	&\qquad \qquad \times \max_{\mathsf{z} \in z^{(t)},\hat{z}^{(t)}}\Big(\sum_i \bigpnorm{ \mathfrak{D}_{\partial_{ki} \mathsf{z}^{(t)}} \hat{N}_{t-s}^{(t-s+1)} \Delta \hat{\mathsf{B}}_{s-1} \mathsf{F}_{s-2} }{}^2\Big)^{1/2}\cdot \hat{\mathscr{R}}^{2q-2}\Big]\\
	&\leq (K\Lambda)^c\cdot n^{-1/2} \E \pnorm{A_k}{} (\hat{\mathfrak{Z}}^{(t)})^3 (\hat{\mathfrak{B}}^{(t)})^2 (\hat{\mathfrak{M}}^{(t)})^2\cdot \hat{\mathscr{R}}^{2p-2}.
	\end{align*}
	For $I_{4,b;2}$, some calculations lead to
	\begin{align*}
	\abs{\hat{P}^{(t),\ast}_r}&\leq (K\Lambda)^c\cdot n^{-1/2} \E \pnorm{A_k}{}\big(\hat{\mathfrak{Z}}^{(t)}\hat{\mathfrak{M}}^{(t)} \hat{\mathfrak{B}}^{(t)}\big)^2\cdot\hat{\mathscr{R}}^{2p-2},
	\end{align*}
	and therefore  $I_{4,b;2}$ is bounded by a $t$ multiple of the above display. For $I_{4,b;3}$,
	\begin{align*}
	\abs{\hat{R}^{(t),\ast}_r}&\leq (K\Lambda)^c\cdot n^{-1} \E\Big[ (1+\pnorm{A}{\op})^2\cdot \hat{\mathfrak{Z}}_0^{(t)} (\hat{\mathfrak{M}}^{(t)})^2 \hat{\mathfrak{B}}^{(t)}\\
	&\qquad\qquad \times  \max_{\mathsf{z} \in z^{(t)},\hat{z}^{(t)}}\Big(\sum_i \bigpnorm{ \mathfrak{D}_{\partial_{ki} \mathsf{z}^{(t)}}  \Delta \hat{\mathsf{B}}_{s-1} \mathsf{F}_{s-2} }{}^2\Big)^{1/2}\cdot \hat{\mathscr{R}}^{2q-2}\bigg]\\
	&\leq (K\Lambda)^c\cdot n^{-1/2} \E(1+\pnorm{A}{\op})^2 (\hat{\mathfrak{Z}}^{(t)})^3 (\hat{\mathfrak{B}}^{(t)})^2 (\hat{\mathfrak{M}}^{(t)})^2 \cdot \hat{\mathscr{R}}^{2p-2}.
	\end{align*} 
	Therefore,
	\begin{align*}
	\abs{I_{4,b}}&\leq  (K\Lambda)^{c}t\cdot n^{-1/2} \E(1+\pnorm{A}{\op})^2 (\hat{\mathfrak{Z}}^{(t)})^3 (\hat{\mathfrak{B}}^{(t)}\hat{\mathfrak{M}}^{(t)})^2 \cdot \hat{\mathscr{R}}^{2p-2}.
	\end{align*}
	For $I_{4,c}$, using $\pnorm{\partial_{ki} \hat{b}_{s-1}}{\infty} \leq (K\Lambda)^c \cdot n^{-1/2} \pnorm{\partial_{ki} \hat{z}^{(s-1)}}{}$,
	\begin{align*}
		\abs{I_{4,c}} &\leq \biggabs{\frac{1}{n} \E \sum_{i \in [n]} V_{ki}^2\big(\hat{M}_{t-s}^{(t)} \Delta \hat{\mathsf{B}}_{s-1} \mathsf{F}_{s-2}\big)_i \iprod{A_k}{ \hat{M}_{t-s}^{(t)} (\partial_{ki} \hat{\mathsf{B}}_{s-1}) \mathsf{F}_{s-2}}\cdot \hat{\mathscr{R}}^{2p-2}  }\\
		&\leq (K\Lambda)^c\cdot n^{-1/2} \E (1 + \pnorm{A}{\op})(\hat{\mathfrak{Z}}^{(t)})^3 \hat{\mathfrak{B}}^{(t)}  (\hat{\mathfrak{M}}^{(t)})^2\cdot \hat{\mathscr{R}}^{2p-2}  .
	\end{align*}
	For $I_{4,d}$, some calculations lead to
	\begin{align*}
		\abs{I_{4,d}} 
		&\leq (K\Lambda)^c \cdot n^{-1/2} \E (1 + \pnorm{A}{\op})(\hat{\mathfrak{Z}}^{(t)} \hat{\mathfrak{B}}^{(t)}  \hat{\mathfrak{M}}^{(t)})^2\cdot \hat{\mathscr{R}}^{2p-2}.
	\end{align*}
	Combining, we find that
	\begin{align*}
		\abs{I_4}&\leq  (K\Lambda)^{c}t\cdot n^{-1/2} \E(1+\pnorm{A}{\op})^2 (\hat{\mathfrak{Z}}^{(t)})^3 (\hat{\mathfrak{B}}^{(t)}\hat{\mathfrak{M}}^{(t)})^2 \cdot \hat{\mathscr{R}}^{2p-2}.
	\end{align*}
	Combining all estimates and using Young's inequality (cf. Lemma \ref{lem:young_ineq}) to conclude. 
\end{proof}

\subsubsection{Proof of Theorem \ref{thm:AMP_oracle_data_diff}}

	The claimed estimate follows from the decomposition in Lemma \ref{lem:AMP_oracle_data_diff_loo} with all the estimates obtained in Lemmas \ref{lem:AMP_oracle_data_diff_Delta_b}-\ref{lem:AMP_oracle_data_diff_rem}, along with the simple apriori estimate $\max_{s \in [t]}\max_{r \in [s]}\big(\pnorm{\hat{M}_r^{(s)}}{\op}\vee \pnorm{\hat{N}_r^{(s)}}{\op}\big)\leq \big(K\Lambda (1+\pnorm{A}{\op})\big)^{c_0 t}$.\qed

\subsubsection{Sketch of the proof of (\ref{ineq:AMP_oracles_diff})}
We shall use notation $\bar{\#}$ for all quantities defined with $\hat{\#}$ in the beginning of this subsection. Note that now $\Delta \bar{\mathsf{B}}_t$, $\Delta \bar{b}_{t,k}=\bar{b}_{t,k}-\bar{b}_{t,k}$ are deterministic, and that Lemma \ref{lem:AMP_oracle_data_diff_loo} continues to hold with formal changes (with all quantities $\hat{\#}$ replaced by $\bar{\#}$). Lemma \ref{lem:AMP_oracle_data_diff_Delta_b} requires a modification. Indeed, as
\begin{align*}
	\Delta\bar{b}_{t,k} = b_{t,k}-\bar{b}_{t,k}= \frac{1}{n}\sum_{\ell \in [n]}V_{k\ell}^2 \Big( \E[\mathsf{F}_{t,\ell}'(z_\ell^{(t)})|z^{(0)}] - \E [\mathsf{F}_{t,\ell}'(Z_\ell^{(t)})|z^{(0)}] \Big),
\end{align*}
we may now use the second estimate (\ref{ineq:AMP_sym_2}) in Theorem \ref{thm:AMP_sym} to obtain a deterministic estimate:
\begin{align*}
\pnorm{\Delta \bar{b}_t}{\infty} \leq  (c_1 K\Lambda \sigma_\ast^{[t+1],-1}\log n \cdot ( 1+ \pnorm{z^{(0)}}{\infty} ))^{c_0t^3} \cdot n^{-1/2}. 
\end{align*}
Lemmas \ref{lem:AMP_oracle_data_diff_deloc} and \ref{lem:AMP_oracle_data_diff_rem} also hold upon formal changes with the same estimates (which are easier to prove because $\Delta \bar{\mathsf{B}}_t$, $\Delta \bar{b}_{t,k}=\bar{b}_{t,k}-\bar{b}_{t,k}$ are deterministic). Specifically:
\begin{itemize}
	\item In the proof of the revised version of Lemma \ref{lem:AMP_oracle_data_diff_deloc}, we define the leave-one-out AMP iterate 
	\begin{align*}
	\bar{z}_{[-k]}^{(t+1)} \equiv  A_{[-k]}\mathsf{F}_t(\bar{z}_{[-k]}^{(t+1)})-\bar{b}_{t}\circ \mathsf{F}_{t-1}(\bar{z}^{(t-1)}_{[-k]})
	\end{align*}
	using the same $\bar{b}_{t}$. Indeed, the proof then proceeds exactly in the same way as that in Proposition \ref{prop:AMP_iterate_infty}.
	\item In the proof of the revised version of Lemma \ref{lem:AMP_oracle_data_diff_rem}, we no longer need to handle the terms $I_2$ and $I_{4,c}$.
\end{itemize} 
The estimate (\ref{ineq:AMP_oracles_diff}) follows from the above revised version of Lemmas \ref{lem:AMP_oracle_data_diff_Delta_b}-\ref{lem:AMP_oracle_data_diff_rem}.

\section{Proofs for Section \ref{section:regularized_ls}}\label{section:proof_ridge}

\subsection{Proof of Proposition \ref{prop:ridge_fpe_exist_uniq}}\label{subsection:proof_ridge_fpe_exist_uniq}

The following lemma is needed.

\begin{lemma}\label{lem:metriDpro}
	Let $\mathfrak{d}_{\mathbb{R}}:\R_{>0}\times \R_{>0}\to \R_{>0}$ defined via $\mathfrak{d}_{\mathbb{R}}(x,y)=\abs{x-y}^2/(xy)$. The following properties hold for $\mathfrak{d}_{\mathbb{R}}$:
	\begin{enumerate}
		\item  For any $x, y >0$, $\mathfrak{d}_{\mathbb{R}}( 1/ x, 1/y ) = \mathfrak{d}_{\mathbb{R}}(  x, y )$.
		\item For any $x, y,\beta >0$, $\mathfrak{d}_{\mathbb{R}}(x + \beta , y + \beta ) = (1 + \beta / x)^{-1}(1 + \beta / y)^{-1} \mathfrak{d}_{\mathbb{R}}(  x, y )$.
		\item For any $x, y,\alpha \in \mathbb{R}^m_{> 0}$, $\mathfrak{d}_{\mathbb{R}}\big(\sum_{k \in [m]}\alpha_kx_k,\sum_{k \in [m]}\alpha_ky_k\big) \leq \max_{k \in [m]}\mathfrak{d}_{\mathbb{R}}(x_k,y_k) $.
	\end{enumerate}
\end{lemma}
\begin{proof}
	(1) and (2) follow by direct calculations; we provide the proof of (3).
	Let us fix $\alpha, x, y \in \mathbb{R}^m_{> 0}$. First note that by AM-GM inequality, we have ${\big( x_{i} y_{i} x_{j} y_{j}\cdot\mathfrak{d}_{\mathbb{R}}(x_{i}, y_{i} )\mathfrak{d}_{\mathbb{R}}(x_{j}, y_{j} )  \big)^{1/2}}\leq {2^{-1}(  x_{i}y_{j}\mathfrak{d}_{\mathbb{R}}(x_{i}, y_{i} ) +  x_{j} y_{i}\mathfrak{d}_{\mathbb{R}}(x_{j}, y_{j} ) ) }$. This means 
	\begin{align*}
	&\biggabs{\sum_{i \in [m]}\alpha_i x_i  - \sum_{j \in [m]}\alpha_j y_j }^2 \leq \sum_{i, j \in [m]} \alpha_{i}\alpha_{j}  \cdot |x_{i} - y_{i} |\cdot|x_{j} - y_{j} |  \\
	&= \sum_{i, j \in [m]} \alpha_{i}\alpha_{j} \cdot \Big(x_{i} y_{i} x_{j} y_{j}\cdot\mathfrak{d}_{\mathbb{R}}(x_{i}, y_{i} )\mathfrak{d}_{\mathbb{R}}(x_{j}, y_{j} )  \Big)^{1/2}\\
	& \leq \frac{1}{2}\bigg(\sum_{i, j \in [m]} \alpha_{i}\alpha_{j}   x_{i} y_{j}\mathfrak{d}_{\mathbb{R}}(x_{i}, y_{i})+\sum_{i, j \in [m]} \alpha_{i}\alpha_{j}  x_{j} y_{i}\mathfrak{d}_{\mathbb{R}}(x_{j}, y_{j} )\bigg)\\
	&\leq \max_{k \in [m]} \mathfrak{d}_{\mathbb{R}}(x_{k}, y_{k} )\cdot  \bigg(\sum_{i \in [m]} \alpha_{i}x_i\bigg)\cdot\bigg(\sum_{j \in [m]}\alpha_j y_j\bigg). 
	\end{align*}
	The claim follows by the definition of $\mathfrak{d}_{\mathbb{R}}\big(\sum_{k \in [m]}\alpha_kx_k,\sum_{k \in [m]}\alpha_ky_k\big)$.
\end{proof}

\begin{proof}[Proof of Proposition \ref{prop:ridge_fpe_exist_uniq}: Existence and uniqueness of $b_\ast$] 
	We shall first define several notations used in the proof. Let $\mathfrak{u}\equiv \lambda^{-1}(1_m-b)$. Then viewed as a function of $\lambda>0$, $\mathfrak{u}(\lambda) \in \R^m$ satisfies the system of equations
	\begin{align}\label{eq:recQVE}
	\frac{1}{\mathfrak{u}_{k}(\lambda)} = \lambda + \sum_{\ell \in [n]} \frac{V_{k\ell}^2}{m} \Bigg( 1 + \sum_{k' \in [m]}\frac{V_{k'\ell}^2}{m}\cdot \mathfrak{u}_{k'}(\lambda) \Bigg)^{-1}, \quad k \in [m].
	\end{align}	
	The above system of equations can be rewritten in a vector form as a quadratic vector equation of $\mathfrak{u}(\lambda)$. The existence and uniqueness of the solution to this type of equation with $\lambda$ in the upper half complex plane has been discussed in \cite{alt2017local,ajanki2019quadratic}. Below we prove analogous results when $\lambda$ ranges over the half real line for the above system of equations.
	
	Let us fix some $\eta \in (0,1)$, and let $\mathbb{R}_\eta \equiv [\eta,1/\eta]$. Fix any pair $0<\underline{L}_\eta\leq \overline{L}_\eta<\infty$ such that $\underline{L}_\eta\leq {\eta}/\big(1 + m^{-1}\eta\max_{k \in [m]}\pnorm{V_{k\cdot}}{}^2 \big)$ and $\overline{L}_\eta\geq 1/\eta$.  Let the function space $\mathfrak{B}_\eta(\underline{L}_\eta,\overline{L}_\eta)$ be defined by
	\begin{align*}
	\mathfrak{B}_\eta(\underline{L}_\eta,\overline{L}_\eta) \equiv \Big\{\mathfrak{u}:\mathbb{R}_\eta\to  \mathbb{R}_{\geq 0}^m\, \hbox{ s.t. } \inf_{\lambda \in \mathbb{R}_\eta, k \in [m]} \mathfrak{u}_k(\lambda) \geq \underline{L}_\eta,\,\sup_{\lambda \in \mathbb{R}_\eta} \pnorm{ \mathfrak{u}(\lambda)}{\infty} \leq \overline{L}_\eta \Big\}.
	\end{align*}
	We will write $\mathfrak{B}_\eta\equiv\mathfrak{B}_\eta(\underline{L}_\eta,\overline{L}_\eta)$ below for notational simplicity.
	
	Recall $\mathfrak{d}_{\mathbb{R}}$ defined in Lemma \ref{lem:metriDpro}. Let $\mathfrak{d}_{\mathfrak{B}_\eta}:\mathfrak{B}_\eta\times \mathfrak{B}_\eta \to \R_{\geq 0}$ be a symmetric, non-negative map defined by 
	\begin{align*}
	\mathfrak{d}_{\mathfrak{B}_\eta}(\mathfrak{u}, \mathfrak{w}) \equiv \sup_{\lambda \in \mathbb{R}_\eta} \max_{k \in [m]} \mathfrak{d}_{\mathbb{R}}(\mathfrak{u}_k(\lambda), \mathfrak{w}_k(\lambda) ) \equiv \sup_{\lambda \in \mathbb{R}_\eta} \max_{k \in [m]} \frac{|\mathfrak{u}_k(\lambda)-\mathfrak{w}_k(\lambda) |^2}{ \mathfrak{u}_k(\lambda)\cdot\mathfrak{w}_k(\lambda) }.
	\end{align*}
	The idea of the above construction essentially dates back to \cite{earle1970fixed}.
	Furthermore, let the map $\Phi: \mathfrak{B}_\eta\to (\R_\eta \to \R_{\geq 0}^m)$ be defined via
	\begin{align*}
	\Phi(\mathfrak{u})(\lambda) \equiv \frac{1}{\lambda 1_m + \mathscr{V}_m(1_n+ \mathscr{V}_m^\top \mathfrak{u}(\lambda))^{-1}}, \quad \mathfrak{u} \in \mathfrak{B}_\eta, \; \lambda \in \mathbb{R}_\eta.
	\end{align*}
	Then the system of nonlinear equations (\ref{eq:recQVE}) takes the form $\mathfrak{u} = \Phi(\mathfrak{u})$. 
	
	\noindent (\textbf{Step 1}). We first verify that $\Phi(\mathfrak{B}_\eta)\subset \mathfrak{B}_\eta$, so is well-defined as a map from $\mathfrak{B}_\eta$ to itself.  Note that for any $k \in [m]$, $\lambda \in \mathbb{R}_\eta$ and $\mathfrak{u} \in \mathfrak{B}_\eta$, we have $
	\eta\leq \lambda + e_k^{\top} \mathscr{V}_m(1_n+ \mathscr{V}_m^\top \mathfrak{u}(\lambda))^{-1} \leq \lambda + e_k^\top\mathscr{V}_m  \leq \eta^{-1} + \pnorm{V_{k\cdot}}{}^2/m$, so
	\begin{align*}
	\frac{1}{\eta}\geq \Phi(\mathfrak{u})_k(\lambda) &= \frac{1}{\lambda + e_k^\top\mathscr{V}_m(1_n+ \mathscr{V}_m^\top \mathfrak{u}(\lambda))^{-1}} \geq \frac{\eta}{1 + m^{-1}\eta\max_{k \in [m]}\pnorm{V_{k\cdot}}{}^2 }.
	\end{align*}
	This verifies $\Phi(\mathfrak{B}_\eta)\subset \mathfrak{B}_\eta$.

	\noindent (\textbf{Step 2}).  Next we prove that for any $\mathfrak{u}, \mathfrak{w} \in \mathfrak{B}_\eta$, the map $\Phi : \mathfrak{B}_\eta \to \mathfrak{B}_\eta$ has the following contraction property with respect to $\mathfrak{d}_{\mathfrak{B}_\eta}$: for any $\mathfrak{u},\mathfrak{w} \in \mathfrak{B}_\eta$, 
	\begin{align}\label{eq:contractionPhi}
	\mathfrak{d}_{\mathfrak{B}_\eta}(\Phi(\mathfrak{u}), \Phi(\mathfrak{w})) \le \Big(1 + \frac{ 1/\overline{L}_\eta}{m^{-1}\max\limits_{k \in [m],\ell \in [n]}\pnorm{V_{k\cdot}}{}^2\vee \pnorm{V_{\cdot\ell }}{}^2 } \Big)^{-4} \cdot \mathfrak{d}_{\mathfrak{B}_\eta}(\mathfrak{u}, \mathfrak{w}). 
	\end{align}	
	To this end, for any $k \in [m], \ell \in [n], \lambda \in \R_\eta$ and $\mathfrak{u} \in \mathfrak{B}_\eta$, let
	\begin{align*}
	\mathsf{c}_{k;\mathfrak{u}(\lambda)} &\equiv \bigg(1+\frac{\lambda}{ e_k^{\top} \mathscr{V}_m(1_n+ \mathscr{V}_m^\top \mathfrak{u}(\lambda))^{-1}   }\bigg)^{-1}\leq \bigg(1+\frac{\eta}{ \max_{k \in [m]}\pnorm{V_{k\cdot}}{}^2/m }\bigg)^{-1},\\
	\mathsf{d}_{\ell;\mathfrak{u}(\lambda)} &\equiv \bigg(1+\frac{1}{e_\ell^\top \mathscr{V}_m^\top \mathfrak{u}(\lambda)}\bigg)^{-1}\leq \bigg(1+\frac{1/\overline{L}_\eta }{ \max_{\ell \in [n]}\pnorm{V_{\cdot\ell}}{}^2/m }\bigg)^{-1}.
	\end{align*}
	Then, for $\mathfrak{u}, \mathfrak{w} \in \mathfrak{B}_\eta$, by applying Lemma \ref{lem:metriDpro} (where (i)-(iii) correspond to (1)-(3) therein respectively), we have for any $k \in [m]$,
	\begin{align*}
	&\mathfrak{d}_{\mathbb{R}}(\Phi(\mathfrak{u})_k(\lambda), \Phi(\mathfrak{w})_k(\lambda) ) \\
	&\overset{(\mathrm{i})}{=} \mathfrak{d}_{\mathbb{R}}\Big(\lambda + e_k^{\top} \mathscr{V}_m(1_n+ \mathscr{V}_m^\top \mathfrak{u}(\lambda))^{-1}, \lambda + e_k^{\top} \mathscr{V}_m(1_n+ \mathscr{V}_m^\top \mathfrak{w}(\lambda))^{-1}\Big)\\
	&\overset{(\mathrm{ii})}{=}  \mathsf{c}_{k;\mathfrak{u}(\lambda)}  \mathsf{c}_{k;\mathfrak{w}(\lambda)}\cdot \mathfrak{d}_{\mathbb{R}}\Big( e_k^{\top} \mathscr{V}_m(1_n+ \mathscr{V}_m^\top \mathfrak{u}(\lambda))^{-1}, e_k^{\top} \mathscr{V}_m(1_n+ \mathscr{V}_m^\top \mathfrak{w}(\lambda))^{-1} \Big)\\
	&\overset{(\mathrm{iii})}{\le} \mathsf{c}_{k;\mathfrak{u}(\lambda)}  \mathsf{c}_{k;\mathfrak{w}(\lambda)}\cdot \max_{\ell \in [n]} \mathfrak{d}_{\mathbb{R}}\Big(e_\ell^\top(1_n+ \mathscr{V}_m^\top \mathfrak{u}(\lambda))^{-1}, e_\ell^\top(1_n+ \mathscr{V}_m^\top \mathfrak{w}(\lambda))^{-1} \Big)\\
	&\overset{(\mathrm{i})}{=}\mathsf{c}_{k;\mathfrak{u}(\lambda)}  \mathsf{c}_{k;\mathfrak{w}(\lambda)}\cdot \max_{\ell \in [n]} \mathfrak{d}_{\mathbb{R}}\Big(e_\ell^\top(1_n+ \mathscr{V}_m^\top \mathfrak{u}(\lambda)),  e_\ell^\top(1_n+ \mathscr{V}_m^\top \mathfrak{w}(\lambda)) \Big)\\
	&\overset{(\mathrm{ii})}{=}\mathsf{c}_{k;\mathfrak{u}(\lambda)}  \mathsf{c}_{k;\mathfrak{w}(\lambda)}\cdot \max_{\ell \in [n]}  \mathsf{d}_{\ell;\mathfrak{u}(\lambda)} \mathsf{d}_{\ell;\mathfrak{w}(\lambda)}\cdot \mathfrak{d}_{\mathbb{R}}\Big(e_\ell^\top \mathscr{V}_m^\top \mathfrak{u}(\lambda),  e_\ell^\top\mathscr{V}_m^\top \mathfrak{w}(\lambda) \Big)\\
	&\overset{(\mathrm{iii})}{\le} \mathsf{c}_{k;\mathfrak{u}(\lambda)}  \mathsf{c}_{k;\mathfrak{w}(\lambda)}\cdot \max_{\ell \in [n]}  \mathsf{d}_{\ell;\mathfrak{u}(\lambda)} \mathsf{d}_{\ell;\mathfrak{w}(\lambda)}\cdot \max_{k \in [m]} \mathfrak{d}_{\mathbb{R}}\big( \mathfrak{u}_k(\lambda),  \mathfrak{w}_k(\lambda) \big).
	\end{align*}
	Taking maxima over $k \in [m]$ and $\lambda \in \R_{\eta}$ completes the proof of (\ref{eq:contractionPhi}).

	\noindent (\textbf{Step 3}). We now wish to use the contraction property proved in Step 2 to conclude the existence and uniqueness.

	We start with the existence proof. For any continuous $\mathfrak{u}^{(0)} \in \mathfrak{B}_\eta$, we iteratively define $\mathfrak{u}^{(t+1)}\equiv \Phi(\mathfrak{u}^{(t)})$. Then the contraction property in (\ref{eq:contractionPhi}) yields that, for some constants $\epsilon_0\equiv \epsilon_0(V)\in (0,1)$ and $C_0=C_0(V,\eta)>1$, 
	\begin{align}\label{ineq:ridge_fpe_exist_uniq_b_1}
	\mathfrak{d}_{\mathfrak{B}_\eta}(\mathfrak{u}^{(t+1)}, \mathfrak{u}^{(t)})\leq (1-\epsilon_0)^t\cdot \mathfrak{d}_{\mathfrak{B}_\eta}(\mathfrak{u}^{(1)}, \mathfrak{u}^{(0)})\leq C_0(1-\epsilon_0)^t.
	\end{align}
	Note that $\mathfrak{d}_{\mathfrak{B}_\eta}$ does \emph{not} define a metric due to the failure of the triangle inequality, so the contraction property with respect to $\mathfrak{d}_{\mathfrak{B}_\eta}$ does not immediately lead to the existence of the solution. To this end, with the inverse hyperbolic function $\arcosh(x)=\log(x+\sqrt{x^2-1})$ defined for $x\geq 1$, let
	$\mathfrak{e}_{\mathfrak{B}_\eta}:\mathfrak{B}_\eta\times \mathfrak{B}_\eta \to \R_{\geq 0}$ be defined by 
	\begin{align*}
	\mathfrak{e}_{\mathfrak{B}_\eta}(\mathfrak{u}, \mathfrak{w}) &\equiv \sup_{\lambda \in \mathbb{R}_\eta} \max_{k \in [m]} \arcosh\Big(\frac{1}{2}\mathfrak{d}_{\mathbb{R}}(\mathfrak{u}_k(\lambda), \mathfrak{w}_k(\lambda) )+1 \Big)\equiv \sup_{\lambda \in \mathbb{R}_\eta}  \bigpnorm{\log \mathfrak{u}(\lambda)-\log \mathfrak{w}(\lambda)}{\infty}.
	\end{align*} 
	Then $\mathfrak{e}_{\mathfrak{B}_\eta}(\cdot,\cdot)$ is a metric on $\mathfrak{B}_\eta\times \mathfrak{B}_\eta$ that in particular satisfies the triangle inequality. Moreover, $(\mathfrak{B}_\eta,\mathfrak{e}_{\mathfrak{B}_\eta})$ is a complete metric space, as $\mathfrak{e}_{\mathfrak{B}_\eta}$ induces the same topology as that of $L_\infty$ on $\mathfrak{B}_\eta$. On the other hand, using the simple inequality that for any $z>0$, $\arcosh(1+z)=\log\big(1+z+\sqrt{z^2+2z}\big)\leq z+\sqrt{z^2+2z}\leq 2z+\sqrt{2z}$, we have for $\mathfrak{u},\mathfrak{w} \in \mathfrak{B}_\eta$, 
	\begin{align}\label{ineq:ridge_fpe_exist_uniq_b_2}
	\mathfrak{e}_{\mathfrak{B}_\eta}(\mathfrak{u}, \mathfrak{w})&\leq \mathfrak{d}_{\mathfrak{B}_\eta}(\mathfrak{u}, \mathfrak{w})+ \mathfrak{d}_{\mathfrak{B}_\eta}^{1/2}(\mathfrak{u}, \mathfrak{w}).
	\end{align}
	Consequently, for any $t< s$, using triangle inequality for $\mathfrak{e}_{\mathfrak{B}_\eta}(\cdot,\cdot)$ and (\ref{ineq:ridge_fpe_exist_uniq_b_1})-(\ref{ineq:ridge_fpe_exist_uniq_b_2}),
	\begin{align*}
	&\mathfrak{e}_{\mathfrak{B}_\eta}(\mathfrak{u}^{(s)}, \mathfrak{u}^{(t)})\leq \sum_{r \in [t:s)}  \mathfrak{e}_{\mathfrak{B}_\eta}(\mathfrak{u}^{(r+1)}, \mathfrak{u}^{(r)})\\
	&\leq \sum_{r \in [t:s)}  \mathfrak{d}_{\mathfrak{B}_\eta}(\mathfrak{u}^{(r+1)}, \mathfrak{u}^{(r)})+\sum_{r \in [t:s)}  \mathfrak{d}_{\mathfrak{B}_\eta}^{1/2}(\mathfrak{u}^{(r+1)}, \mathfrak{u}^{(r)})\leq C_1(1-\epsilon_0)^{t/2}.
	\end{align*}
	Here $C_1=C_1(V,\eta)>1$. This means $\{\mathfrak{u}^{(t)}\}$ is a Cauchy sequence in the complete metric space $(\mathfrak{B}_\eta,\mathfrak{e}_{\mathfrak{B}_\eta})$, and therefore the limit $\mathfrak{u}^{(\infty)}\equiv \lim_{t \to \infty} \mathfrak{u}^{(t)}$ is well defined in $(\mathfrak{B}_\eta,\mathfrak{e}_{\mathfrak{B}_\eta})$. By continuity of $\Phi$ with respect to $\mathfrak{e}_{\mathfrak{B}_\eta}$ on $\mathfrak{B}_\eta$, we may take limit $t \to \infty$ in the equation $\mathfrak{u}^{(t+1)}=\Phi(\mathfrak{u}^{(t)})$ to obtain $\mathfrak{u}^{(\infty)}=\Phi(\mathfrak{u}^{(\infty)})$. As $\mathfrak{u}^{(0)}$ is continuous, and $\Phi$ preserves continuity, the limit $\mathfrak{u}^{(\infty)}$ induced by $L_\infty$-topology must also be continuous. From (\ref{eq:recQVE}), it is easy to see that $\max_{k \in [m]}\abs{\mathfrak{u}_k^{(\infty)}(\lambda)}<1/\lambda$, and therefore $b^{(\infty)}\equiv 1_m-\lambda \mathfrak{u}^{(\infty)}(\lambda) \in [0,1)^m$ is well-defined. This completes the existence proof.
	
	Next we provide the uniqueness proof. The uniqueness proof does not require the triangle inequality property for $\mathfrak{d}_{\mathfrak{B}_\eta}(\cdot,\cdot)$, so is much easier. Let us take $\mathfrak{u}, \mathfrak{u}' \in \mathfrak{B}_\eta$ with $\mathfrak{u}=\Phi(\mathfrak{u})$ and $\mathfrak{u}'=\Phi(\mathfrak{u}')$. Then the contraction property (\ref{eq:contractionPhi}) shows that there exists some $\epsilon_0=\epsilon_0(V) \in (0,1)$ such that
	\begin{align*}
	0\leq \mathfrak{d}_{\mathfrak{B}_\eta}(\mathfrak{u}, \mathfrak{u}')=\mathfrak{d}_{\mathfrak{B}_\eta}(\Phi(\mathfrak{u}), \Phi(\mathfrak{u}'))\leq (1-\epsilon_0)\cdot \mathfrak{d}_{\mathfrak{B}_\eta}(\mathfrak{u}, \mathfrak{u}'). 
	\end{align*}
	This means $\mathfrak{d}_{\mathfrak{B}_\eta}(\mathfrak{u}, \mathfrak{u}')=0$. Using the definition of $\mathfrak{d}_{\mathfrak{B}_\eta}$ and the fact that $\mathfrak{u}, \mathfrak{u}' \in \mathfrak{B}_\eta$, 
	\begin{align*}
	\underline{L}_\eta^{-2} \sup_{\lambda \in \R_\eta} \pnorm{ \mathfrak{u}(\lambda)-\mathfrak{u}'(\lambda) }{\infty}\leq \mathfrak{d}_{\mathfrak{B}_\eta}(\mathfrak{u}, \mathfrak{u}')=0.
	\end{align*}
	This implies the uniqueness claim $\mathfrak{u}\equiv \mathfrak{u}'$ as elements in $\mathfrak{B}_\eta$. 
	
	Combining both claims, we have proven that for any $\eta \in (0,1)$, there exists a unique solution $\mathfrak{u}^\ast\in \mathfrak{B}_\eta\cap C(\R_\eta\to \R_{\geq 0}^m)$. 
	The claim now follows as $\eta \in (0,1)$ is arbitrary. 
\end{proof}

\begin{proof}[Proof of Proposition \ref{prop:ridge_fpe_exist_uniq}: Existence and uniqueness of $\gamma_\ast$]
	Let $b^\ast=b^\ast(\lambda)$ be the unique solution of the second equation of (\ref{eqn:fpe_ridge}) as proved above. It is easy to obtain the estimate 
	\begin{align*}
	\max_{k \in [m]} b^\ast_k(\lambda)&\leq \Big(1+m\lambda/\max_{k \in [m]} \pnorm{V_{k\cdot}}{}^2\Big)^{-1},\nonumber\\
	\max_{\ell \in [n]} \tau_{b^\ast(\lambda),\ell}&\leq \Big(1-\max_{k \in [m]} b^\ast_k(\lambda)\Big)^{-1}\cdot \Big(\min_{\ell \in [n]}\pnorm{V_{\cdot \ell
	}}{}^2/m\Big)^{-1},\nonumber\\
	\min_{\ell \in [n]} \tau_{b^\ast(\lambda),\ell}&\geq \Big(\max_{k \in [m]}\pnorm{V_{k\cdot}}{}^2/m\Big)^{-1}.
	\end{align*}	
	In particular, for any $\eta \in (0,1)$, with $\R_\eta=[\eta,1/\eta]$ as before, we may find some $\epsilon_0=\epsilon_0(V,\eta) \in (0,1)$ such that
	\begin{align}\label{ineq:ridge_fpe_exist_uniq_gamma_estimates}
	\sup_{\lambda \in \R_\eta}\max_{k \in [m]} b^\ast_k(\lambda)\leq 1-\epsilon_0,\quad \epsilon_0\leq \inf_{\lambda \in \R_\eta }\min_{\ell \in [n]} \tau_{b^\ast(\lambda),\ell}\leq \sup_{\lambda \in \R_\eta}\max_{\ell \in [n]} \tau_{b^\ast(\lambda),\ell}\leq 1/\epsilon_0. 
	\end{align}	
	The first equation of (\ref{eqn:fpe_ridge}) of interest can be written as
	\begin{align}\label{ineq:ridge_fpe_exist_uniq_gamma_eqn}
	\gamma^2(\lambda) = q_{b^\ast(\lambda)}+ \mathfrak{D}_{\tau_{b^\ast}(\lambda)}^2\mathscr{V}_m^\top \mathfrak{D}_{1_m-b^\ast(\lambda)}^2 \mathscr{V}_m \mathfrak{D}_{1_n+\lambda \tau_{b^\ast(\lambda)}}^{-2} \gamma^2(\lambda),
	\end{align}
	where $q_{b^\ast}: \R_{>0}\to \R^n_{\geq 0}$ is locally bounded. In particular, for any $\eta \in (0,1)$, there exists $\epsilon_1 = \epsilon_1(V,\eta) \in (0,1)$,
	\begin{align}\label{ineq:ridge_fpe_exist_uniq_gamma_q_bound}
	\sup_{\lambda \in \R_\eta} \pnorm{q_{b^\ast}(\lambda)}{\infty}\leq 1/\epsilon_1.
	\end{align}
	We remark here that simply inverting (\ref{ineq:ridge_fpe_exist_uniq_gamma_eqn}) to solve $\gamma^2(\lambda)$ will not guarantee that the solution is in $\R_{\geq 0}^n$. 
	Equivalently, we may write (\ref{ineq:ridge_fpe_exist_uniq_gamma_eqn}) as
	\begin{align}\label{ineq:ridge_fpe_exist_uniq_gamma_zeta_eqn}
	\zeta(\lambda) &= \mathfrak{D}_{\tau_{b^\ast(\lambda)}}^{-1}q_{b^\ast(\lambda)}+ \mathfrak{D}_{\tau_{b^\ast(\lambda)} }\mathscr{V}_m^\top \mathfrak{D}_{1_m-b^\ast(\lambda)}^2 \mathscr{V}_m \mathfrak{D}_{\tau_{b^\ast(\lambda)}\circ (1_n+\lambda \tau_{b^\ast(\lambda)})^{-2}} \zeta(\lambda)\nonumber\\
	&\equiv \Psi(\zeta)(\lambda),
	\end{align}
	where $\zeta \equiv  \mathfrak{D}_{\tau_{b^\ast}}^{-1}\gamma^2 : \R_{>0}\to \R_{\geq 0}^n$. The essential merit of working with $\zeta$ instead of $\gamma$ is that $\Psi$ will be a strict contraction with respect to a natural $L_\infty$ metric. In particular, consider the function class
	\begin{align*}
	\mathfrak{Z}_\eta\equiv \Big\{\zeta: \R_\eta \to \R_{\geq 0}^n\,\hbox{ s.t. }\sup_{\lambda \in \R_\eta} \pnorm{\zeta(\lambda)}{\infty}\leq \frac{1}{\epsilon_0^2\epsilon_1} \Big\}, 
	\end{align*}
	and the metric (with slight abuse of notation) $\mathfrak{d}_{\mathfrak{Z}_\eta}: \mathfrak{Z}_\eta\times \mathfrak{Z}_\eta\to \R_{\geq 0}$ defined via
	\begin{align*}
	\mathfrak{d}_{\mathfrak{Z}_\eta}(\zeta_1,\zeta_2)\equiv \sup_{\lambda \in \R_\eta} \pnorm{\zeta_{1}(\lambda)-\zeta_{2}(\lambda)}{\infty}.
	\end{align*}
	In the sequel, we shall adopt a similar 3-step strategy as in the proof of existence and uniqueness for $b^\ast$. 
	
	\noindent (\textbf{Step 1}). First, we shall verify that $\Psi(\mathfrak{Z}_\eta)\subset \mathfrak{Z}_\eta$. Using the definition of $\Psi$ in (\ref{ineq:ridge_fpe_exist_uniq_gamma_zeta_eqn}), and the estimates in (\ref{ineq:ridge_fpe_exist_uniq_gamma_estimates}) and (\ref{ineq:ridge_fpe_exist_uniq_gamma_q_bound}), for any $\zeta \in \mathfrak{Z}_\eta$, 
	\begin{align*}
	&\sup_{\lambda \in \R_\eta} \pnorm{\Psi(\zeta)(\lambda)}{\infty}\\
	&\leq \frac{1}{\epsilon_0\epsilon_1}+ \sup_{\lambda \in \R_\eta}\max_{\ell \in [n]}\biggabs{\tau_{b^\ast,\ell} \sum_{k \in [m]} \frac{V_{k\ell}^2}{m}(1-b_k^\ast)^2 \sum_{\ell' \in [n]} \frac{V_{k\ell'}^2}{m}\cdot \frac{\tau_{b^\ast,\ell'}}{(1+\lambda \tau_{b^\ast,\ell'})^2} \zeta_{\ell'}(\lambda)}\\
	&\stackrel{(\ast)}{\leq} \frac{1}{\epsilon_0\epsilon_1}+\sup_{\lambda \in \R_\eta} \max_{\ell \in [n]}\tau_{b^\ast,\ell} \sum_{k \in [m]} \frac{V_{k\ell}^2}{m}(1-b_k^\ast)^2 \cdot \frac{b_k^\ast}{1-b_k^\ast}\cdot \sup_{\lambda \in \R_\eta}\pnorm{\zeta(\lambda)}{\infty}\\  
	&\leq \frac{1}{\epsilon_0\epsilon_1}+ \sup_{\lambda \in \R_\eta}\max_{k \in [m]}b^\ast_k(\lambda)\cdot\max_{\ell \in [n]} \bigg(\tau_{b^\ast,\ell} \sum_{k \in [m]} \frac{V_{k\ell}^2}{m}(1-b_k^\ast)\bigg)\cdot  \sup_{\lambda \in \R_\eta}\pnorm{\zeta(\lambda)}{\infty}\\
	& \stackrel{(\ast\ast)}{=} \frac{1}{\epsilon_0\epsilon_1}+\sup_{\lambda \in \R_\eta}\max_{k \in [m]} b^\ast_k(\lambda)\cdot  \sup_{\lambda \in \R_\eta}\pnorm{\zeta(\lambda)}{\infty}\\
	&\leq \frac{1}{\epsilon_0\epsilon_1}+(1-\epsilon_0) \cdot \sup_{\lambda \in \R_\eta}\pnorm{\zeta(\lambda)}{\infty}\leq \frac{1}{\epsilon_0^2\epsilon_1}.
	\end{align*}
	Here in $(\ast)$ we used the characterizing equation for $b^\ast$, and in $(\ast\ast)$ we used the definition of $\tau_{b^\ast}$. This verifies $\Psi(\mathfrak{Z}_\eta)\subset \mathfrak{Z}_\eta$.

	\noindent (\textbf{Step 2}). Next, we shall verify the following contraction property for $\Psi:\mathfrak{Z}_\eta\to \mathfrak{Z}_\eta$ with respect to $\mathfrak{d}_{\mathfrak{Z}_\eta}$: for any $\zeta_1,\zeta_2 \in \mathfrak{Z}_\eta$,
	\begin{align*}
	\mathfrak{d}_{\mathfrak{Z}_\eta}\big(\Psi(\zeta_1),\Psi(\zeta_2)\big)\leq (1-\epsilon_0)\cdot \mathfrak{d}_{\mathfrak{Z}_\eta}\big(\zeta_1,\zeta_2\big).
	\end{align*}
	The above display follows by essentially the same calculations as in the Step 1, but without the additive term $1/(\epsilon_0\epsilon_1)$ and replacing $\zeta$ therein with $\zeta_1-\zeta_2$:
	\begin{align*}
	&\sup_{\lambda \in \R_\eta}\bigpnorm{\Psi(\zeta_1)(\lambda)-\Psi(\zeta_2)(\lambda)}{\infty}\\
	&= \sup_{\lambda \in \R_\eta}\max_{\ell \in [n]}\biggabs{\tau_{b^\ast,\ell} \sum_{k \in [m]} \frac{V_{k\ell}^2}{m}(1-b_k^\ast)^2 \sum_{\ell' \in [n]} \frac{V_{k\ell'}^2}{m}\cdot \frac{\tau_{b^\ast,\ell'}}{(1+\lambda \tau_{b^\ast,\ell'})^2} (\zeta_{1,\ell'}(\lambda)-\zeta_{2,\ell'}(\lambda))}\\
	& \leq (1-\epsilon_0)\cdot \sup_{\lambda \in \R_\eta} \pnorm{\zeta_1(\lambda)-\zeta_2(\lambda)}{\infty}.
	\end{align*}

	\noindent (\textbf{Step 3}). As $(\mathfrak{Z}_\eta,\mathfrak{d}_{\mathfrak{Z}_\eta})$ is a complete metric space,  $\Psi(\zeta)=\zeta$ admits a unique solution $\zeta^\ast \in \mathfrak{Z}_\eta$ by the contraction mapping theorem. Clearly $\zeta^\ast$ can be taken as continuous. The claim follows as $\eta \in (0,1)$ is arbitrary. 	
\end{proof}

\subsection{Proof of Theorem \ref{thm:ridge_dist}}\label{subsection:proof_ridge_dist}

We shall first prove Proposition \ref{prop:ridge_theta_error}. To this end, we need two lemmas. 

\begin{lemma}\label{lem:ridge_AMP_diff_l2}
	Suppose for some $K\geq 2$, $m^{-1/2}\big(\pnorm{V_{k\cdot}}{}\wedge  \pnorm{V_{\cdot\ell}}{}\big)\geq 1/K$, $V_{k\ell}\leq K$ and hold for all $k \in [m],\ell \in [n]$, and $1/K\leq \lambda \leq K$ holds. Then there exists some universal constant $c_0>1$ such that 
	\begin{align*}
	&\pnorm{\theta^{(t)}-\hat{\theta}}{}\leq K^{c_0}\cdot(1\vee \pnorm{A}{\op})\cdot \big(\pnorm{u^{(t+1)}-u^{(t)}}{}+ \pnorm{v^{(t+1)}-v^{(t)} }{}\big),
	\end{align*}
	and
	\begin{align*}
	&\pnorm{r^{(t)}-\hat{r}}{}\leq K^{c_0}  (1\vee \pnorm{A}{\op})^2 \sum_{s \in [t-1]}(1-K^{-c_0})^{t-1-s} \big(\pnorm{u^{(s+1)}-u^{(s)}}{}+ \pnorm{v^{(s+1)}-v^{(s)} }{}\big).
	\end{align*}
\end{lemma}
\begin{proof} 
	By the first equation of (\ref{def:AMP_ls}), with $\Xi_{r}^{(t)}\equiv \mathfrak{D}_{b^\ast\circ (1_m-b^\ast)^{-1}} (r^{(t)}-r^{(t+1)})$, 
	\begin{align*}
	r^{(t+1)}& = \mathfrak{D}_{1_m-b^\ast}^{-1} \big(A_{b^\ast}(\theta_{0;b^\ast}- \theta^{(t)})+\xi_{b^\ast}\big)+\Xi_{r}^{(t)}.
	\end{align*}
	Plugging into the second equation of (\ref{def:AMP_ls}), we have 
	\begin{align*}
	\theta^{(t+1)}= \prox_{\mathsf{f}_{\tau_{b^\ast}}} \Big[ \theta^{(t)}+A_{b^\ast}^\top \mathfrak{D}_{1_m-b^\ast}^{-1} \Big(A_{b^\ast}\big( \theta_{0;b^\ast}- \theta^{(t)}\big)+\xi_{b^\ast}\Big)+ A_{b^\ast}^\top \Xi_{r}^{(t)}\Big].
	\end{align*}
	Consequently, with
	\begin{align*}
	\delta^{(t)}&\equiv (1_n+\tau_{b^\ast}\lambda)^{-1}A_{b^\ast}^\top \mathfrak{D}_{b^\ast\circ (1_m-b^\ast)^{-1}} (r^{(t)}-r^{(t+1)})+(\theta^{(t)}-\theta^{(t+1)}),
	\end{align*}
	we have
	\begin{align*}
	&\theta^{(t)}= \prox_{\mathsf{f}_{\tau_{b^\ast}}} \Big[ \theta^{(t)}+A_{b^\ast}^\top \mathfrak{D}_{1_m-b^\ast}^{-1} \Big(A_{b^\ast}\big( \theta_{0;b^\ast}- \theta^{(t)}\big)+\xi_{b^\ast}\Big)\Big]+ \delta^{(t)},\\
	\Leftrightarrow\quad & \mathfrak{D}_{\tau_{b^\ast}}^{1/2} \theta^{(t)} = \prox_{\mathsf{f}_{\tau_{b^\ast} } }  \Big[ \mathfrak{D}_{\tau_{b^\ast}}^{1/2} \theta^{(t)}+\mathfrak{D}_{\tau_{b^\ast}} A^\top \Big(A\big( \mu_{0}-\mathfrak{D}_{\tau_{b^\ast}}^{1/2} \theta^{(t)}\big)+\xi\Big)\Big]+ \mathfrak{D}_{\tau_{b^\ast}}^{1/2} \delta^{(t)},\\
	\Leftrightarrow\quad & \lambda \big(\mathfrak{D}_{\tau_{b^\ast}}^{1/2} \theta^{(t)} \big) + A^\top \big(A (\mathfrak{D}_{\tau_{b^\ast}}^{1/2} \theta^{(t)} )-Y\big) = \mathfrak{D}_{1_n+\lambda\tau_{b^\ast}}\mathfrak{D}_{\tau_{b^\ast}}^{-1/2} \delta^{(t)}.
	\end{align*}
	On the other hand, with $\mathsf{L}(\mu)\equiv 2^{-1}\big(\pnorm{Y-A\mu}{}^2+\lambda \pnorm{\mu}{}^2\big)$, $\nabla \mathsf{L}(\mu)= A^\top (A\mu-Y)+\lambda \mu$, and $\nabla^2 \mathsf{L}\geq \lambda I_n$. Moreover, the above display is equivalent to $\nabla \mathsf{L}(\mathfrak{D}_{\tau_{b^\ast}}^{1/2} \theta^{(t)}) = \mathfrak{D}_{1_n+\lambda\tau_{b^\ast}}\mathfrak{D}_{\tau_{b^\ast}}^{-1/2} \delta^{(t)}$. So by the cost optimality,
	\begin{align*}
	\mathsf{L}\big(\mathfrak{D}_{\tau_{b^\ast}}^{1/2} \theta^{(t)}\big)\geq \mathsf{L}(\hat{\mu})\geq \mathsf{L}\big(\mathfrak{D}_{\tau_{b^\ast}}^{1/2} \theta^{(t)}\big)+\nabla \mathsf{L}\big(\mathfrak{D}_{\tau_{b^\ast}}^{1/2} \theta^{(t)}\big)\cdot \big(\hat{\mu}-\mathfrak{D}_{\tau_{b^\ast}}^{1/2} \theta^{(t)}\big)+\frac{\lambda}{2}\bigpnorm{\hat{\mu}-\mathfrak{D}_{\tau_{b^\ast}}^{1/2} \theta^{(t)}}{}^2.
	\end{align*}
	Under the assumed conditions on $V$, $\pnorm{\tau_{b^\ast}^{-1}}{\infty}\vee \pnorm{\tau_{b^\ast}}{\infty}\vee \pnorm{(1_m-b^\ast)^{-1}}{\infty}\leq K^c $, so  
	\begin{align*}
	\bigpnorm{\hat{\mu}-\mathfrak{D}_{\tau_{b^\ast}}^{1/2} \theta^{(t)}}{}&\leq 2\lambda^{-1} \bigpnorm{\nabla \mathsf{L}\big(\mathfrak{D}_{\tau_{b^\ast}}^{1/2} \theta^{(t)}\big) }{}\leq K^c\cdot  \pnorm{\delta^{(t)}}{}\\
	&\leq K^c \cdot \big(\pnorm{A}{\op}\cdot \pnorm{r^{(t+1)}-r^{(t)}}{}+ \pnorm{\theta^{(t+1)}-\theta^{(t)} }{}\big).
	\end{align*}
	Now using that $\pnorm{r^{(t+1)}-r^{(t)}}{}=\pnorm{u^{(t+1)}-u^{(t)}}{}$ and that $\pnorm{\theta^{(t+1)}-\theta^{(t)} }{}\leq \pnorm{v^{(t+1)}-v^{(t)} }{}+\pnorm{A_{b^\ast}}{\op} \pnorm{r^{(t+1)}-r^{(t)} }{}$, we have
	\begin{align*}
	\pnorm{\theta^{(t)}-\hat{\theta}}{}\leq K^c \cdot(1\vee \pnorm{A}{\op})\cdot \big(\pnorm{u^{(t+1)}-u^{(t)}}{}+ \pnorm{v^{(t+1)}-v^{(t)} }{}\big).
	\end{align*}
	For the bound for $\pnorm{r^{(t)}-\hat{r}}{}$, using the first equation of (\ref{def:AMP_ls}),
	\begin{align*}
	\pnorm{r^{(t)}-\hat{r}}{}&\leq \pnorm{A_{b^\ast}}{\op} \pnorm{\theta^{(t-1)}-\hat{\theta}}{}+\pnorm{b^\ast}{\infty}\pnorm{r^{(t-1)}-\hat{r}}{}.
	\end{align*}
	Iterating the bound, we arrive at
	\begin{align*}
	&\pnorm{r^{(t)}-\hat{r}}{}\leq \pnorm{A_{b^\ast}}{\op} \sum_{s \in [t-1]} \pnorm{b^\ast}{\infty}^{t-1-s}\pnorm{\theta^{(s)}-\hat{\theta}}{}\\
	&\leq  K^c  (1\vee \pnorm{A}{\op})^2 \sum_{s \in [t-1]}(1-K^{-c})^{t-1-s} \big(\pnorm{u^{(s+1)}-u^{(s)}}{}+ \pnorm{v^{(s+1)}-v^{(s)} }{}\big),\nonumber
	\end{align*}
	as desired.
\end{proof}

\begin{lemma}\label{lem:ridge_u_v_error}
	Suppose $1/K\leq m/n,\lambda\leq K$ for some $K\geq 2$. Further suppose $m^{-1/2}\big(\pnorm{V_{k\cdot}}{}\wedge  \pnorm{V_{\cdot\ell}}{}\big)\geq 1/K$, and  $V_{k\ell}\leq K$ holds for all $k \in [m],\ell \in [n]$. Then for any $D>1$, there exist some universal constant $c_0>0$ and another constant $c_1=c_1(D)>0$, such that for all $1\leq t\leq \log n/c_0$, with $\Prob^\xi(\cdot|\theta^{(0)})$-probability at least $1-c_1 n^{-D}$,
	\begin{align*}
	&n^{-1} \max\big\{\pnorm{u^{(t+1)}-u^{(t)}}{}^2,\pnorm{v^{(t+1)}-v^{(t)}}{}^2\big\}\\
	&\leq K^{c_0} \cdot n^{-1}\big(\pnorm{\mu_0}{}^2\vee \pnorm{\xi}{}^2\big)\cdot (1-K^{-c_0})^t+\big(K \log n\cdot (1+\pnorm{z^{(0)}}{\infty})\big)^{c_0 t^3}\cdot n^{-1/c_0^t}.
	\end{align*}
\end{lemma}
\begin{proof}
	In this proof the expectation is all taken conditional on $\xi$. First, with $(\sigma_u^{(s)})^2 \equiv \big(\var(U^{(s)}_k)\big)_{k \in [m]} \in \R^m$ and $(\sigma_v^{(s)})^2 \equiv \big(\var(V^{(s)}_\ell)\big)_{\ell \in [n]} \in \R^n$, using the state evolution in Definition \ref{def:AMP_asym_se}, the first equation therein reduces to 
	\begin{align*}
	(\sigma_u^{(t+1)})^2&=\big(\mathfrak{D}_{1_m-b^\ast}\mathscr{V}_m\mathfrak{D}_{\tau_{b^\ast}}\big) \E\mathsf{F}_t^2(\sigma_v^{(t)} Z_n)\\
	&=\big(\mathfrak{D}_{1_m-b^\ast}\mathscr{V}_m\mathfrak{D}_{\tau_{b^\ast}}\big) \mathfrak{D}_{1+\lambda \tau_{b^\ast}}^{-2}\big(\mathfrak{D}_{\lambda \tau_{b^\ast}}^2 \theta_{0;b^\ast}^2+(\sigma_v^{(t)})^2\big),
	\end{align*}
	and the second equation reduces to
	\begin{align*}
	(\sigma_v^{(t+1)})^2&= \big(\mathfrak{D}_{1_m-b^\ast}\mathscr{V}_m\mathfrak{D}_{\tau_{b^\ast}}\big)^\top \E \mathsf{G}_{t+1}^2(\sigma_u^{(t+1)}Z_m) \\
	& = \mathfrak{D}_{\tau_{b^\ast}}\mathscr{V}_m^\top \mathfrak{D}_{1_m-b^\ast} \big[ (\sigma_u^{(t+1)})^2+ \xi_{b^\ast}^2 \big]\\
	& = \mathfrak{D}_{\tau_{b^\ast}}\mathscr{V}_m^\top \mathfrak{D}_{1_m-b^\ast}^2 \big( \mathscr{V}_m \mathfrak{D}_{\lambda \tau_{b^\ast}\circ (1+\lambda \tau_{b^\ast})^{-1}}^2 \mu_{0}^2+\xi^2\big)\\
	&\qquad + \mathfrak{D}_{\tau_{b^\ast}}\mathscr{V}_m^\top \mathfrak{D}_{1_m-b^\ast}^2 \mathscr{V}_m \mathfrak{D}_{1+\lambda \tau_{b^\ast}}^{-2} \mathfrak{D}_{\tau_{b^\ast}}(\sigma_v^{(t)})^2.
	\end{align*}
	Comparing the above display to (\ref{ineq:ridge_fpe_exist_uniq_gamma_zeta_eqn}) and the proofs therein, 
	\begin{align*}
	\pnorm{(\sigma_v^{(t+1)})^2-\zeta_v^\ast}{\infty}\leq  K^c (1-K^{-c})^t \cdot \pnorm{(\sigma_v^{(1)})^2-(\sigma_v^{(0)})^2}{\infty}.
	\end{align*}
	Here (with slight abuse of notation) $\zeta_v^\ast=\zeta_v^\ast(\lambda) \in \R_{\geq 0}^n$ is the unique fixed point for the equation (\ref{ineq:ridge_fpe_exist_uniq_gamma_zeta_eqn}). Under the assumed conditions on $V$, 
	\begin{align}\label{ineq:ridge_u_v_error_1}
	\pnorm{(\sigma_v^{(t+1)})^2-\zeta_v^\ast}{\infty}&\leq K^c \cdot m^{-1}(\pnorm{\mu_0}{}^2\vee \pnorm{\xi}{}^2)\cdot (1-K^{-c})^t.
	\end{align}
	Next, let us define the correlation vectors $\mathfrak{r}_u^{(s+1)} \equiv  \big(\cov(U^{(s+1)}_k,U^{(s)}_k)\big)_{k \in [m]} \in \R^m$ and
	$\mathfrak{r}_v^{(s+1)} = \big(\cov(V^{(s+1)}_\ell,V^{(s)}_\ell)\big)_{\ell \in [n]} \in \R^n$. Using the state evolution in Definition \ref{def:AMP_asym_se} again, with some calculations along the above lines we obtain the exact same recursion:
	\begin{align*}
	\mathfrak{r}_v^{(t+1)}
	& = \mathfrak{D}_{\tau_{b^\ast}}\mathscr{V}_m^\top \mathfrak{D}_{1_m-b^\ast}^2 \big( \mathscr{V}_m \mathfrak{D}_{\lambda \tau_{b^\ast}\circ (1+\lambda \tau_{b^\ast})^{-1}}^2 \mu_{0}^2+\xi^2\big)\\
	&\qquad + \mathfrak{D}_{\tau_{b^\ast}}\mathscr{V}_m^\top \mathfrak{D}_{1_m-b^\ast}^2 \mathscr{V}_m \mathfrak{D}_{1+\lambda \tau_{b^\ast}}^{-2} \mathfrak{D}_{\tau_{b^\ast}}\mathfrak{r}_v^{(t)}.
	\end{align*}
	While the difference lies in the initialization, the following bound remains valid:
	\begin{align}\label{ineq:ridge_u_v_error_2}
	\pnorm{\mathfrak{r}_v^{(t+1)}-\zeta_v^\ast}{\infty}&\leq K^c \cdot m^{-1}(\pnorm{\mu_0}{}^2\vee \pnorm{\xi}{}^2)\cdot (1-K^{-c})^t.
	\end{align} 
	Similarly, with $
	\zeta_u^\ast\equiv \big(\mathfrak{D}_{1_m-b^\ast}\mathscr{V}_m\mathfrak{D}_{\tau_{b^\ast}}\big) \mathfrak{D}_{1+\lambda \tau_{b^\ast}}^{-2}\big(\mathfrak{D}_{\lambda \tau_{b^\ast}}^2 \theta_{0;b^\ast}^2+\zeta_v^\ast\big)$, we have the estimate
	\begin{align}\label{ineq:ridge_u_v_error_3}
	\pnorm{(\sigma_u^{(t+1)})^2-\zeta_u^\ast}{\infty}\vee \pnorm{\mathfrak{r}_u^{(t+1)}-\zeta_u^\ast}{\infty}\leq K^c \cdot m^{-1}(\pnorm{\mu_0}{}^2\vee \pnorm{\xi}{}^2)\cdot (1-K^{-c})^t.
	\end{align}
	Now using the above display (\ref{ineq:ridge_u_v_error_3}) and Theorem \ref{thm:AMP_asym_avg}, with $c_0>0$ a proper universal constant, $c_1=c_1(D)>0$ depending on $D>1$ only, and  $\err_n^{(t)}\equiv \big(c_1 K \log n_\phi\cdot (1+\pnorm{z^{(0)}}{\infty})\big)^{c_0 t^3}\cdot n^{-1/c_0^t}$, we have with probability $1-c_1 n^{-D}$, 
	\begin{align*}
	n_\phi^{-1}\pnorm{u^{(t+1)}-u^{(t)}}{}^2&\leq n_\phi^{-1}\E\pnorm{U^{(t+1)}-U^{(t)}}{}^2+ \err_n^{(t)}\\
	&\leq K^c \cdot \max_{k \in [m]} \bigabs{(\sigma_{u,k}^{(t+1)})^2+(\sigma_{u,k}^{(t)})^2-2 \mathfrak{r}_{u,k}^{(t+1)}}+ \err_n^{(t)}\\
	&\leq K^c \cdot m^{-1}(\pnorm{\mu_0}{}^2\vee \pnorm{\xi}{}^2)\cdot (1-K^{-c})^t+\err_n^{(t)}.
	\end{align*}
	A completely similar argument via the help of (\ref{ineq:ridge_u_v_error_1})-(\ref{ineq:ridge_u_v_error_2}) establishes the same estimate for $n_\phi^{-1}\pnorm{v^{(t+1)}-v^{(t)}}{}^2$.
\end{proof}

\begin{proof}[Proof of Proposition \ref{prop:ridge_theta_error}]
	The claim follows from Lemmas \ref{lem:ridge_AMP_diff_l2} and \ref{lem:ridge_u_v_error} and the standard subgaussian tail estimate for $\pnorm{A}{\op}$, cf. Lemma \ref{lem:A_op_subgaussian_tail}.
\end{proof}

Now we may prove Theorem \ref{thm:ridge_dist}.

\begin{proof}[Proof of Theorem \ref{thm:ridge_dist} for $\hat{\mu}$]
	In this proof expectation is all taken conditional on $\xi$. Constants unspecified in the proof may depend on $K,\Lambda$. Recall $\hat{\mu}=(A^\top A+\lambda I)^{-1}A^\top Y$, so $\pnorm{\hat{\mu}}{}\leq \lambda^{-1} (1+\pnorm{A}{\op})^2(\pnorm{\mu_0}{}\vee \pnorm{\xi}{})$. By subgaussian tail of $\pnorm{A}{\op}$ (cf. Lemma \ref{lem:A_op_subgaussian_tail}), with probability at least $1-e^{-n/c}$, 
	\begin{align*}
	n^{-1/2}\pnorm{\hat{\mu}}{} \leq c\cdot n^{-1/2} (\pnorm{\mu_0}{}\vee \pnorm{\xi}{}).
	\end{align*}
	So combined with Proposition \ref{prop:ridge_theta_error}, there exists some $\epsilon_0=\epsilon_0(K) \in (0,1)$ such that with probability $1-e^{-n/c}$,
	\begin{align}\label{ineq:ridge_dist_1}
	&\biggabs{\frac{1}{n}\sum_{j \in [n]} \psi\big(\hat{\mu}_j, \mu_{0,j}\big)-\frac{1}{n}\sum_{j \in [n]} \psi\big(\mu_j^{(t)},\mu_{0,j}\big)}\nonumber\\
	&\leq c \cdot \big(1+n^{-1/2}(\pnorm{\hat{\mu}}{}\vee \pnorm{\mu^{(t)}}{}\vee \pnorm{\mu_0}{})\big)\cdot n^{-1/2}\pnorm{\mu^{(t)}-\hat{\mu}}{}\nonumber\\
	&\leq c \cdot \Big(1+n^{-1/2}(\pnorm{\hat{\mu}}{}\vee \pnorm{\mu_0}{})+ n^{-1/2}\pnorm{\mu^{(t)}-\hat{\mu}}{} \Big)\cdot n^{-1/2}\pnorm{\mu^{(t)}-\hat{\mu}}{}\nonumber\\
	&\leq c \cdot \big(L_{\mu_0,\xi} +\err_n^{(t)}\big)\cdot \big(L_{\mu_0,\xi} (1-\epsilon_0)^t+\err_n^{(t)}\big).
	\end{align}
	Here $\err_n^{(t)}(c_1)\equiv (c_1\log n)^{c_0 t^3}\cdot n^{-1/c_0^t}$, $\err_n^{(t)}\equiv \err_n^{(t)}(1)$.
	
	On the other hand, by Proposition \ref{prop:ridge_AMP}, $\mu^{(t)}=\prox_{\lambda\pnorm{\cdot}{}^2/2}(\mu_0-\mathfrak{D}_{\tau_{b^\ast}}^{1/2}v^{(t)};\tau_{b^\ast}) = (1_n+\lambda \tau_{b^\ast})^{-1}\circ \big(\mu_0-\mathfrak{D}_{\tau_{b^\ast}}^{1/2}v^{(t)}\big)$. Now applying Theorem \ref{thm:AMP_asym_avg} to $\{v^{(t)}\}$ and recalling (\ref{ineq:ridge_u_v_error_1}) with $(\zeta_v^\ast)^{1/2}=\mathfrak{D}_{\tau_{b^\ast}}^{-1/2}\gamma^\ast$ therein, for $t\leq \log n/c$, with probability at least $1-c_1 n^{-D}$ where $c_1=c_1(D)>0$,
	\begin{align}\label{ineq:ridge_dist_2}
	&\biggabs{\frac{1}{n}\sum_{j \in [n]} \psi\big(\mu_j^{(t)},\mu_{0,j}\big)-\frac{1}{n}\sum_{j \in [n]} \E\psi\Big(\prox_{\lambda\abs{\cdot}^2/2}\big(\mu_{0,j}+ \gamma_j^\ast Z;\tau_{b^\ast,j}\big),\mu_{0,j}\Big)}\nonumber\\
	&\leq c\cdot \big(\err_n^{(t)}(c_1) + L_{\mu_0,\xi}  \cdot (1-\epsilon_0)^t\big).
	\end{align}
	By noting $\prox_{\lambda\abs{\cdot}^2/2}\big(\mu_{0,j}+ \gamma_j^\ast Z;\tau_{b^\ast,j}\big)=\hat{\mu}^{\seq}_j(\gamma^\ast;\tau_{b^\ast})$, combining (\ref{ineq:ridge_dist_1}) and (\ref{ineq:ridge_dist_2}), for $t\leq \log n/c$, with probability at least $1-c_1 n^{-D}$ where $c_1=c_1(D)>0$,
	\begin{align*}
	&\biggabs{\frac{1}{n}\sum_{j \in [n]} \psi\big(\hat{\mu}_j,\mu_{0,j}\big)-\frac{1}{n}\sum_{j \in [n]} \E \psi\big(\hat{\mu}^{\seq}_j(\gamma^\ast;\tau_{b^\ast}),\mu_{0,j}\big)}\\
	&\leq c\cdot \big(L_{\mu_0,\xi} +\err_n^{(t)}\big)\cdot \big(L_{\mu_0,\xi} (1-\epsilon_0)^t+\err_n^{(t)}(c_1)\big). 
	\end{align*}
	The claim follows by choosing similarly $t=(\log \log n)^{0.99}$. 
\end{proof}

\begin{proof}[Proof of Theorem \ref{thm:ridge_dist} for $\hat{R}$]
	Similar to (\ref{ineq:ridge_dist_1}), we may use Proposition \ref{prop:ridge_theta_error} together with $\pnorm{(1_m - b^\ast)^{-1}}{\infty} \leq K^c$ to derive that there exists some $\epsilon_0=\epsilon_0(K) \in (0,1)$ such that with probability $1-e^{-n/c}$,
	\begin{align}\label{ineq:ridge_dist_res_1}
	&\biggabs{\frac{1}{m}\sum_{i \in [m]} \psi_0\big(\hat{R}_i\big)-\frac{1}{n}\sum_{i \in [m]} \psi_0\big(R_i^{(t)}\big)} \nonumber\\
	&\leq c \cdot \big(L_{\mu_0,\xi} +\err_n^{(t)}\big)\cdot \big(L_{\mu_0,\xi} (1-\epsilon_0)^t+\err_n^{(t)}\big).
	\end{align}
	On the other hand, by Proposition \ref{prop:ridge_AMP}, $R^{(t)}=-\mathfrak{D}_{1_m-b^\ast}^{1/2}u^{(t)}+\mathfrak{D}_{1_m-b^\ast}\xi$. Now applying Theorem \ref{thm:AMP_asym_avg} to $\{u^{(t)}\}$ and recalling (\ref{ineq:ridge_u_v_error_3}) with $\zeta_v^\ast=\mathfrak{D}_{\tau_{b^\ast}}^{-1}(\gamma^\ast)^2$ and $
	\zeta_u^\ast\equiv \big(\mathfrak{D}_{1_m-b^\ast}\mathscr{V}_m\mathfrak{D}_{\tau_{b^\ast}}\big) \mathfrak{D}_{1+\lambda \tau_{b^\ast}}^{-2}\big(\mathfrak{D}_{\lambda \tau_{b^\ast}}^2 \theta_{0;b^\ast}^2+\zeta_v^\ast\big)$ therein, for $t\leq \log n/c$, with probability at least $1-c_1 n^{-D}$ where $c_1=c_1(D)>0$,
	\begin{align}\label{ineq:ridge_dist_res_2}
	&\biggabs{\frac{1}{m}\sum_{i \in [m]} \psi_0\big(R_i^{(t)}\big)-\frac{1}{m}\sum_{i \in [m]} \E \psi_0\big( e_i^\top \mathfrak{D}^{1/2}_{1_m-b^\ast}(\zeta_u^\ast)^{1/2} \cdot Z + (1-b^\ast_i) \cdot\xi_i \big)} \nonumber\\
	&\leq c\cdot \big(\err_n^{(t)}(c_1) + L_{\mu_0,\xi}  \cdot (1-\epsilon_0)^t\big).
	\end{align}
	By noting $ e_i^\top \mathfrak{D}^{1/2}_{1_m-b^\ast}(\zeta_u^\ast)^{1/2} \cdot Z + (1-b^\ast_i) \cdot\xi_i \overset{d}{=} \hat{R}^\ast_i$, combining (\ref{ineq:ridge_dist_res_1}) and (\ref{ineq:ridge_dist_res_2}), for $t\leq \log n/c$, with probability at least $1-c_1 n^{-D}$ where $c_1=c_1(D)>0$,
	\begin{align*}
	&\biggabs{\frac{1}{m}\sum_{i \in [m]} \psi_0\big(\hat{R}_i\big)-\frac{1}{m}\sum_{i \in [m]} \E \psi_0\big(R^\ast_j\big)}\leq c\cdot \big(L_{\mu_0,\xi} +\err_n^{(t)}\big)\cdot \big(L_{\mu_0,\xi} (1-\epsilon_0)^t+\err_n^{(t)}(c_1)\big). 
	\end{align*}
	The claim follows by choosing, e.g., $t=(\log \log n)^{0.99}$. 
\end{proof}

\appendix

\section{Auxiliary results}

Let $\mathscr{M}_n(\R)$ be the space of $n\times n$ real symmetric matrices. The following result is frequently referred to `subgaussian tail of $\pnorm{A}{\op}$' in the proof.

\begin{lemma}\label{lem:A_op_subgaussian_tail}
	Let the upper triangular elements of $G_n \in \mathscr{M}_n(\R)$ be i.i.d. $\mathcal{N}(0,1/n)$, and $V \in \mathscr{M}_n(\R)$ be a deterministic matrix such that $0\leq V_{ij}\leq K$ for some $K>1$. Let $A=V\circ G_n$. Then there exists some universal constant $c_0>0$ such that the following hold:
	\begin{enumerate}
		\item For $x\geq 0$, $\Prob\big(\pnorm{A}{\op}/K\geq c_0+\sqrt{x/n}\big)\leq c_0 e^{-x/c_0}$.
		\item For $t\leq n/(c_0 \log n)$, $\E (\pnorm{A}{\op}/K)^t\leq c_0^t$. 
	\end{enumerate}
\end{lemma}
\begin{proof}
	(1) follows from, e.g., \cite[Theorem 4.4.5]{vershynin2018high}. For (2), let $E_0\equiv \{ \pnorm{A}{\op}/K\leq c_1\}$. Then $\Prob(E_0^c)\leq c_1 e^{-n/c_1}$, and therefore for the prescribed range of $t$,
	\begin{align*}
	\E (\pnorm{A}{\op}/K)^t \leq c_1^t+ \E (\pnorm{A}{\op}/K)^t\bm{1}_{E_0^c}\leq c_1^t+ (c_2 n^{1/2})^t e^{-n/c_2}\leq c_3^t,    
	\end{align*}
	as desired. 
\end{proof}

The following version of the Gaussian concentration inequality, allowing for `high probability bounded Lipschitz constant', is useful. 
\begin{lemma}\label{lem:gaussian_conc}
	Let $H:\mathscr{M}_n(\R)\to \R$ be a measurable map. Suppose there exist $t,L\geq 2$ such that the following hold:
	\begin{enumerate}
		\item For all $A \in \mathscr{M}_n(\R)$, $\abs{H(A)}\leq L (1+\pnorm{A}{\op})^t$. 
		\item For all $A_1,A_2 \in \mathscr{M}_n(\R)$,
		\begin{align*}
		\abs{H(A_1)-H(A_2)}\leq L\big(1+\pnorm{A_1}{\op}+\pnorm{A_2}{\op}\big)^t\cdot \pnorm{A_1-A_2}{\op}.
		\end{align*}
	\end{enumerate}
	Let the upper triangular elements of $G_n \in \mathscr{M}_n(\R)$ be i.i.d. $\mathcal{N}(0,1/n)$, and $V \in \mathscr{M}_n(\R)$ be a deterministic matrix such that $0\leq V_{ij}\leq K$ for some $K\geq 2$. Then there exists some universal constant $c_0>0$ such that if $t\leq n/(c_0\log n)$ and $x\geq c_0 L K^{c_0 t} e^{-n/c_0}$,
	\begin{align*}
	\Prob\big(\abs{H(V\circ G_n)-\E H(V\circ G_n)}\geq x\big)\leq c_0 e^{-n\cdot \min\{x^2/L^2K^{c_0t},1/c_0\}}.
	\end{align*}
	Moreover, there exists some universal $c_1>0$ such that for $t\leq n/(c_1 \log n)$,
	\begin{align*}
	\var\big(H(V\circ G_n)\big)\leq L^2 K^{c_1 t}\log n\cdot n^{-1}.
	\end{align*}
\end{lemma}
\begin{proof}
	The proof below is inspired by that of \cite[Lemma 3.5]{chen2021universality}. Fix $M>0$ to be chosen later. Let
	\begin{align*}
	\mathscr{H}(A)&\equiv \inf_{\substack{A' \in \mathscr{M}_n(\R), \pnorm{A'}{\op}\leq M}}\Big\{H(A')+L\cdot \Gamma(A,A')\big(\pnorm{A-A'}{\op}\wedge (2M)\big)\Big\},
	\end{align*}
	where 
	\begin{align*}
	\Gamma(A,A')\equiv\big( (\pnorm{A}{\op}\wedge M)+\pnorm{A'}{\op}+1\big)^t.
	\end{align*}
	Note that for any $A \in \mathscr{M}_n(\R)$ with $\pnorm{A}{\op}\leq M$, the definition of $\mathscr{H}$ entails that $\mathscr{H}(A)\leq H(A)$. On the other hand, the condition (2) implies that $H(A)\leq \mathscr{H}(A)$. In summary, 
	\begin{align}\label{ineq:gaussian_conc_1}
	\mathscr{H}(A)=H(A) \hbox{ for all } A \in \mathscr{M}_n(\R) \hbox{ such that }\pnorm{A}{\op}\leq M.
	\end{align}
	Next, for $A_1,A_2 \in \mathscr{M}_n(\R)$,
	\begin{align*}
	&\abs{\mathscr{H}(A_1)-\mathscr{H}(A_2)}/L\\
	&\leq   \sup_{\substack{A' \in \mathscr{M}_n(\R), \pnorm{A'}{\op}\leq M}}\bigabs{\Gamma(A_1,A') \big(\pnorm{A_1-A'}{\op}\wedge 2M\big) - \Gamma(A_2,A') \big(\pnorm{A_2-A'}{\op}\wedge 2M\big)  }\\
	&\leq 2M\cdot \sup_{A' \in \mathscr{M}_n(\R), \pnorm{A'}{\op}\leq M}\bigabs{\Gamma(A_1,A')-\Gamma(A_2,A')}+(2M+1)^t\cdot \pnorm{A_1-A_2}{\op}\\
	&\leq 2t(2M+1)^t\cdot \pnorm{A_1-A_2}{\op}\leq 2t(2M+1)^t\cdot \pnorm{A_1-A_2}{F}. 
	\end{align*}
	For the random matrix model $A=V\circ G_n$, we then have
	\begin{align*}
	\abs{\mathscr{H}(V\circ G_{n,1})-\mathscr{H}(V\circ G_{n,2})}&\leq KLM^{c_0t} \cdot \pnorm{G_{n,1}-G_{n,2}}{F}.
	\end{align*}
	Using Gaussian concentration inequality, for any $x>0, M>c' K$,
	\begin{align}\label{ineq:gaussian_conc_2}
	2\exp\bigg(-\frac{nx^2}{(KL)^2M^{ct}}\bigg)&\geq \Prob\big(\abs{\mathscr{H}(A)-\E \mathscr{H}(A)}\geq x\big)\nonumber\\
	&\geq \Prob\big(\abs{\mathscr{H}(A)-\E \mathscr{H}(A)}\geq x, \pnorm{A}{\op}\leq M\big)\nonumber\\
	&\geq \Prob\big(\abs{H(A)-\E \mathscr{H}(A)}\geq x\big)-c e^{-n M^2/cK^2}. 
	\end{align}
	Here in the last inequality we used (\ref{ineq:gaussian_conc_1}) and the subgaussian tail estimate for $\pnorm{A}{\op}$, cf. Lemma \ref{lem:A_op_subgaussian_tail}. On the other hand, as $\E \mathscr{H}(A)=\E H(A)\bm{1}_{\pnorm{A}{\op}\leq M}+\E \mathscr{H}(A)\bm{1}_{\pnorm{A}{\op}> M}$, using the condition (1), we have for $t\leq n/(c\log n)$,
	\begin{align*}
	&\bigabs{\E \mathscr{H}(A)-\E H(A)}\leq \E\big(\abs{\mathscr{H}(A)}+\abs{H(A)}\big)\bm{1}_{\pnorm{A}{\op}\geq M}\\
	&\leq \E^{1/2}\big(\abs{H(0)}+L(1+\pnorm{A}{\op})^{t+1}+\abs{H(A)}\big)^2\cdot \Prob^{1/2}(\pnorm{A}{\op}\geq M)\\
	&\leq c L \E(1+\pnorm{A}{\op})^{ct}\cdot e^{-nM^2/cK^2}\leq L K^{c_1 t} e^{-n M^2/c_1K^2}.
	\end{align*}
	So for any $x>0, M>c_1' K$ such that $x\geq L K^{c_1 t} e^{-n M^2/c_1 K^2}$, and $t\leq n/(c_2\log n)$,
	\begin{align*}
	\Prob\big(\abs{H(A)-\E H(A)}\geq 2x\big)\leq c_2 \exp\bigg(-n\cdot \bigg\{\frac{x^2}{c_2 (KL)^2 M^{c_2 t}}\wedge \frac{M^2}{c_2 K^2}\bigg\} \bigg),
	\end{align*}
	proving the subgaussian estimate by choosing $M=c_2'K$  for a large universal $c_2'>1$, and adjusting the constant. 
	
	For the variance bound, note that for a large enough, universal $c_3>0$, the above inequality yields that the event $E_0\equiv \{\abs{H(A)-\E H(A)}\leq L K^{c_3 t} \sqrt{\log n}\cdot n^{-1/2}\}$ occurs with $\Prob(E_0)\geq 1-c_3 n^{-100}$. Consequently, for $t\leq n/(c_3\log n)$,
	\begin{align*}
	&\var(H(A))\leq L^2 K^{ct}\log n\cdot  n^{-1}+\E \abs{H(A)-\E H(A)}^2\bm{1}_{E_0^c}\\
	&\leq L^2 K^{c t}\log n\cdot  n^{-1}+ L^2\E^{1/2}(1+\pnorm{A}{\op})^{ct}\cdot \Prob^{1/2}(E_0^c)\leq L^2 K^{c_4 t}\log n\cdot n^{-1},
	\end{align*}
	as desired. 
\end{proof}

We also need the following elementary lemma on the matrix square root.

\begin{lemma}\label{lem:cov_sqrt_diff}
	Let $\Sigma_1,\Sigma_2 \in \R^{t\times t}$ be two positive semi-definite matrices. Then it holds that $\pnorm{\Sigma_1^{1/2}-\Sigma_2^{1/2}}{F}^2\leq \sqrt{t} \cdot \pnorm{\Sigma_1-\Sigma_2}{F}$.
\end{lemma}
\begin{proof}
	The result is well-known, see e.g., \cite[Eqn. (3.2)]{wihler2009holder} for a more general formulation. For the convenience of the reader, we provide a direct and short proof below. Let $\Sigma_\ell=U_\ell \Lambda_\ell U_\ell^\top$ ($\ell=1,2$) be the spectral decomposition of $\Sigma_\ell$. By writing $\Lambda_\ell=\mathrm{diag}(\{\lambda_{\ell,i}\}_{i \in [t]})$ where $\lambda_{\ell,i}\geq 0$, we have
	\begin{align*}
	\bigpnorm{\Sigma_1^{1/2}-\Sigma_2^{1/2}}{F}^2&= \bigpnorm{U_1 \Lambda_1^{1/2}U_1^\top -U_2\Lambda_2^{1/2} U_2^\top}{F}^2 = \bigpnorm{ \Lambda_1^{1/2}U_1^\top U_2 -U_1^\top U_2\Lambda_2^{1/2}}{F}^2\\
	& = \sum_{i,j \in [t]} \big(\lambda_{1,i}^{1/2}-\lambda_{2,j}^{1/2}\big)^2 (U_1^\top U_2)_{ij}^2 \stackrel{(\ast)}{\leq} \sum_{i,j \in [t]} \abs{\lambda_{1,i}-\lambda_{2,j}} (U_1^\top U_2)_{ij}^2\\
	&\leq \bigg(\sum_{i,j \in [t]} \big(\lambda_{1,i}-\lambda_{2,j}\big)^2 (U_1^\top U_2)_{ij}^2\bigg)^{1/2}\cdot \pnorm{U_1^\top U_2}{F} = \sqrt{t}\cdot \pnorm{\Sigma_1-\Sigma_2}{F}.
	\end{align*}
	Here in $(\ast)$ we used the fact that for any $a,b\geq 0$, $(\sqrt{a}-\sqrt{b})^2\leq \abs{a-b}$.
\end{proof}

The following well-known Young's inequality is used in our analysis.
\begin{lemma}\label{lem:young_ineq}
	If $a, b \geq 0$ and $p, q > 1$ such that $p^{-1} + q^{-1} = 1$, then
	\begin{align*}
		ab \leq \frac{a^{p}}{p} + \frac{b^{q}}{q},
	\end{align*}
	with equality holds if and only if $a^{p} = b^{q}$.
\end{lemma}

\section*{Acknowledgments}
The authors would like to thank two referees and an Associate Editor for their helpful comments and suggestions that significantly improved the quality of the paper.

The research of Z. Bao is partially supported by Hong Kong RGC Grant 16303922, NSFC12222121 and NSFC12271475. The research of Q. Han is partially supported by NSF Grant DMS-2143468. The research of X. Xu is partially supported by Hong Kong RGC Grant 16305421.

\bibliographystyle{alpha}
\bibliography{mybib}

\end{document}